\documentclass[10pt, a4paper]{article}

\usepackage{amsmath,amsfonts,amssymb}
\usepackage{amsthm} 
\usepackage{hyperref}
\usepackage[nameinlink,capitalize]{cleveref}
\hypersetup{colorlinks,urlcolor=blue,linkcolor = blue,citecolor = blue}
\crefname{equation}{}{} 

\usepackage{enumerate}
\usepackage{float}
\usepackage[normalem]{ulem}
\usepackage{xcolor}

\usepackage{geometry}
\usepackage{bbm}
\usepackage{mathtools} 

\usepackage{enumitem}
\setlist[enumerate]{label*=\alph*),ref=\alph*)}

\newcommand{\Asup}{\mathcal{A}_{\sup}}
\newcommand{\Anosup}{\mathcal{A}_{no\sup}}

\newcommand{\B}{\mathcal{B}}
\newcommand{\E}{\mathbb{E}}
\newcommand{\e}{\mathrm{e}}
\newcommand{\F}{\mathcal{F}}
\newcommand{\G}{\mathcal{G}}
\newcommand{\PP}{\mathbb{P}}
\newcommand{\R}{\mathbb{R}}
\newcommand{\N}{\mathbb{N}}
\newcommand{\one}{\mathbbm{1}}
\newcommand{\dd}{\mathrm{d}}
\newcommand{\define}{\vcentcolon =}
\newcommand{\EO}{\mathbb{E}_{\F_0}}
\newcommand{\PO}{\mathbb{P}_{\F_0}}
\newcommand{\pF}{_{p,\F_0}}
\newcommand{\qF}{_{q,\F_0}}

\newcommand{\nn}{^{(n)}}

\newcommand{\cadlag}{\text{C}([-r, \infty);\R^d)}
\newcommand{\cadlagg}{c\`adl\`ag }

\newcommand{\pred}{\mathcal{P}}
\newcommand{\Xt}{X_t}

\newcommand{\Xsm}{X_{s^-}}
\newcommand{\s}{_s}

\newcommand{\uu}{_u}
\newcommand{\sm}{_{s^-}}
\newcommand{\ti}{_t}

\newcommand{\Xvdkln}{(X^*_{\tau_{k-1}})^{p}}
\newcommand{\Xvdkn}{(X^*_{\tau_{k}})^{p}}
\newcommand{\range}{\rm{range}}
\newcommand{\id}{\mathrm{id}}

\DeclareMathOperator*{\esssup}{ess\,sup}

\newtheorem{theorem}{Theorem}[section]
\newtheorem{corollary}[theorem]{Corollary}
\newtheorem{lemma}[theorem]{Lemma}

\newtheorem{defi}[theorem]{Definition}

\theoremstyle{definition}
\newtheorem{example}[theorem]{Example}
\newtheorem{remark}[theorem]{Remark}

\newtheorem{cntexample}[theorem]{Counterexample}

\title{Sharp convex generalizations of stochastic Gronwall inequalities}
\author{%
Sarah~Geiss\footnote{Technische Universit\"at Berlin, Germany. E-mail: \href{mailto:geiss@math.tu-berlin.de}{geiss@math.tu-berlin.de}}
\thanks{The author was supported by an 
Elsa-Neumann-Stipendium des Landes Berlin.} }

\date{April 20, 2023}   

\begin{document}

\maketitle

\begin{abstract} 
We provide generalizations of a class of stochastic Gronwall inequalities that has been studied by von Renesse and Scheutzow (2010), Scheutzow (2013), Xie and Zhang (2020) and Mehri and Scheutzow (2021). This class of stochastic Gronwall inequalities is a useful tool for SDEs.

Our focus are convex generalizations of the Bihari-LaSalle type. The constants we obtain are sharp. In particular, we provide new sharp constants for the stochastic Gronwall inequalities. The proofs are connected to a domination inequality by Lenglart (1977), an inequality by Pratelli (1976) and  a characterization of Lenglart's concept of domination via the Snell envelope.

The inequalities we study appear for example in connection with exponential moments of solutions to path-dependent SDEs: For non-path-dependent SDEs, criteria for the finiteness of exponential moments are known. To be able to extend these proofs to the path-dependent case, a convex generalization of a stochastic Gronwall inequality seems necessary. Using the results of this paper, we obtain a criterion for the finiteness of exponential moments which is similar to that known for non-path-dependent SDEs. 

Stochastic Gronwall inequalities can also be applied to study other types of SDEs than path-dependent SDEs: An estimate of this paper is applied by Agresti and Veraar (2023) to prove global well-posedness for reaction-diffusion systems with transport noise.
\end{abstract}

[0.25em]
\textbf{MSC2020 subject classifications:} 
34K50, 
60H10, 
60G44, 
60G51, 
60J65 

\bigskip

\tableofcontents

\section{Introduction}\label{sec:intro}

In this article we provide sharp generalizations of two types of stochastic Gronwall inequalities. In particular, we establish new sharp constants for the stochastic Gronwall inequalities we generalize. The inequalities we study appear for example in connection with path-dependent SDEs: For non-path-dependent SDEs, criteria for the finiteness of exponential moments are known. To be able to extend these results to the path-dependent case, a convex generalization of a stochastic Gronwall inequality seems necessary. Using the results of this paper, we obtain a criterion for the finiteness of exponential moments that complements the criterion known for non-path-dependent SDEs. Stochastic Gronwall inequalities can also be applied to study other types of SDEs than path-dependent SDEs: An estimate of this paper is applied by Agresti and Veraar \cite{AgrestiVeraar2} to prove global well-posedness for reaction-diffusion systems with transport noise.

\bigskip

\textbf{Results in the literature on stochastic Gronwall inequalities:} 
The following  \emph{stochastic Gronwall inequality with supremum} is due to von Renesse and Scheutzow \cite[Lemma 5.4]{RenesseScheutzow} and was generalized by Mehri and Scheutzow \cite[Theorem 2.1]{MehriScheutzow}: Let $(\Xt)_{t\geq 0}$ be a non-negative stochastic process that satisfies
\begin{equation}\label{eq:introsuplinear}
\Xt \leq  \int_{(0,t]} \Xsm^* \dd A\s + M\ti + H\ti \qquad \text{for all } t\geq 0,
\end{equation}
where $X^*\s= \sup_{u\in[0,s]}X\uu$  denotes the running supremum. Here, $M$ is a c\`adl\`ag local martingale that starts in $0$, and $H$ and $A$ are suitable non-decreasing stochastic processes. Then, for all $T>0$ and $p\in(0,1)$ there exists an explicit upper bound for $\E[\sup_{t\in[0,T]} X\ti^p]$ which does not depend on the local martingale $M$.

There is also a \emph{stochastic Gronwall inequality without supremum} which is closely connected to the previous inequality with supremum. This result is due to Scheutzow \cite[Theorem 4]{Scheutzow} and was generalized by Xie and Zhang \cite[Lemma 3.7]{XieZhang} from continuous local martingales to c\`adl\`ag local martingales: If we assume instead of \eqref{eq:introsuplinear} the slightly stronger assumption that 
\begin{equation}\label{eq:intronosuplinear}
X\ti \leq  \int_{(0,t]} X\sm \dd A\s + M\ti + H\ti \qquad \text{for all } t\geq 0,
\end{equation}
sharper bounds can be obtained.

Both previously mentioned inequalities are useful tools for SDEs. The stochastic Gronwall inequality with supremum 
is applied to study SDEs with memory, see for  example  \cite{AppSup1}, \cite{AppSup4},  \cite{AppSup3}, \cite{AppSup2},  \cite{HutzenthalerNguyen}, \cite{MehriScheutzow}, \cite{AppSup5} and \cite{RenesseScheutzow}. 
The stochastic Gronwall inequality without supremum is applied to study various SDEs without memory, see e.g.
 \cite{AppNoSup4}, \cite{AppNoSup3}, \cite{AppNoSup8}, \cite{AppNoSup10}, \cite{AppNoSup11},
 \cite{AppNoSup1}, \cite{AppNoSup7}, \cite{AppNoSup5}, \cite{AppNoSup2}, \cite{AppNoSup6} and
 \cite{AppNoSup9}.
 
Also other stochastic Gronwall inequalities have been studied, see e.g. Glatt-Holtz and Ziane \cite[Lemma 5.3]{GlattZiane} and Agresti and Veraar \cite[Lemma A.1]{AgrestiVeraar}.

\bigskip
\textbf{Results in the literature on generalizations of stochastic Gronwall inequalities:}
Also nonlinear extensions of the stochastic Gronwall inequalities have been studied. Makasu \cite[Theorem 2.2]{Makasu1} and Le and Ling \cite[Lemma 3.8]{LeLing} studied a generalization where in the assumption the term $\int_{(0,t]} X\sm \dd A\s$  is replaced by $\big(\int_{(0,t]} (X_s)^\theta \dd A\s\big)^{1/\theta}$ for $\theta >0$. Under similar assumptions as above, estimates for $\E[\sup_{t\in[0,T]}X^p_t]$  were obtained.
 
 A further extension of the stochastic Gronwall inequalities has been studied by Mekki, Nieto and Ouahab \cite[Theorem 2.4]{Nieto}: For continuous local martingales a so-called stochastic Henry Gronwall's inequality 
with upper bounds that do not depend on the local martingale $M$ can be proven.

In addition, also extensions of the stochastic Gronwall inequalities have been studied, where the upper bounds depend on the quadratic variation of the martingale $M$, see e.g. Makasu \cite{Makasu1}, Makasu \cite{Makasu2} and Mekki, Nieto and Ouahab \cite{Nieto}. In the present paper, we focus on bounds which do not depend on the local martingale $M$.  Furthermore,  Hudde, Hutzenthaler and Mazzonetto \cite{Hudde} have extended the stochastic Gronwall inequality without supremum to the setting of It\^o processes which satisfy a suitable one-sided affine-linear growth condition.

\bigskip \bigskip

In this paper, we study the following generalization of the above mentioned stochastic Gronwall inequalities: We replace the assumptions \eqref{eq:introsuplinear} and \eqref{eq:intronosuplinear} by
\begin{equation}\label{eq:introsup}
X\ti \leq  \int_{(0,t]} \eta(X^*\sm) d A\s + M\ti + H\ti \qquad \text{for all } t\geq 0,
\end{equation}
and
\begin{equation}\label{eq:intronosup}
X\ti \leq  \int_{(0,t]} \eta(X\sm) d A\s + M\ti + H\ti \qquad \text{for all } t\geq 0,
\end{equation}
respectively, where $\eta:[0,\infty) \to [0,\infty)$ is a \textit{convex} non-decreasing function. 

For continuous local martingales and $\eta$ that satisfy 
$\int_{1}^\infty \frac{\dd u}{\eta(u)} =+\infty$, \eqref{eq:introsup} is studied by von Renesse and Scheutzow \cite[Lemma 5.1]{RenesseScheutzow} in the context of global solutions of stochastic functional differential equations. For \cadlagg martingales and concave $\eta$, and $\eta$ that satisfy either $\int_{0+} \frac{\dd u}{\eta(u)} =+\infty$ or $\int_{1}^\infty \frac{\dd u}{\eta(u)} =+\infty$, these inequalities are studied in \cite{GeissConcave}.

The Bihari-LaSalle inequality provides an upper bound for $X\ti$ in the deterministic case (i.e. $M \equiv 0$), see \cite{Bihari}, \cite{LaSalle}. More general versions than \cref{lemma:DetBihari} are known in the literature, see also e.g. \cite[Lemma 1.1.]{GeissConcave} for a version which allows $x$ and $A$ to be \cadlagg.

\begin{lemma}[Deterministic Bihari-LaSalle inequality]\label{lemma:DetBihari}
Let $x,\varphi \colon [0,\infty)\to[0,\infty)$ be positive continuous functions and $T,H\geq 0$ constants. Further, let $\eta\colon[0,\infty)\to[0,\infty)$ be a non-decreasing continuous function and set $A(t) \define \int_0^t \varphi(u) \dd u$ for all $t\in[0,T]$. Then, the inequality
\begin{equation}\label{eq:detBihari}
x(t) \leq \int_{0}^t \eta(x(s)) \dd A(s) + H \qquad \forall t\in[0,T]
\end{equation}
implies the inequality 
\begin{equation*}
x(t) \leq G^{-1}(G(H) + A_t)) \qquad \forall t\in[0,T'].
\end{equation*}
where $G(v) \define \int_c^v \frac{\dd u}{\eta(u)}$ for some constant $c>0$. Here, $T'$ is chosen small enough such that $G(H) + A_t \in \textrm{domain}(G^{-1})$ is ensured for all $t\in[0,T']$.
\end{lemma}
Note that the upper bound on $x$ does not depend on the choice of the constant $c>0$ used in the definition of $G$. We obtain the well-known Gronwall inequality by choosing $c=1$ and $\eta(u) \equiv u$ in \cref{lemma:DetBihari} so that $G(u) \equiv \log(u)$, which implies the upper bound
\begin{equation*}
x(t)\leq H \exp(A(t)).
\end{equation*}

In \cite[Theorem 3.1]{GeissConcave} it is shown for \eqref{eq:introsup} and \eqref{eq:intronosup}, \emph{concave} $\eta$ and deterministic $A$ that
inequalities of the form 
\begin{equation}\label{eq:intro-bound}
\E[(X^*_T)^p]^{1/p} \leq \bar{c}_p G^{-1}(G(\tilde{c}_p
\E[(H_T)^p]^{1/p})+\hat{c}_p A_T) \qquad \forall T>0
\end{equation}
hold true, where $G$ is defined as in \cref{lemma:DetBihari} and $\bar{c}_p$, $\tilde{c}_p$ and $\hat{c}_p$ are constants which only depend on $p\in(0,1)$.

\bigskip

\textbf{Results of this article:} For the stochastic Bihari-LaSalle assumptions \eqref{eq:introsup} and \eqref{eq:intronosup} with \emph{convex} $\eta$, estimates of the form \eqref{eq:intro-bound} do not hold true in general: Let for example $Y$ be a Cox-Ingersoll-Ross process and set $X=\exp(\lambda Y)$ for some $\lambda >0$. For suitably chosen parameters, $X$ satisfies \eqref{eq:intronosup} for some convex $\eta$ and $\|X^*_t\|_p$ explodes at a finite time (see e.g. \cite[Proposition 3.1]{CIR1} or \cite[Proposition 3.2]{CIR2}). Alternatively, see \cref{counterexample} for a simple counterexample.

Instead, the following types of estimates can be shown for example for stochastic Bihari-LaSalle assumption without supremum \eqref{eq:intronosup}, see \cref{thm:stochBihariConvex} for the complete theorem:
\begin{theorem}[Convex stochastic Bihari-LaSalle inequality without supremum]
Let 
\begin{itemize}
\item  $(X_t)_{t\geq 0}$ be an adapted \cadlagg process with $X\geq 0$,
\item  $(A_t)_{t\geq 0}$ be a predictable non-decreasing c\`adl\`ag process with $A_0=0$, 
\item  $(H_t)_{t\geq 0}$ be a predictable non-negative non-decreasing c\`adl\`ag process,
\item  $(M_t)_{t\geq 0}$ be a c\`adl\`ag local martingale with $M_0=0$,
\item  $\eta:[0,\infty) \to [0,\infty)$ be a convex non-decreasing function with $\eta(x)>0$ for all $x>0$.
\end{itemize}
Assume for all $t\geq 0$
\begin{equation*}
X\ti \leq  \int_{(0,t]} \eta(X\sm) d A\s + M\ti + H\ti.
\end{equation*}
Then, for all $T\geq 0$, $p\in(0,1)$ and $u>c_0$ and $w, R>0$ the expressions $G^{-1}(G(X^*_T) - \hat{c}_p A_T)$ and $G^{-1}(G(u) - R)$ are well-defined and the following estimates hold:
\begin{align}
\big\| G^{-1}(G(X^*_T) - \hat{c}_p A_T)\big \|_p & \leq  \tilde{c}_p \| H_T\|_p,\label{eq:intro-type1} \\[0.5em]
\PP\bigg[\sup_{t\in[0,T]} X_t > u \,\bigg]  & \leq \frac{\E[H_T \wedge w ]}{G^{-1}(G(u) - R)} + \PP[H_T\geq w] + \PP[A_T > R], \label{eq:intro-type2} 
\end{align}
where $\hat{c}_p \define  1$ and $\tilde{c}_p \define (1-p)^{-1/p} p^{-1}$.
\end{theorem}

In particular, for the stochastic Gronwall inequality without supremum (i.e. $\eta(x) = x$), we obtain the new estimate
\begin{equation}\label{eq:intro-type3} 
\PP\bigg[\sup_{t\in[0,T]} X_t > u \,\bigg]   \leq \frac{\e^R}{u}\E[H_T \wedge w ] + \PP[H_T\geq w] + \PP[A_T > R].
\end{equation}
Estimate \cref{eq:intro-type3} complements \cite[Theorem 4]{Scheutzow} and  \cite[Lemma 3.7]{XieZhang}, where under similar assumptions upper bounds for $\E[\sup_{t\in[0,T]} X^p_t]$ for $p\in(0,1)$ are proven. The formulation \eqref{eq:intro-type3} is useful when $A$ or $H$ do not have suitable integrability properties to satisfy the assumptions of \cite[Theorem 4]{Scheutzow} and \cite[Lemma 3.7]{XieZhang}.

A main result of this paper is that for the stochastic Bihari-LaSalle inequality with  supremum \eqref{eq:introsup} we also obtain a bound of the form \eqref{eq:intro-type1} (see \cref{thm:stochBihariConvex}). Another main result of this paper is the sharpness of the estimates and constants. In particular, we show that for \eqref{eq:introsup} the estimate \eqref{eq:intro-type2} does not hold true in general.
As an application of the sharpness results, we obtain
that the tail behaviour of path-dependent and non-path-dependent SDEs differ, see \cref{cor:tail-of-path-SDE}.

An overview of the stochastic Bihari-LaSalle inequalities obtained can be found in \cref{table:BihariLaSalle}, see \cref{subsec:results}.

\bigskip

We prove the stochastic Bihari-LaSalle inequalities (\cref{thm:stochBihariConvex}) by extending an inequality by Lenglart \cite[Th\'eor\`eme I]{Lenglart} and using a characterization of Lenglart's concept of domination by the Snell envelope.  The main lemmas of this paper (\cref{lemma:1}, \cref{lemma:smoothenJumps}) can also be applied to prove concave and other generalizations of the stochastic Gronwall inequalities, see \cite{GeissConcave} or \cref{table:BihariLaSalle}.

Moreover, in addition, our proofs are inspired by the proof of an inequality by Pratelli \cite[Proposition 1.2]{Pratelli} and by the proofs of the stochastic Gronwall inequalities by Mehri and Scheutzow \cite[Theorem 2.1]{MehriScheutzow} and Xie and Zhang \cite[Lemma 3.7]{XieZhang}.

\bigskip \bigskip

\textbf{Example where a stochastic Gronwall inequality instead of the deterministic Gronwall inequality is needed:} Deterministic Gronwall inequalities are widely used to study SDEs. We provide a simple toy example when \emph{stochastic} generalizations are useful to shorten calculations, for a proper application example see e.g. \cite[Lemma 5.3, Lemma 6.7]{AgrestiVeraar2}, in which energy estimates for reaction-diffusion systems with transport noise are proven, which in turn are applied to obtain global well-posedness. For a toy example assume $(\tilde{X}_t)_{t\geq 0}$ to be a global solution of the  following $d$-dimensional (not path-dependent) SDE driven by an $m$-dimensional Brownian motion $B$
\begin{equation*}
\dd \tilde{X}_t = \tilde{f}(t,\tilde{X}_t) \dd t  + \tilde{g}(t,\tilde{X}_t) \dd B_t \qquad \forall t\geq 0,
\end{equation*}
and assume that the random coefficients $\tilde{f}$ and $\tilde{g}$ satisfy the one-sided coercivity assumption 
$2\langle \tilde{f}(t,x), x \rangle  + \|\tilde{g}(t, x)\|_F^2 \leq K_t |x|^2$ for all $x\in \R^d$, $t\geq 0$ where $(K_t)_{t\geq 0}$ is an adapted non-negative stochastic process. Then, by Ito's formula we obtain
\begin{equation*}
|\tilde{X}_t|^2 \leq |\tilde{X}_0|^2 +  \int_0^t |\tilde{X}_s|^2 K_s \dd s + 2\int_0^t \langle \tilde{X}_s , \tilde{g}(t,\tilde{X}_t) \dd B_s \rangle  \qquad \forall t\geq 0.
\end{equation*} 
If $(K_t)_{t\geq 0}$ is random, it is \emph{not} possible to take expectations and then directly apply the Gronwall inequality to $t\mapsto \E[|\tilde{X}_t|^2]$. In this case  stochastic Gronwall inequalities are useful, as they can be applied directly. Moreover, without further work they immediately give estimates for the running supremum of $(\tilde{X}_t)_{t\geq 0}$. For example by \eqref{eq:intro-type3} (see \cref{cor:GronwallNoSup}) we have for all $u>0$, $w>0$ and $R>0$
\begin{equation*}
\PP\left[\sup_{s\in[0,t]}|\tilde{X}_s|^2 > u \right] \leq \frac{\e^R}{u} \E[|X_0^2|\wedge w]  + \PP[ |X_0^2|\geq w] + \PP\left[\int_0^t K_s > R\right]. 
\end{equation*}

\textbf{Application example of the convex Bihari-LaSalle inequality:} The convex stochastic Bihari-LaSalle inequality with supremum is a useful tool to derive exponential moment estimates for path-dependent SDEs. For non-path-dependent SDEs, these types of estimates are known, see e.g.  Cox, Hutzenthaler and Jentzen \cite{CoxHutzenthalerJentzen}, Hudde, Hutzenthaler and Mazzonetto \cite{Hudde} and the references therein. We sketch a simple case of \cite[Corollary 2.4]{CoxHutzenthalerJentzen} or \cite[Corollary 3.3]{Hudde}. Afterwards, we shortly explain why our convex stochastic Bihari-LaSalle inequality succeeds in extending these types of results to path-dependent SDEs.

Consider the  following $d$-dimensional (not path-dependent) SDE driven by an $m$-dimensional Brownian motion $B$
\begin{equation*}
\dd \tilde{X}_t = \tilde{f}(t,\tilde{X}_t) \dd t  + \tilde{g}(t,\tilde{X}_t) \dd B_t, \qquad X_0 = x_0 \in \R^d,
\end{equation*}
and assume suitable measurability and integrability assumptions on the coefficients $\tilde{f}$ and $\tilde{g}$. Fix some $U\in C^2(\R^d, \R)$ and assume that the coefficients $\tilde{f}$ and $\tilde{g}$ satisfy
\begin{equation}\label{eq:conditionSDEcoeff}
\tilde{\G}_{\tilde{f},\tilde{g}} U(x)  + \frac{1}{2\e^{\alpha t}} |(\nabla U(x))^T \tilde{g}(t,x)|^2 \leq \alpha U(x) \qquad \forall t\geq 0, \quad \forall x\in\R,
\end{equation}
for some $\alpha>0$. Here $\tilde{\G}_{\tilde{f},\tilde{g}}$ denotes the infinitesimal generator of the SDE. Then, \cite[Corollary 3.3]{Hudde} (see also \cite[Corollary 2.4]{CoxHutzenthalerJentzen}) implies for any (say global) solution $\tilde{X}$ that for all $p\in(0,1)$
\begin{equation*}
\E[\sup_{t\in[0,T]} \exp(pU(\tilde{X}_t)\e^{-\alpha t})] \leq c(p) \E[\exp(p U(x_0))],
\end{equation*}
where $c(p)$ is a constant that only depends on $p$. This can be proven by applying It\^o's formula to compute $\tilde{Z}_t \define \exp(U(\tilde{X}_t)\e^{-\alpha t})$ and then applying the assumption \eqref{eq:conditionSDEcoeff}:
\begin{equation}
\begin{aligned}
\dd \tilde{Z}_t & = \tilde{Z}_t \e^{-\alpha t} \big[ \tilde{\G}_{\tilde{f},\tilde{g}}U(t,x) - \alpha U(X_t) + \tfrac{1}{2}\e^{-\alpha t} |\nabla U(x)^T \tilde{g}(t,x)|^2 \big] \dd t +  \tilde{Z}_t \e^{-\alpha t} \nabla U(X_t)^T \tilde{g}(t,X_t) \dd B_t \\
& \leq  \tilde{Z}_t \e^{-\alpha t} \nabla U(X_t)^T \tilde{g}(t,X_t) \dd B_t.
\end{aligned}
\end{equation}
By Fatou's lemma, we have for all bounded stopping times $\tau$ the inequality $\E[ \exp(U(\tilde{X}_\tau)\e^{-\alpha \tau})] \leq \E[\exp(U(x_0))]$. This implies e.g. by Lenglart's domination inequality (see e.g. \cref{lemma:Lenglart}) the claim.

This argument does not work for path-dependent SDEs: Let $r>0$ be some constant and $x_0\in C([-r;0],\R^d)$ the initial value. Consider
\begin{equation*}
\begin{aligned}
\dd X_t & = f(t,X_{-r:t}) \dd t  + g(t,X_{-r:t}) \dd B_t, \qquad \forall t > 0 \\
X_t & = x_0(t) \quad \forall t\in[-r,0],
\end{aligned}
\end{equation*}
where we use the notation $X_{-r:t}$ for the path segment $\{X(s),s\in[-r,t]\}$. (For details on the assumptions on the coefficients see \cref{sec:applications}.) Correspondingly, we define for all $U\in C^2(\R^d, \R)$, $x\in\cadlag$ and all $t\geq 0$
\begin{equation*}
(\mathcal{G}_{f,g} U)(t,x_{-r,t}) \define  
 \big(\nabla U\big)^T(x(t)) f(t,x_{-r,t})\\
  +  \tfrac{1}{2}\text{trace}(g(t,x_{-r,t}) g(t,x_{-r,t})^T (\text{Hess}\,U)(x(t))).
\end{equation*}
It seems reasonable to weaken the condition \eqref{eq:conditionSDEcoeff} to
\begin{equation}\label{eq:conditionSDEcoeffpath}
(\G_{f,g} U)(t,x_{-r,t})  + \frac{1}{2\e^{\alpha t}} |(\nabla U(x(t)))^T g(t,x_{-r,t})|^2 \leq \alpha \sup_{s\in[-r,t]} U(x(s)) \qquad \forall x\in \cadlag, \forall t\geq 0,
\end{equation}
as the terms on the left-hand side depend on $x_{-r:t}$, not only $x(t)$.
However, when computing $Z_t \define \exp(U(t,X_t)\e^{-\alpha t})$, the terms  in the integral w.r.t $\dd t$ fail to cancel out after applying \eqref{eq:conditionSDEcoeffpath} due to the supremum in our condition on the coefficients:
\begin{equation*}
\begin{aligned}
\dd Z_t & = Z_t \e^{-\alpha t} \big[ (\G_{f,g}U)(t,X_{-r:t}) - \alpha U(X_t) + \tfrac{1}{2}\e^{-\alpha t} |\nabla U(X_t)^T g(t,X_{-r:t})|^2 \big] \dd t + \dd M_t \\
& \leq \alpha Z_t \e^{-\alpha t} [\sup_{s\in[-r,t]}U(X_s)-U(X_t)]\dd t + \dd M_t
\end{aligned}
\end{equation*}
where $M_t \define Z_t \e^{-\alpha t} \nabla U(X_t)^T g(t,X_{-r,t}) \dd B_t.$ Assuming in addition that $U\geq 0$ and noting that $U(X_s) = \log(Z_t)e^{\alpha t}$, we obtain a convex stochastic Bihari-LaSalle inequality for $Z_t$ and $\eta(x) \define \alpha x(\log(x)+ \sup_{s\in[-r,0]}U(x_0(s)))$
\begin{equation*}
\dd Z_t \leq  \eta(Z_t^*)\dd t +  M_t, \qquad \forall t\in[0,T].
\end{equation*}
Applying our main theorem (a stochastic Bihari-LaSalle inequality) gives us a similar estimate as in the non-path-dependent case. We generalize and improve this approach in \cref{sec:applications}.


\section{Notation, assumptions and overview}
We assume that all processes are defined on an underlying filtered probability space $(\Omega, \F, \PP, (\F_t)_{t\geq 0})$ satisfying the usual conditions, i.e. which is complete and right-continuous.

\subsection{Assumptions}\label{subsec:assumptions}
We study the following two cases:
\begin{defi}[Assumption $\mathcal{A}_{\sup}$] \label{def:sup}
Let 
\begin{itemize}
\item  $(X_t)_{t\geq 0}$ be an adapted right-continuous process  with $X\geq c_0$ for some $c_0 \geq 0$,
\item  $(A_t)_{t\geq 0}$ be a predictable non-decreasing c\`adl\`ag process with $A_0=0$, 
\item  $(H_t)_{t\geq 0}$ be an adapted non-negative non-decreasing c\`adl\`ag process,
\item  $(M_t)_{t\geq 0}$ be a c\`adl\`ag local martingale with $M_0=0$,
\item  $\eta:[c_0,\infty) \to [0,\infty)$ be a continuous non-decreasing function with $\eta(x)>0$ for all $x>c_0$.
\end{itemize}
We say the processes ($X$, $A$, $H$, $M$) satisfy $\mathcal{A}_{\sup}$ if they satisfy the inequality below for all $t\geq 0$:
\begin{equation}\label{eq:GronwallAssumption}
X\ti \leq \int_{(0,t]} \eta(X^*\sm) \dd A\s + M\ti + H\ti \qquad \PP\text{-a.s},
\end{equation}
where $X^*_{s^-} = \sup_{u< s}X\s$.
\end{defi}

\noindent The following assumption is slightly stronger:
\begin{defi}[Assumption $\mathcal{A}_{no\sup}$]\label{def:nosup}
 Under the same assumptions on the processes as in the previous definition, we say that the processes satisfy $\mathcal{A}_{no\sup}$ if in addition $X$ has left limits and the processes satisfy the following inequality for all $t\geq 0$:

\begin{equation}\label{eq:GronwallAssumptionNoSup}
X\ti \leq \int_{(0,t]} \eta(X\sm) \dd A\s + M\ti + H\ti \qquad \PP\text{-a.s}.
\end{equation}
\end{defi}

We also use the two definitions above for processes defined on a finite time interval $[0,T]$ and correspondingly adapt the definition in this case.

\subsection{Notation and constants} \label{subsec:constants}
\textbf{Constants:} For $p\in(0,1)$ define the following constants:
\begin{equation}\label{eq:constants}
\beta = (1-p)^{-1},\quad   \alpha_1 = 
    (1-p)^{-1/p}, \quad  \alpha_2 = p^{-1}.
\end{equation}
\textbf{Quasinorms:} We denote by $|\cdot|$  the Euclidean norm and $|\cdot|_F$ the Frobenius norm. Let $Y$ be a random variable. We use for $p\in(0,1]$ the notation (if well-defined)
\begin{equation*}
\EO[Y] \define \E[\,Y\mid \F_0], \qquad \qquad 
\|Y\|_p \define \E[\,|Y|^p]^{1/p}, \qquad  \qquad \|Y\|_{p,\F_0} \define \E[\,|Y|^p\mid \F_0]^{1/p}.
\end{equation*}

\noindent
\textbf{Running supremum:} Let $X$ be a non-negative stochastic process with right-continuous
paths. We use the following notation for the running supremum and its left limits:
\begin{equation*}
\begin{aligned}
X^*\ti  & \define \sup_{0\leq s \leq t} X\s \qquad \forall t\geq 0, \\
X^*_{t^-} & \define \lim_{s\nearrow t} X^*\s = \sup_{s< t}X\s \qquad \forall t>0.
\end{aligned}
\end{equation*}
As usual, we set $X^*_{0^-} \define X_0$. If $X$ is c\`adl\`ag, then also $X^*_{t^-} = \sup_{s\leq t} X\sm$ holds true. If $X$ is only right-continuous then $X^*_t$ and $X^*_{t^-}$ take values in $[0,+\infty]$.\\

\noindent \textbf{Functions:} For $\eta:[c_0,\infty) \to [0,\infty)$ from \nameref{def:sup} or  \nameref{def:nosup}  we choose some $c>c_0$ and define the following functions for $p\in(0,1)$.
\begin{align}
G(x) & \define \int_c^x \frac{\dd u}{\eta(u)} \qquad & \forall x\in[c_0,\infty), \label{eq:defG}\\
\eta_p(x) & \define \frac{p}{1-p}\eta(x^{1/p})x^{1-1/p} \qquad   &\forall x\in[c^p_0,\infty)\cap(0,\infty), \label{eq:defEta_p} \\
\tilde G_p(x) & \define \int_{c^{p}}^{x} \frac{\dd u}{\eta_p(u)} \qquad  & \forall x\in[c^p_0,\infty)\cap(0,\infty).\label{eq:defTildeG}
\end{align}
The functions have the following properties:
\begin{itemize}
\item The function $G$ satisfies $G(c_0)\in[-\infty,0)$. Moreover, $G$ is increasing and concave. In particular, it has a well-defined increasing inverse $G^{-1}:\textrm{range}(G)\cap(-\infty,\infty) \mapsto [c_0,\infty)$. If $G(c_0) = -\infty$, then we set $G^{-1}(G(c_0)+a)\define c_0$ for $a\in(-\infty,\infty)$. If $\eta$ is continuous, then $G$ is continuously differentiable on $(c_0,\infty)$.
\item For any $h\geq 0$ and any  $a\in \R$ such that $ \int_h^{c_0}\frac{\dd u}{\eta(u)} < a < \int_h^{\infty}\frac{\dd u}{\eta(u)} $  the expression $G^{-1}(G(h) + a)$ is well-defined and does not depend on the choice of $c\in]c_0,\infty[$ used in the definition of $G$. In particular, the upper bound given in the Bihari-LaSalle inequality \cref{lemma:DetBihari} does not depend on the choice of $c$.
\item The functions $G$ and $\tilde{G}_p$ satisfy for all $x\in (c^p_0,\infty)$
\begin{equation}\label{eq:G_umrechnen}
\begin{aligned}
\tilde G_p(x) & = \frac{1-p}{p}\int_{c^p}^{x}\frac{\dd u}{ \eta(u^{1/p}) u^{1-1/p}} =  (1-p)\int_{c}^{x^{1/p}}\frac{\dd v}{ \eta(v)}  &= (1-p)G(x^{1/p})
\end{aligned}
\end{equation}
and for all $x\in\textrm{domain}(\tilde{G}^{-1}_p)$
\begin{equation}\label{eq:G_umrechnen2}
 \tilde{G}^{-1}_p(x) = \big(G^{-1}\big(\tfrac{x}{1-p} \big)\big)^p.
\end{equation}
\end{itemize}

\subsection{Overview of the results}\label{subsec:results}

The following two tables summarize some of the results in the literature and the results of this paper. The first table contains the stochastic Gronwall inequalities, the second table
contains generalizations of Bihari-LaSalle type. The constant $\alpha_1\alpha_2$ is the sharp constant from Lenglart's inequality, $\alpha_2$ is the sharp constant from a monotone version of Lenglart's inequality \cite{GeissScheutzow}.

\begin{table}[H]
\bgroup
\scriptsize
\def\arraystretch{2}
    \begin{tabular}{|p{2.2cm}|p{5.9cm}|p{6.1cm}|}
         \hline
             &  \nameref{def:nosup}, $\quad \eta(x) \equiv x$ \newline (Special case of $\Asup$)  & \nameref{def:sup}, $\quad \eta(x) \equiv x$ \\ \hline            
             $A$ deterministic,   \newline $p\in(0,1)$  
             & 
              See also the results in the case 
              '$\Asup$ with deterministic $A$' and '$\Anosup$ with random $A$'
              \newline \newline               
             $\blacktriangleright$ $H$ predictable or  
              $\Delta M \geq 0$: \newline
              $\|X^*_T\|_p \leq \alpha_1 \alpha_2 \|H_T\|_p 
              \e^{A_T} $ \newline and \newline
          $\PP[X_T^* > u] \leq \frac{e^{A_T}}{u} \E[H_T\wedge w] + \PP[H_T \geq w]$
              \newline See \cref{cor:GronwallNoSup}
               \newline\newline
              $\blacktriangleright$ $\E[H_T] <\infty$: \newline
              $\|X^*_T\|_p \leq \alpha_1 \|H_T\|_1 \e^{A_T}$
              \newline and \newline
          $\PP[X_T^* > u] \leq \frac{e^{A_T}}{u} \E[H_T]$
              \newline See \cref{cor:GronwallNoSup} 
             & 
               von Renesse and Scheutzow
              \cite[Lemma 5.4]{RenesseScheutzow}, Mehri and
               Scheutzow \cite[Theorem 2.1]{MehriScheutzow}:
                \newline
             $\blacktriangleright$ $H$ predictable: 
             $\|X^*_T\|_p \leq c_1\|H_T\|_p \e^{c_2 A_T} $  
              \newline             
             $c_1 = p^{-1/p}\alpha_1\alpha_2, \,\,$
             $c_2 = p^{-1}c_p^{1/p}$ \newline \newline
            
             $\blacktriangleright$ $\Delta M \geq 0$: 
             $\|X^*_T\|_p \leq c_1\|H_T\|_p \e^{c_2 A_T} $ 
              \newline
             $c_1 = ((c_p +1)/p)^{1/p}$
             $c_2 = p^{-1}(c_p +1)^{1/p}$ \newline \newline
             
             $\blacktriangleright$ $\E[H_T] <\infty$: $\|X^*_T\|_p \leq c_1 \|H_T\|_1\e^{c_2 A_T}$ \newline
             $c_1 = \alpha_1 \alpha_2 p^{-1/p}, \,\,$
             $c_2 = p^{-1}c_p^{1/p}$ \newline \newline
             
             \cref{cor:GronwallSup}: 
              Sharp constants for the three cases above \newline
               $\blacktriangleright$ $H$ predictable or  
              $\Delta M \geq 0$: \newline
              $c_1 = \alpha_1 \alpha_2, \,\,  c_2 = \beta$ \newline
             $\blacktriangleright$ $\E[H_T] <\infty$: \newline
               $c_1 = \alpha_1, \,\,  c_2 = \beta$ \\
            \hline
            $A$ random, \newline $0<q<p<1$
              & 
              $\blacktriangleright$ $H$ predictable or 
               $\Delta M \geq 0$ \newline
              $\|X^*_T\|_q \leq \alpha_1 \alpha_2 
              \|H_T\|_p \|e^{A_T}\|_{qp/(p-q)} $ 
              \newline and 
               \newline $\PP[X_T^* > u] $ \newline $\leq  \frac{e^{R}}{u}\E[H_T \wedge w ] + \PP[H_T\geq w] + \PP[A_T> R]$ 
              \newline See \cref{cor:GronwallNoSup}
              \newline\newline
              $\blacktriangleright$ 
              $\E[H_T] <\infty$: \newline
              $\|X^*_T\|_q \leq C \|H_T\|_1 \|e^{A_T}\|_{qp/(p-q)}$     
              \newline where C is given by: \newline \newline
              [Scheutzow \cite[Theorem 4]{Scheutzow}, \newline Xie and Zhang \cite[Lemma 3.7]{XieZhang}]:
              \newline
              $C = (pq^{-1}(1-p)^{-1})^{1/q}$
              \newline   \newline
              Slightly improved constant (\cref{cor:GronwallNoSup}):
              $C = \alpha_1$
              \newline
               and 
              \newline $\PP[X_T^* > u] \leq  \frac{e^{R}}{u}\E[H_T] + \PP[A_T> R]$ 
              & 
              $\blacktriangleright$ $H$ predictable or  $\Delta M \geq 0$: \newline
              $\|X^*_T\|_q \leq \alpha_1\alpha_2 \|H_T\|_p \|e^{\beta A_T}\|_{qp/(p-q)} $\newline See \cref{cor:GronwallSup}
              \newline\newline
              $\blacktriangleright$ $\E[H_T] <\infty$:  \newline
              $\|X^*_T\|_q \leq \alpha_1 \|H_T\|_1 \|e^{\beta A_T}\|_{qp/(p-q)}$  \newline See \cref{cor:GronwallSup}\\
     \hline
         Constants & \multicolumn{2}{c|}{$\beta = (1-p)^{-1},\quad   \alpha_1 = 
    (1-p)^{-1/p}, \quad  \alpha_2 = p^{-1}, \quad c_p = (\alpha_1 \alpha_2)^{p}$}  \\
    \hline
    Notation  & \multicolumn{2}{c|}{$\|Y\|_p \define \E[|Y|^p]^{1/p}$ for random variables $Y$, $\quad G(x) \define \int_c^x \frac{\dd u}{\eta(u)}$ $\quad \forall x\geq c_0$}  \\
       \hline 
    \end{tabular}    
    \egroup
\caption{Summary of stochastic Gronwall inequalities (i.e. $\eta(x)\equiv x$)}
\label{table:Gronwall}
\end{table}

\begin{table}[H] 
\bgroup
\scriptsize
\def\arraystretch{2}
    \begin{tabular}{|p{2.07cm}|p{6.05cm}|p{6.23cm}|}
         \hline
             & \nameref{def:nosup} \newline (Special case of $\Asup$) & \nameref{def:sup} \\ \hline
             $A$ random & 
             
             & \raggedright 
             $\blacktriangleright$  For $\int_1^\infty
             \frac{\dd u}{\eta(u)} = +\infty$ and continuous
             $M$ \newline see von Renesse and Scheutzow
             \cite[Lemma 5.1]{RenesseScheutzow}
             \newline \newline
             $\blacktriangleright$ If $X\geq c$, $\lim_{x\to\infty}\eta(x) = +\infty$, $\E[H_T]<\infty$ and  $\E[A^p_T]<\infty$:
             \newline
             $\|G(X^*_T)\|_p \leq 
             \alpha_1 \alpha_2 \|A_T +  G(\E[H_T])\|_p$ 
             \newline for $p\in(0,1)$, $G(x) \define \int_c^x 
             \frac{\dd u}{\eta(u)}$, $c>0$             
             \newline  See 
             \cite[Theorem 3.9]{GeissConcave}
             \newline\newline
             $\blacktriangleright$ If $\int_{0+}
             \frac{\dd u}{\eta(u)} = +\infty$ and 
             $\E[H_T]=0$: \newline
             $X_T^*=0 \quad \PP$-a.s. \newline
             See \cite[Theorem 3.9]{GeissConcave}
             \tabularnewline
             \hline            
             $\eta$ concave, \newline  $A$ deterministic, \newline $p\in(0,1)$    
             & 
              $\blacktriangleright$ $H$ predictable or  
              $\Delta M \geq 0$: \newline
              $\|X^*_T\|_p \leq \alpha_1 G^{-1}(G(\alpha_2 
              \|H_T\|_p)+A_T) $ 
              \newline See \cite[Theorem 3.1]{GeissConcave}
               \newline\newline
              $\blacktriangleright$ $\E[H_T] <\infty$: \newline
              $\|X^*_T\|_p \leq \alpha_1 G^{-1}(G(\|H_T\|_1)+  A_T)$
               \newline and \newline
              $\PP[X_T^*>u] \leq \frac{G^{-1}(G(\E[H_T]) + A_T)}{
              u}$
              \newline See \cite[Theorem 3.1]{GeissConcave} 
             & 
             $\blacktriangleright$ $H$ predictable or  $\Delta M \geq 0$: \newline
              $\|X^*_T\|_p \leq G^{-1}(G(\alpha_1\alpha_2\|H_T\|_p)+\beta A_T) $ 
              \newline See \cite[Theorem 3.1]{GeissConcave} \newline\newline
              $\blacktriangleright$ $\E[H_T] <\infty$: \newline
              $\|X^*_T\|_p \leq G^{-1}(G(\alpha_1 \|H_T\|_1)+\beta A_T)$ \newline 
              See \cite[Theorem 3.1]{GeissConcave}
             \\ 
             \hline
             $\eta$ convex,  \newline $A$ random, \newline $p\in(0,1)$     &
              $\blacktriangleright$ $H$ predictable or  
              $\Delta M \geq 0$: \newline
               $\|\sup_{t\in[0,T]} G^{-1}(G(X_t) - A_t)\|_p  \leq
              \alpha_1\alpha_2 \|H_T\|_p$ \newline
              See \cref{thm:stochBihariConvex}
              \newline \newline
             $\blacktriangleright$  $\E[H_T] <\infty$: \newline 
            $\|\sup_{t\in[0,T]} G^{-1}(G(X_t) - A_t)\|_p  \leq
              \alpha_1 \|H_T\|_1$
              \newline and \newline
              $\PP[X_T^*>u] \leq \frac{\E[H_T]}{G^{-1}(G(u)-R)} 
              + \PP[A_T>R]$
              \newline
              See \cref{thm:stochBihariConvex}
             & $\blacktriangleright$ $H$ predictable or  
              $\Delta M \geq 0$: \newline
               $\|G^{-1}(G(X_T^*) - \beta A_T)\|_p  \leq
              \alpha_1\alpha_2 \|H_T\|_p$ \newline
              See \cref{thm:stochBihariConvex}
              \newline \newline
             $\blacktriangleright$  $\E[H_T] <\infty$: \newline 
            $\|G^{-1}(G(X_T^*) - \beta A_T)\|_p  \leq
              \alpha_1 \|H_T\|_1$
              \newline
              See \cref{thm:stochBihariConvex}\\
     \hline
     Constants & \multicolumn{2}{c|}{$\beta = (1-p)^{-1},\quad   \alpha_1 = 
    (1-p)^{-1/p}, \quad  \alpha_2 = p^{-1}$}  \\
     \hline
     Notation  & \multicolumn{2}{c|}{$\|Y\|_p \define \E[|Y|^p]^{1/p}$ for random variables $Y$, $\quad G(x) \define \int_c^x \frac{\dd u}{\eta(u)}$ $\quad \forall x\geq c_0$}  \\
       \hline 
    \end{tabular}
    \egroup
\caption{Summary of stochastic Bihari-LaSalle inequalities obtained in \cite{RenesseScheutzow}, \cite{GeissConcave} and this paper}
\label{table:BihariLaSalle}
\end{table}
\color{black}


\section{Main results}\label{sec:MainResults}
In this section we provide generalizations of stochastic Gronwall inequalities and study the sharpness of the constants  and estimates. The proofs are contained in \cref{sec:proofs}. In \cref{sec:Gronwall} we compare the results of this paper with the  literature, and in particular formulate the results for the Gronwall case $\eta(x) \equiv x$. 

Recall the following definition from \eqref{eq:constants} for $p\in(0,1)$:
\begin{equation*}
\beta = (1-p)^{-1},\quad   \alpha_1 = 
    (1-p)^{-1/p}, \quad  \alpha_2 = p^{-1}.
\end{equation*}

\subsection{Stochastic Bihari-LaSalle inequalities for convex $\eta$}
For \nameref{def:nosup}  or \nameref{def:sup}, \textit{concave} $\eta$, deterministic $(A_t)_{t\geq 0}$ and suitable additional assumptions, estimates of the type
\begin{equation*}
\|X^*_T\|_p \leq \tilde{c}_1 G^{-1}(G( \tilde{c}_2 \|H_T\|_p)+ \tilde{c}_3 A_T)  \qquad \text{ for all } T\geq 0
\end{equation*}
can be shown for $p\in(0,1)$, see \cite{GeissConcave}. Here, $\tilde{c}_1$, $\tilde{c}_2$  and  $\tilde{c}_3$  denote constants that only depend on $p$ and 
\begin{equation*}
G(x) \define \int_c^x \frac{\dd u}{\eta(u)}
\end{equation*}
is the function from the deterministic Bihari-LaSalle inequality \cref{lemma:DetBihari}. Hence, in this case estimates with a similar structure as in the deterministic case are possible. However, for \textit{convex} $\eta$, even when $(H_t)_t \equiv H$ is a constant and $A_t\equiv t$, estimates of the type
\begin{equation}\label{eq:notPossibleBound}
\|X^*_T \|_p \leq  c_1 G^{-1}(c_2 G(c_3 H) + c_4 A_T)
\end{equation}
(where $c_1, c_2, c_3$ and $c_4$ are constants which depend only on $p\in(0,1)$) are in general false: For processes $X$ which satisfy \nameref{def:nosup} or \nameref{def:sup} for convex $\eta$,  the quantity $\|X_t^*\|_p$ may explode at finite time. This type of behaviour cannot be captured by a bound of the type \eqref{eq:notPossibleBound}, for a counterexample see \cref{sec:intro} or \cref{counterexample}.

However, the estimate of the deterministic Bihari-LaSalle inequality (see e.g. \cref{lemma:DetBihari})  can be rearranged to
\begin{equation*}
G^{-1}(G(x(t))-A(t))\leq H.
\end{equation*} 
This rearranged inequality can be generalized to the stochastic case for convex $\eta$. The following theorem can be used to study the finiteness of exponential moments of path-dependent SDEs, see \cref{sec:intro} and \cref{sec:applications}.

\begin{theorem}[A sharp stochastic Bihari-LaSalle inequality for convex $\eta$]\label{thm:stochBihariConvex} \phantom{a} 
\begin{enumerate} 
\item \label{item:convexSup} Let $(X, A, H, M)$ and $\eta$ satisfy \nameref{def:sup} and assume that $\eta_p \equiv \frac{p}{1-p}\eta(x^{1/p})x^{1-1/p}$ (defined in \eqref{eq:defEta_p}) is convex and  $C^1$, and $\eta_p(c_0^p)\define \lim_{x \to c_0^p}\eta_p(x)=0$. Then,
for all $p\in(0,1)$ and $T\geq 0$, the expression $ G^{-1}(G(X^*_T) - \beta A_T)$ is well-defined and the following estimates hold:
\begin{equation*}
\big\| G^{-1}(G(X^*_T) - \beta A_T)\big \|_{p,\F_0}  \leq 
\begin{cases}
\alpha_1\alpha_2 \| H_T\|\pF & \text{if }  \E[H_T^p] < \infty \text{ and }  H \text{ is predictable,} \\
\alpha_1\alpha_2 \| H_T\|\pF & \text{if } \E[H_T^p] < \infty \text{ and } \Delta M \geq 0, \\
\alpha_1 \| H_T\|_{1,\F_0}  & \text{if } \E[H_T] < \infty,
\end{cases}
\end{equation*}
where $\beta\define(1-p)^{-1}$, $\alpha_1 \define(1-p)^{-1/p}$ and $\alpha_2 \define p^{-1}$. We use the notation $(\Delta M_t)_{t\geq 0}  \define (M_t - M_{t^-})_{t\geq 0}$ and $\|Y\|_{p,\F_0} \define \E[\,|Y|^p\mid \F_0]^{1/p}$ for random variables $Y$.

\item \label{item:convexNoSup} Let $(X, A, H, M)$ and $\eta$ satisfy  \nameref{def:nosup} and assume that $\eta$ is convex and $C^1$ and $\eta(c_0)=0$. Then, for all $p\in(0,1)$, $T\geq 0$, $t\in[0,T]$ and $u, w >0$, the expressions $ G^{-1}(G(X_t) - A_t)$ and $ G^{-1}(G(X_T^*) - A_T)$ are well-defined and the following estimates hold:
\begin{equation}\label{eq:orignal-estimate}
\begin{aligned}
\big\| G^{-1}(G(X^*_T) - A_T)\big \|\pF
& \leq \big\|\sup_{t\in[0,T]} G^{-1}(G(X_t) - A_t)\big \|\pF  \\ 
& \leq   \begin{cases}
\alpha_1\alpha_2 \| H_T\|\pF & \text{if }  \E[H_T^p] < \infty \text{ and }  H \text{ is predictable,} \\
  \alpha_1\alpha_2 \| H_T\|\pF & \text{if } \E[H_T^p] < \infty \text{ and } \Delta M \geq 0, \\
  \alpha_1 \| H_T\|_{1,\F_0} & \text{if } \E[H_T] < \infty. 
\end{cases}
\end{aligned}
\end{equation} 
and 
\begin{equation}\label{eq:Mark1}
\begin{aligned}
\PP\bigg[\sup_{t\in[0,T]} G^{-1}(G(X_t) - A_t) > u \,\bigg|\, \F_0\bigg] 
& \leq \begin{cases}
  \frac{1}{u}\EO[H_T \wedge w ] + \PP[H_T\geq w \mid \F_0 ] & \text{if } H \text{ is predictable or }  \Delta M \geq 0 \\
   \frac{1}{u}\EO[H_T]\wedge u & \text{if } \E[H_T] < \infty. 
\end{cases}
\end{aligned}
\end{equation}
The latter can be reformulated as follows: For all  $u> c_0$, $w, R >0$ we have:
\begin{equation}\label{eq:Mark1-r}
\begin{aligned}
&\PP\bigg[\sup_{t\in[0,T]} X_t > u \,\bigg|\, \F_0\bigg]  \\
& \leq \begin{cases}
  \frac{\EO[H_T \wedge w ]}{G^{-1}(G(u) - R)} + \PP[H_T\geq w \mid \F_0 ] + \PP[A_T > R\mid \F_0] & \text{if } H \text{ is predictable or }  \Delta M \geq 0 \\[1em]
   \left(\frac{\EO[H_T]}{G^{-1}(G(u) - R)}\right)\wedge 1 + \PP[A_T > R\mid \F_0] & \text{if } \E[H_T] < \infty. 
\end{cases}
\end{aligned}
\end{equation}
\end{enumerate}
\end{theorem}

We prove \cref{thm:stochBihariConvex} b) by further developing the proof idea of Xie and Zhang \cite[Lemma 3.7]{XieZhang}. \cref{thm:stochBihariConvex} a) is more difficult to prove and requires new techniques. 

The constants $\alpha_1$, $\alpha_1\alpha_2$ and $\beta$ are sharp, for details see \cref{subsec:sharpness}. Under the stronger assumption $\Anosup$  \eqref{eq:Mark1} provides an upper bound for the weak $L^1$ norm, \cref{thm:sharp-tail} shows that under $\Asup$ the weak $L^1$ norm may be infinite.

\begin{remark}[On the relation between the convexity of $\eta_p$ and $\eta$] 
Assume that $c_0=0$ and $\eta_p(0)=0$. Then, convexity of $\eta_p$ implies that $\eta$ is convex: 
If $\eta_p$ is convex, then it is almost everywhere differentiable. For any $x>0$ in which $\eta_p$ is differentiable, we have for $y=x^{1/p}$
\begin{equation*}
\eta(y)  = \eta_p(y^p) y^{1-p} \text{ and } \eta'(y) = p  \eta_p'(y^p) + (1-p)\eta_p(y^p)y^{-p}.
\end{equation*}
Due to $\eta_p$ being convex and $\eta_p(0) = 0$, we have that $z\mapsto\frac{\eta_p(z)}{z}$ is non-decreasing. 
In particular, $\eta$ is  almost everywhere differentiable and $\eta'$ is non-decreasing (on its domain).
As convexity of $\eta_p$ implies, that $\eta_p$ (and hence also $\eta$) is locally Lipschitz continuous, we have that $\eta$ is absolutely continuous. Together, this implies that $\eta$ is convex. \\

However, convexity of $\eta$ does \textit{not} imply that $\eta_p$ is convex: For example $\eta(x) = x \arctan(x)$ is convex (due to $\eta''(x) = 2(x^2+1)^{-2} >0$) and $\eta_{1/2}(x) = x \arctan(x^2)$ is not convex (as $\eta_{1/2}''(x) = -2x(x^4-3)(x^4+1)^{-2}$).
\end{remark}

\begin{remark}
The function $\eta_p$ appears (upto the factor $1-p$) naturally in connection with the deterministic Bihari-LaSalle equalities: Let $x$ be as in \cref{lemma:DetBihari} and assume that $x$ is non-decreasing. Then, it can be shown, that $x^p$ (for some $p\in(0,1)$) satisfies:
\begin{equation}\label{eq:detBihari-p}
x(t)^p \leq (1-p)\int_0^t\eta_p(x(s^-)^p)\dd A(s) + H^p  \quad \text{for all } t\in[0,T].
\end{equation}

Moreover, for any continuous $\eta$ such that $x(t) = G^{-1}(G(H)+t)$ is well-defined for $t\in[0,T]$, 
we have that $x$ satisfies
\begin{equation*}
x(t) = \int_0^t \eta(x(s)) \dd s + H
\end{equation*}
and $x^p$ satisfies
\begin{equation*}
x^p(t) = (1-p)\int_0^t \eta_p( x^p(s)) \dd s + H^p.
\end{equation*}
\end{remark}

\begin{corollary}\label{cor:convex}
In the case \nameref{def:sup} and $\E[H_T]<\infty$ of the previous theorem we get the following estimates.
\begin{enumerate}
\item For $\eta(x)=x$  for all $x\geq 0$, we have
\begin{equation*}
 \| \e^{-\beta A_T}X^*_T \|_{p,\F_0} \leq \alpha_1 \|H_T\|_{1,\F_0}.
\end{equation*}
\item  If $X\geq 1$, $H \geq 1$ and $\eta(x) = x\log(x)$ for $x\geq 1$, we have 
\begin{equation*}
\|(X_T^*)^{\e^{-\beta A_T}}\|_{p,\F_0} \leq \alpha_1 \|H_T\|_{1,\F_0}.
\end{equation*}

\item  If $X\geq \e$, $H \geq \e$ and  $\eta(x)=x\log(x)\log(\log(x))$  for all $x\geq \e$, we have
\begin{equation*}
\| \e^{(\log(X^*_T))^{\e^{-\beta A_T}} }\|_{p,\F_0} \leq \alpha_1 \|H_T\|_{1,\F_0}.
\end{equation*}
\end{enumerate}

\end{corollary}
\begin{proof}[Proof of \cref{cor:convex}]
\begin{enumerate}
\item For $\eta(x)=x$ and $c=1$ we have $\eta_p(x)=\frac{p}{1-p} x$,
\begin{equation}
G(x) = \log(x) \text{ and  }G^{-1}(x) = \e^x.
\end{equation} 
\item For $\eta(x)=x\log(x)$ and $c=\e$ we have $\eta_p(x)=\frac{1}{1-p} x\log(x)$,
\begin{equation*}
G(x) = \log(\log(x)) \text{ and } G^{-1}(x) = \exp(\exp(x)).
\end{equation*}
\item For $\eta(x)=x\log(x)\log(\log(x))$ and $c=\e^\e$ we have  $\eta_p(x)=\frac{1}{1-p} x\log(x)\log(\log(x^{1/p}))$,
\begin{equation*}
G(x) = \log(\log(\log(x))) \text{ and }  G^{-1}(x) = \exp(\exp(\exp(x))).
\end{equation*}
\end{enumerate}
\end{proof}

\subsection{Sharpness of the constants and estimates}\label{subsec:sharpness}

In this section we study the sharpness of the constants $\beta$, $\alpha_1 \alpha_2$ and $\alpha_1$. These constants also appear in other generalizations of stochastic Gronwall inequalities, see \cite{GeissConcave}. Moreover, we prove that the assumption $\Asup$ does not imply a upper bound on the weak $L^1$ norm (in contrast to the stronger assumption $\Anosup$). \\

To study the sharpness of the stochastic Gronwall inequality with supremum we will use the following lemma. The idea to study a process $(X_t)_{t\geq 0}$ of the following type to prove sharpness of the constant $\beta$ is due to Michael Scheutzow.
\begin{lemma}\label{lemma:pathdependentSDE}
Let $\varepsilon, \delta \in (0,1)$ and let $l:[0,\delta \varepsilon) \to [0,\infty)$ be increasing, continuous, bijective and such that $\int_0^{\delta \varepsilon} l^2(u) \dd u = \infty$. We denote
\begin{equation*}
g_{\varepsilon, \delta}:[0,\infty) \to \{0,1\},  \qquad 
g_{\varepsilon, \delta}(s) \define \sum_{k=0}^\infty \one_{[(k+\varepsilon) \delta, (k+1)\delta]}(s).
\end{equation*}
Define for all $x\in C([0,\infty),\R)$
\begin{equation*}
\begin{aligned}
b(s,x) & \define g_{\varepsilon, \delta}(s) \sup_{0\leq u\leq s} x(u)  \qquad  & \forall s\geq 0,  \\
\sigma(s,x) & \define  
\one_{\{x(s)>0\}} \one_{(k\delta, k\delta + \varepsilon\delta)}(s) l
(s-k\delta)  \qquad  & \forall s\in(k\delta, (k+1)\delta], k\in \N.
\end{aligned}
\end{equation*}
Let $(W_t)_{t\geq 0}$ be a one-dimensional Wiener process on some underlying filtered probability space satisfying the usual conditions. Then there exists an adapted continuous non-negative process $(X_t)_{t\geq 0}$ enjoying the following properties:
\begin{enumerate}
\item For any (determistic) $t>0$ we have:
\begin{equation}\label{eq:lemma-path-dependent-SDE-1}
\int_0^t |b(s,X)| \dd s + \int_0^t |\sigma(s,X)|^2 \dd s < \infty \qquad \PP\text{-a.s.}
\end{equation}
\item  The process $(X_t)_{t\geq 0}$ is a solution of the path-dependent SDE
\begin{equation}\label{eq:path-dependent-SDE-sharpness}
X_t = 1 + \int_0^t b(s,X) \dd s + \int_0^t \sigma(s,X)\dd W_s \qquad \forall t\geq 0.
\end{equation}
In particular, $(X_t)_{t\geq 0}$ satisfies
\begin{equation*}
X_t \leq  \int_0^t X_s^* \dd s + M_t + 1 \qquad \forall t\geq 0
\end{equation*}
for $M_t \define \int_0^t \sigma(s,X)\dd W_s$. 
\item For all $p\in(0,1)$, $k\in \N_0$
\begin{equation*}
\E[(X_{k\delta + \varepsilon\delta}^*)^p] = \frac{1}{1-p} \left(1+\frac{p}{1-p}(1-\varepsilon) \delta\right)^k
\end{equation*}
holds true.
\end{enumerate}
\end{lemma}

\begin{theorem}[Sharpness of the constant $\beta$]\label{thm:sharpnessBeta}
Let $p\in(0,1)$ and assume that $\tilde{\alpha}$, $\tilde{\beta}$  are positive constants (depending on $p$) such that for any non-negative adapted continuous process $(X_t)_{t\geq 0}$ which satisfies 
\begin{equation}\label{eq:beta-sharpness-assumption0}
X_t \leq \int_0^t X^*_s \dd s + M_t + H, \qquad \forall t\geq 0.
\end{equation}
for some continuous local martingale $(M_t)_{t\geq 0}$ starting in $0$ and some constant $H>0$, we have for all $t\geq 0$
\begin{equation}\label{eq:beta-sharpness-assumption}
\|X^*_t\|_p \leq \tilde{\alpha} H \exp(\tilde{\beta} t).
\end{equation}
Then, $\tilde{\beta} \geq \beta \define (1-p)^{-1}$ holds true.
\end{theorem}

\begin{corollary}
The constant $\beta$ in \cref{thm:stochBihariConvex} \ref{item:convexSup} is sharp. It is already sharp when $\eta(x) \equiv x$, $A_t \equiv t$ , $H$ is a constant and $M$ and $X$ are continuous processes. 
\end{corollary}

\begin{theorem}[No $\Anosup$, no tail estimate of order $\mathcal {O}(1/u)$]\label{thm:sharp-tail} 
For any $T>0$ let $\varepsilon, \delta \in (0,1)$ and $k\in\N$ be chosen such that $T = k\delta + \varepsilon\delta$. Let $(X_t)_{t\geq 0}$ and $(M_t)_{t\geq 0}$  denote the process from \cref{lemma:pathdependentSDE}. Then $(X_t)_{t\geq 0}$ is a non-negative adapted continuous process and $(M_t)_{t\geq 0}$ a continuous local martingale starting in $0$ which satisfy 
\begin{equation}\label{eq:beta-sharpness-assumption0-1}
X_t \leq \int_0^t X^*_s \dd s + M_t + 1, \qquad \forall t\geq 0
\end{equation}
and
\begin{equation}\label{eq:beta-sharpness-assumption-1}
\sup_{u\geq 0}\left(u \PP\bigg[\sup_{t\in[0,T]} X_t\bigg] \right) = \infty.
\end{equation}
In particular, estimates of the form \eqref{eq:Mark1} do not hold in the case of \nameref{def:sup}.
\end{theorem}

\noindent The next theorem studies the sharpness of the constants $\alpha_1$ and $\alpha_1 \alpha_2$ which appear in \cref{thm:stochBihariConvex}. The constant $\alpha_1 \alpha_2$ is the sharp constant of Lenglart's domination inequality (see \cref{lemma:Lenglart}). In particular, \cref{thm:sharpnessAlpha} a) and b) are closely connected to \cite[Theorem 2.1]{GeissScheutzow}. The upper bound given in \cref{thm:sharpnessAlpha} a) is by Fatou's Lemma a special case of Lenglart's domination inequality. Assertion c) of \cref{thm:sharpnessAlpha} is known in the literature, see for example \cite[Theorem 7.6, p. 300]{Osekowski}.

\begin{theorem}[Sharpness of the constants $\alpha_1$ and $\alpha_1\alpha_2$]
\label{thm:sharpnessAlpha}
Assume \nameref{def:sup} and $A\equiv 0$, i.e. let $(X_t)_{t\geq 0}$ be an adapted non-negative right-continuous process, $(H_t)_{t\geq 0}$ be an adapted non-negative non-decreasing c\`adl\`ag process, $(M_t)_{t\geq 0}$ be a c\`adl\`ag local martingale with $M_0=0$. Assume that for all $t\geq 0$
\begin{equation*}
X\ti \leq M\ti + H\ti \qquad \PP\text{-a.s}.
\end{equation*}
Then the following assertions hold for $p\in(0,1)$.
\begin{enumerate}
\item \label{item:sharpnessAlpha1} If $H$ is predictable and $\E[H_T^p]<\infty$, then $\|X\ti^*\|\pF \leq \alpha_1 \alpha_2 \|H\ti\|\pF$ for all $t\in[0,T]$ and the constant $\alpha_1 \alpha_2= (1-p)^{-1/p}p^{-1}$ is sharp. The constant is already sharp if $X$ and $H$ are continuous and $M$ has no negative jumps.
\item \label{item:sharpnessAlpha2} If $M$ has no negative jumps and $\E[H_T^p]<\infty$, then $\|X\ti^*\|\pF \leq \alpha_1 \alpha_2 \|H\ti\|\pF$ for all $t\in[0,T]$ and the constant $\alpha_1 \alpha_2$ is sharp. The constant is already sharp if $X$ and $H$ are continuous.
\item\label{item:sharpnessAlpha3} If $\E[H_T] < \infty$, then $\|X\ti\|\pF \leq \alpha_1 \|H\ti\|_{1,\F_0}$  for all $t\in[0,T]$ and the constant $\alpha_1=(1-p)^{-1/p}$ is sharp. The constant is already sharp if $X$ and $M$ are continuous and $H$ is a constant.
\end{enumerate}
\end{theorem}

\begin{corollary}
The constants $\alpha_1 \alpha_2$ and $\alpha_1$ in \cref{thm:stochBihariConvex} are sharp and they are already sharp when $A\equiv 0$.
\end{corollary}

\section{Proofs of the results of Section 3}\label{sec:proofs}

\subsection{Lenglart's concept of domination and the Snell envelope}
We prove the stochastic Bihari-LaSalle inequality \cref{thm:stochBihariConvex} \ref{item:convexSup} by extending the proof technique of Lenglart's domination inequality, which we recall here for the convenience of the reader.

The following concept of domination was introduced by Lenglart in \cite[D\'efinition II]{Lenglart} and slightly generalized by Lenglart, L\'epingle, and Pratelli \cite[Lemma 1.4]{LLP}, see also Ren and Shen \cite{RenShen} and Mehri and Scheutzow \cite{MehriScheutzow}.

\begin{defi}[Lenglart's concept of domination]\label{def:domination}
Let 
\begin{itemize}
\item $(X_t)_{t\geq 0}$ be an adapted right-continuous non-negative process,
\item and $(H_t)_{t\geq 0}$ be a predictable \cadlagg non-negative non-decreasing  process,
\end{itemize}
such that for all bounded stopping times $\tau$
\begin{equation}\label{eq:def-lenglart}
\E[X_\tau \mid \F_0] \leq \E[H_\tau \mid \F_0] \leq \infty
\end{equation}
holds. Then we call $X$ dominated by $H$.
\end{defi}

By Fatou's lemma \eqref{eq:def-lenglart} also holds for all finite stopping times $\tau$.

The following lemma is \cite[Lemma 2.2]{MehriScheutzow}, which is a sharpened generalisation of  \cite[Th\'eor\`eme I, Corollaire II]{Lenglart}. See also the references listed in \cite{MehriScheutzow}.

\begin{lemma}[Lenglart's domination inequality] \label{lemma:Lenglart}
Let $X$ be dominated by $H$. Then, we have
\begin{enumerate}
\item for all $u>0$, $\lambda >0$ and $T>0$:
\begin{equation}\label{eq:LenglartProof1}
\PP[X_T^*>u \mid \F_0] \, u \leq \lambda \, \E[(H_T \lambda^{-1}) \wedge u\mid \F_0] + \PP[H_T\lambda^{-1} \geq u\mid \F_0] \, u.
\end{equation}
\item for all $p\in(0,1)$ and $T>0$:
\begin{equation*}
\E\bigg[\bigg(\sup_{t\in[0,T]} X_t\bigg)^p \,\bigg| \, \F_0 \bigg]^{1/p}
\leq \alpha_1 \alpha_2 \E\bigg[\bigg(\sup_{t\in[0,T]} H_t\bigg)^p \, \bigg|\, \F_0 \bigg]^{1/p},
\end{equation*}
where $\alpha_1\alpha_2 =(1-p)^{-1/p} p^{-1}$.
\end{enumerate}
\end{lemma} 

Note that \cref{lemma:Lenglart} b) follows from a) by integrating equation \eqref{eq:LenglartProof1} w.r.t. $pu^{p-1} \dd u$, e.g. using the formulas of \cref{rmk:formulaZ} and choosing $\lambda = p$.

\begin{remark}[Calculation of $Z^p$, $p\in(0,1)$]\label{rmk:formulaZ}
Let $Z$ be a non-negative random variable and $p\in(0,1)$.
Then $Z^p$ can be calculated using the three formulas below.
\begin{equation}\label{eq:formulaZ}
\begin{aligned}
Z^p &= p \int_0^\infty \one_{\{Z \geq u\}} u^{p-1} \dd u \\ 
Z^p & = (1-p) \int_0^\infty Z \one_{\{Z\leq u\}} u^{p-2} \dd u \\
Z^p &= p (1-p) \int_0^\infty (Z\wedge u)u^{p-2}\dd u
\end{aligned}
\end{equation}
The third equality follows e.g. by using the first and second equality.
In particular, we also have  $Z^{p-1}  = (1-p) \int_0^\infty \one_{\{Z\leq u\}} u^{p-2} \dd u$ for $Z>0$. The third equality exists more generally also for concave functions, see Burkholder \cite[Theorem 20.1, p.38-39]{Burkholder} and Pratelli \cite[Proposition 1.2]{Pratelli}.
\end{remark}

We use the Snell envelope contained in \cite[Appendix 1: (22), p.416-417]{DellacherieMeyer}. This version uses optional strong supermartingales. An optional strong supermartingale $(Z_t)_{t\geq 0}$ is an optional process such that for any bounded stopping time $\tau$ the random variable $Z_\tau$ is integrable and such that for any pair of bounded stopping times $\sigma\leq \tau$ the inequality $\E[X_{\tau} \mid \F_{\sigma}] \leq X_{\sigma}$ holds almost surely (see \cite[Appendix 1: Definition I, p.393-394]{DellacherieMeyer}). Note that \cadlagg supermartingales are optional strong supermartingales.

The following corollary of the Snell envelope is useful to prove the convex Bihari-LaSalle inequality \cref{thm:stochBihariConvex} a). 
Alternatively, with some more work the Snell envelope could also be directly applied  in the proof
\cref{thm:stochBihariConvex} a). However, we prefer to state the following corollary as it yields in addition also a characterization of Lenglart's concept of domination. As the author did not find this corollary in the literature, a proof is provided in the appendix.

\begin{corollary}[Characterization of Lenglart's concept of domination]\label{cor:characterizationOfLenglartDomination}
Let $X$ be dominated by $H$ and assume that $\E[H_0] <\infty$. Then there exists a \cadlagg local  supermartingale $(N_t)_{t\geq 0}$ with $N_0\leq 0$ such that 
\begin{equation*}
X_t\leq H_t + N_t \qquad \text{ for all } t\geq 0.
\end{equation*}
\end{corollary}

By the general Doob-Meyer decomposition theorem we could also replace 'local supermartingale' by 'local martingale' in the corollary.

\subsection{Main Lemma: A Lenglart type estimate}
An extension of \cref{lemma:Lenglart} a) is provided in \cref{lemma:1} for the cases
\nameref{def:sup} and \nameref{def:nosup}. This lemma is one of the key steps of the proof of \cref{thm:stochBihariConvex} \ref{item:convexSup}. Moreover, \cref{lemma:1} can be used to prove other generalizations of stochastic Gronwall inequalities, which is done in the closely connected paper \cite{GeissConcave}.

\begin{lemma}[Lenglart type estimates]\label{lemma:1}
Fix some $T>0$ and $p\in(0,1)$ and let $(X,A,H,M)$ satisfy \nameref{def:sup} or \nameref{def:nosup}. We consider the following 6 cases, which arise from combining $\Asup$ or $\Anosup$ with one of the following three assumptions:
\begin{enumerate}
\item $H$ is predictable,
\item $M$ has no negative jumps,
\item $\E[H_T]<\infty$.
\end{enumerate}
Fix arbitrary $u,\lambda >0$  and set:
\begin{equation*}
\tau_u \define \tau \define \inf\{s \geq 0 \mid H\s \geq \lambda u\}, \qquad \sigma_u \define \sigma \define \inf\{ s\geq 0 \mid X\s > u\},
\end{equation*}
where $\inf\emptyset \define +\infty$. Then, the following estimate holds true for all $t\in[0,T]$:
\begin{equation}\label{eq:lemma1}
\one_{\{X^*_t > u\}} u \leq  X_{t\wedge \sigma_u} \wedge u  \leq I^{L,u}\ti + M^{L,u}\ti + H^{L,u}\ti.
\end{equation}
Here
$(I^{L,u}\ti)_{t\geq 0}$ is a non-decreasing process containing the integral term from \eqref{eq:GronwallAssumption} and \eqref{eq:GronwallAssumptionNoSup} respectively with an additional indicator function
\begin{equation*}
I^L\ti \define I^{L,u}\ti \define
\begin{cases}
\int_{(0,t]} \eta(X^*\sm) \one_{\{X^*\sm \, \leq u \}} \dd A\s  & \text{for \nameref{def:sup}}, \\[1em]
\int_{(0,t]} \eta(X\sm) \one_{\{X^*\sm \, \leq u \}} \dd A\s  & \text{for \nameref{def:nosup}}, 
\end{cases}
\end{equation*}
the process $(M^{L,u}\ti)_{t\geq 0}$ is a local martingale with \cadlagg paths starting in $0$ defined by
\begin{equation*}
M^L\ti \define M^{L,u}\ti \define
\begin{cases}
 \lim_{n\to\infty} M_{t\wedge \tau^{(n)} \wedge \sigma}  & \text{ if $H$ is predictable},  \\
M_{t\wedge \tau \wedge \sigma} &  \text{ if $M$ has no negative jumps}, \\  
\tilde M_{t\wedge \sigma}\one_{\{\EO[H_T]\leq u\}}  & \text{ if } \E[H_T]<\infty,
\end{cases}
\end{equation*}
(where $\tau^{(n)}$ denotes an announcing sequence of $\tau$ and $\tilde M\ti \define M\ti + \E[H_T\mid F_t] -\EO[H_T]$ for $t\in[0,T]$), 
and $(H^{L,u}\ti)_{t\geq 0}$ is a non-decreasing process depending on $H$:
\begin{equation*}
H^L\ti \define H^{L,u}\ti \define 
\begin{cases}
H\ti \wedge (\lambda u)  + u\one_{\{H\ti \geq  \lambda u\}} & \text{ if $H$ is predictable},  \\
H\ti \wedge (\lambda u)  + u\one_{\{H\ti \geq  \lambda u\}}  &  \text{ if $M$ has no negative jumps}, \\  
\EO[H_T]\wedge u &  \text{ if } \E[H_T]<\infty.  
\end{cases}
\end{equation*}
\end{lemma}

\begin{remark}[Connection of \cref{lemma:1} and Lenglart's inequality]
\cref{lemma:1} is connected to Lenglart's inequality as follows: Let $(X, A, H, M)$ satisfy \nameref{def:sup} or \nameref{def:nosup}. Moreover, assume that $A\equiv 0$ and that $H$ is predictable. Then, we can apply \cref{lemma:1} yielding
\begin{equation}\label{eq:remark64-1}
\one_{\{X^*_t > u\}} u  \leq M^{L,u}\ti + H\ti \wedge (\lambda u)  + u\one_{\{H\ti \geq  \lambda u\}}.
\end{equation}
Under the assumptions listed above, $X$ is also dominated by $H$. Hence, we can also apply Lenglart's inequality and obtain:
\begin{equation}\label{eq:remark64-2}
\PP[X_T^*>u \mid \F_0] \, u \leq \lambda \, \E[(H_T \lambda^{-1}) \wedge u\mid \F_0] + \PP[H_T\lambda^{-1} \geq u\mid \F_0] \, u,
\end{equation}
Taking the conditional expectation given $\F_0$ of inequality \eqref{eq:remark64-1} implies by monotone convergence inequality \eqref{eq:remark64-2}.
\end{remark}

\begin{proof}[Proof of Lemma \ref{lemma:1}]
The inequality $\one_{\{X^*_t > u\}}u \leq X_{t\wedge \sigma}\wedge u$ follows easily from the right-continuity and non-negativity of $X$. It remains to prove the second inequality of \eqref{eq:lemma1}. \\

We denote by $Y$ the upper bound for $X$ which is given by \nameref{def:sup} or \nameref{def:nosup} respectively, i.e. for all $t\geq 0$
\begin{equation*}
X\ti \leq Y\ti \define M_{t} + 
H\ti + \begin{cases}
\int_{(0,t]}\eta(X^*\sm) \dd A\s &\text{for } \Asup, \\[1em]
\int_{(0,t]} \eta(X\sm) \dd A\s  & \text{for } \Anosup.
\end{cases}
\end{equation*}
\bigskip

\noindent \textbf{Step (a):}
We first prove the  inequality for the case that $H$ is predictable. Fix some $t\in[0,T]$. Because $H$ is predictable there exists a sequence of stopping times $(\tau^{(n)})_{n\in\N}$ that announces $\tau$. 
In particular, due to $H$ being non-decreasing, we have on $\{H_0 < \lambda u\}$ the inequality $H_{t\wedge\tau\nn\wedge\sigma} \leq H\ti\wedge (\lambda u)$.
Moreover, by definition of $\sigma$ the equality 
\begin{equation*}
\{s\leq \sigma\}  = \{s > \sigma\}^c   = \{ \exists r<s \mid X_r>u \}^c = 
\{X^*\sm > u\}^c = \{X^*\sm \leq u\}
\end{equation*}
 holds true for all $s>0$. Therefore, using 
 \eqref{eq:GronwallAssumption} or \eqref{eq:GronwallAssumptionNoSup} respectively, we have on $\{H_0 < \lambda u\}$:
\begin{equation*}
\begin{aligned}
 X_{t\wedge \tau^{(n)} \wedge \sigma}  & \leq M_{t\wedge \tau^{(n)} \wedge \sigma} + 
H_{t\wedge \tau^{(n)} \wedge \sigma} + \begin{cases}
\int_{(0,t]}\eta(X^*\sm) \one_{\{s\leq \sigma\} \cap \{s\leq \tau^{(n)}\}} \dd A\s &\text{for } \Asup, \\[1em]
\int_{(0,t]} \eta(X\sm) \one_{\{s\leq \sigma\} \cap \{s\leq \tau^{(n)}\}} \dd A\s  & \text{for } \Anosup.
\end{cases} \\
&\leq M_{t\wedge \tau^{(n)} \wedge \sigma} +  H\ti \wedge (\lambda u) + I^L\ti.
\end{aligned}
\end{equation*}
Moreover, note that due to non-negativity of $X$, we have
\begin{equation*}
X_{t\wedge \sigma}\wedge u - X_{t\wedge \tau^{(n)} \wedge \sigma}\wedge u \leq u \one_{\{\tau^{(n)} < t \}}.
\end{equation*}
The previous two inequalities and the definitions of $\tau$ and $(\tau^{(n)})_n$ imply on $\{H_0 < \lambda u\}$:
\begin{equation*}
\begin{aligned}
X_{t\wedge \sigma}\wedge u & \leq \limsup_{n\to\infty} X_{t\wedge \tau^{(n)} \wedge \sigma}\wedge u + \limsup_{n\to\infty} \big(X_{t\wedge \sigma}\wedge u - X_{t\wedge \tau^{(n)} \wedge \sigma}\wedge u\big)  \\
& \leq I^L\ti +  \lim_{n\to\infty}M_{t\wedge \tau^{(n)} \wedge \sigma} +  H\ti \wedge (\lambda u) + \lim_{n\to\infty} u \one_{\{\tau^{(n)} < t \}} \\
& \leq I^L\ti +  \lim_{n\to\infty}M_{t\wedge \tau^{(n)} \wedge \sigma} +  H\ti \wedge (\lambda u) + u \one_{\{H\ti \geq \lambda u\}}.
\end{aligned}
\end{equation*}
Hence, we have proven that \eqref{eq:lemma1} holds true on $\{H_0 < \lambda u\}$. On $\{H_0 \geq \lambda u\}$ we bound $X_{t\wedge \sigma} \wedge u$ by $u$, i.e. we have 
\begin{equation*}
X_{t\wedge \sigma}\wedge u \one_{\{H_0 \geq \lambda u\}} \leq  u\one_{\{H_0 \geq \lambda u\}} \leq  u\one_{\{H\ti \geq \lambda u\}}.
\end{equation*}
Noting that on $\{H_0 \geq \lambda u\}$ we have $\tau^{(n)} = \tau = 0$ and $M_0 = 0$, this implies by non-negativity of $I^L\ti$ the inequality \eqref{eq:lemma1} on $\{H_0 \geq \lambda u\}$ . \bigskip

\noindent \textbf{Step (b):} Next we prove the inequality for the case that $M$ has no negative jumps. Fix again some $t\in[0,T]$. In the proof of the previous assertion we used  $\lim_{n\to\infty} H_{t \wedge \sigma \wedge \tau^{(n)}} \leq H\ti\wedge (\lambda u)$ on $\{H_0 < \lambda u\}$. As the existence of an announcing sequence $(\tau^{(n)})_{n\in\N}$ of $\tau$ is not guaranteed if $H$ is not predictable, we need to take into account that $H$ might jump at time $\tau$ above $\lambda u$. To this end, we define for all $t\geq 0$:
\begin{equation*}
\begin{aligned}
\tilde{X}\ti &\define X_{t\wedge\tau}\one_{\{t<\tau\}} = X\ti \one_{\{t<\tau\}}, \\
\tilde Y\ti &\define M_{t\wedge \tau} + 
H\ti\wedge (\lambda u) + \begin{cases}
\int_{(0,t\wedge \tau]}\eta(X^*\sm) \dd A\s &\text{for } \Asup, \\[1em]
\int_{(0,t\wedge \tau]} \eta(X\sm) \dd A\s  & \text{for } \Anosup.
\end{cases}
\end{aligned}
\end{equation*}
It can be seen that $\tilde X$ is an adapted right-continuous process. 
Furthermore, $\tilde{Y}$ is non-negative:  For $t<\tau$, $\tilde{Y}_t =Y_t \geq X_t \geq 0$ by assumption. For $t\geq \tau$, $\tilde{Y}_t =\tilde{Y}_{\tau}$. The local martingale $M$ having no negative jumps implies that $\tilde{Y}$ has no negative jumps, and hence $0\leq Y_{\tau^-}= \tilde{Y}_{\tau^-} \leq \tilde{Y}_{\tau}$.  Hence, $\tilde{Y}$ is non-negative. This implies 
\begin{equation*}
\tilde X\ti 
= \begin{cases} X_t & \text{ on } \{t<\tau\} \\   0 & \text{ on } \{t\geq \tau\} 
\end{cases} 
\leq     \begin{cases} Y_t & \text{ on } \{t<\tau\} \\   \tilde{Y}_t & \text{ on } \{t\geq \tau\} 
\end{cases} = \tilde Y\ti
\end{equation*}
for all $t\geq 0$. By construction we have $\tilde X\s = X\s$ for $s\in[0,t]$ on $\{\tau > t\} = \{H_t<\lambda u\}$, and hence
\begin{equation*}
X_{t\wedge\sigma} \wedge u - \tilde{X}_{t\wedge \sigma \wedge \tau}\wedge u\leq 
(X_{t\wedge\sigma} \wedge u - X_{t\wedge \sigma \wedge\tau}\wedge u)\one_{\{t<\tau\}} + u\one_{\{t\geq \tau\}}  = u\one_{\{t\geq \tau\}}
\end{equation*}
Moreover, we have as in the proof of step (a) $\{s\leq \sigma\}= \{X^*\sm \leq u\}$. Therefore, on $\{H_0 <\lambda u \}$, we have for all $t\in[0,T]$
\begin{equation*}
\begin{aligned}
X_{t\wedge \sigma} \wedge u & = \tilde{X}_{t\wedge \tau \wedge \sigma}\wedge u + \big(X_{t\wedge \sigma} \wedge u - \tilde{X}_{t\wedge\tau \wedge \sigma}\wedge u\big)  \\
& \leq \tilde{Y}_{t\wedge \sigma \wedge \tau} + u \one_{\{\tau \leq  t \}} \\
& \leq I^L_t + M_{t\wedge\tau \wedge \sigma} + H_t \wedge (\lambda u) + u \one_{\{H_t \geq \lambda u \}}.
\end{aligned}
\end{equation*}
On $\{H_0 \geq \lambda u \} = \{\tau = 0\}$  we have as before 
$X_{t\wedge \sigma} \wedge u  \leq u \one_{\{H_t \geq \lambda u\}}$, implying that \eqref{eq:lemma1} is satisfied on $\{H_0 \geq \lambda u \}$. Combining the inequalities for $\{H_0 < \lambda u\}$ and $\{H_0 \geq \lambda u\}$  implies the claim. \\

\noindent \textbf{Step (c):} Now we prove the inequality for the case that $\E[H_T]<\infty$. We weaken \eqref{eq:GronwallAssumption} and \eqref{eq:GronwallAssumptionNoSup} as follows:
\begin{equation}
X\ti \leq \tilde M\ti +  \EO[H_T] + \begin{cases}
\int_{(0,t]} \eta(X^*\sm) \dd A\s  &  \text{for } \Asup, \\[1em]
\int_{(0,t]} \eta(X\sm) \dd A\s &  \text{for } \Anosup,
\end{cases} 
\end{equation}
where $\tilde M\ti \define M\ti + \E[H_T\mid F_t] -\EO[H_T]$. As our filtration satisfies by assumption the usual conditions, we may assume w.l.o.g. that $\tilde{M}$ is c\`adl\`ag. Using again, that $\{s \leq \sigma\} = \{X^*_{s^-} \leq u\}$, we obtain  for all $u>0$ on $\{\EO[H_T]\leq u\}$:
\begin{equation*}
\begin{aligned}
X_{t\wedge \sigma}\wedge u \leq I^L_t+  \tilde M_{t\wedge \sigma}  + \EO[H_T]\wedge u.
\end{aligned}
\end{equation*}
Noting that $X_{t\wedge \sigma}\wedge u \leq \E[H_T\mid \F_0] \wedge u$ on $\{\EO[H_T]\geq u\}$ implies the claim. 
\end{proof}

For the proof of \cref{thm:stochBihariConvex}\ref{item:convexNoSup} we need the following lemma, which is an immediate consequence of combining \cref{rmk:formulaZ} with  \cref{lemma:1}. For predictable $H$ \cref{lemma:2} is an immediate corollary of Lenglart's domination inequality. \cref{lemma:2} provides the upper bounds of \cref{thm:sharpnessAlpha}.
 
\begin{lemma}\label{lemma:2}
Let \nameref{def:sup} and $A\equiv 0$ hold and let $p\in(0,1)$.
\begin{enumerate}
\item If $H$ is predictable or $M$ has no negative jumps and $\E[H_T^p]<\infty$, we have for all $T>0$:
\begin{equation*}
\|X_T^*\|\pF \leq \alpha_1 \alpha_2 \|H_T\|\pF.
\end{equation*}

\item If $\E[H_T]<\infty$ we have for all $T>0$:
\begin{equation*}
\|X_T^*\|\pF \leq \alpha_1 \|H_T\|_{1,\F_0}.
\end{equation*}
\end{enumerate}
\end{lemma}
\begin{proof}[Proof of \cref{lemma:2}]
Assume w.l.o.g. that $M$ is a martingale, and hence also $M^{L,u}$ for any $u>0$. The inequalities can be proven by taking the conditional expectation of \eqref{eq:lemma1} given $\F_0$ and integrating  w.r.t $p u^{p-2}\dd u$. This gives (due to $A\equiv 0$):
\begin{equation*}
\begin{aligned}
\EO\bigg[ \int_0^\infty \one_{\{X^*_t > u\}} p u^{p-1} \dd u \bigg] & \leq  \int_0^\infty \EO[M^{L,u}\ti]p u^{p-2}\dd u + \EO\bigg[\int_0^\infty H^{L,u}\ti p u^{p-2} \dd u\bigg] \\
 & = \EO\bigg[\int_0^\infty H^{L,u}\ti p u^{p-2} \dd u\bigg].
\end{aligned}
\end{equation*}
It remains to compute the single terms. By \cref{rmk:formulaZ} we have:
\begin{equation*}
p \int_0^\infty \one_{\{X^*_t > u\}}  u^{p-1} \dd u  = (X_t^*)^p
\end{equation*}
If $\E[H_T^p] <\infty$ and either $H$ is predictable or $\Delta M \geq 0$, we have (choosing $\lambda = p$ and applying \cref{rmk:formulaZ} and recalling $\alpha_1 = (1-p)^{-1/p} $, $\alpha_2 =p^{-1}$):
\begin{equation}\label{eq:Hp-calc}
\begin{aligned}
\int_0^\infty H^{L,u}\ti p u^{p-2} \dd u 
& = \int_0^\infty \big(H\ti \wedge (\lambda u)  + u\one_{\{H\ti \lambda^{-1}\geq u\}} \big) p u^{p-2} \dd u \\
& = \lambda (1-p)^{-1}  (H_t \lambda^{-1})^p + (H_t \lambda^{-1})^p  \\
& = \alpha_1^p \alpha_2^p  H_t^p. 
\end{aligned}
\end{equation}
If $\E[H_T] <\infty$, we have: 
\begin{equation}\label{eq:Hp-calc2}
\int_0^\infty H^{L,u}\ti p u^{p-2} \dd u  = \int_0^\infty\big(\EO[H_T]\wedge u\big) p u^{p-2} \dd u =  (1-p)^{-1}  \EO[H_T]^p = \alpha_1^p  \EO[H_T]^p. 
\end{equation}
Combining the calculations above gives the claim.
\end{proof}

\color{black}


\subsection{Proof of the convex stochastic Bihari-LaSalle inequality}

The following lemma allows us to assume w.l.o.g. that the integrator $A$ is continuous and adapted instead of predictable. It can be proven by a time change argument, by smoothening out the large jumps of $A$. The proof is hidden away in the appendix. Recall that in \nameref{def:nosup} and \nameref{def:sup} the process $A$ is assumed to be predictable. In general, the assertion of the lemma is false if $A$ is not predictable, for a counterexample see \cref{example:notpredictable}.

\begin{lemma}[Continuity of integrator by time change]\label{lemma:timeshift}\label{lemma:smoothenJumps}
Assume that $(X_t)_{t\geq 0}$, $(A_t)_{t\geq 0}$, $(H_t)_{t\geq 0}$ and $(M_t)_{t\geq 0}$ satisfy \nameref{def:nosup} (or \nameref{def:sup}) on some filtered probability space  $(\Omega, \F,\PP, (\F_t)_{t\geq 0})$.
Denote by $A^c$ the continuous part of $A$ i.e. 
$A^c_t \define A_t-\sum_{s\leq t} \Delta A_s$ for all $t\geq 0$. Assume that $A^c$ is strictly increasing and $A^c_{\infty} = +\infty$. Then, there exists a time-changed version of $(X, A, H, M)$ such that the integrator is continuous, i.e. a family of stochastic processes  $(\tilde{X}, \tilde{A}, \tilde{H}, \tilde{M})$ satisfying \nameref{def:nosup} (or \nameref{def:sup}) on a filtered probability space  $(\Omega, \F,\PP, (\tilde{\F}_t)_{t\geq 0})$ such that:
\begin{enumerate}
\item $\tilde{A}_t = t$ for all $t\geq 0$,
\item for all $t\geq 0$:
\begin{equation*}
\tilde{X}_{A_t} = X_t, \quad  \tilde{M}_{A_t} = M_t, \quad  \tilde{H}_{A_t} = H_t,
\end{equation*}
and $A_s$ is a $(\tilde{\F}_t)_{t\geq 0}$ stopping time for every $s\geq 0$, 
\item if $H$ is predictable then $\tilde{H}$ is predictable,
\item if $M$ has no negative jumps then $\tilde{M}$ has no negative jumps,
\item and $(\Omega, \F,\PP, (\tilde{\F}_t)_{t\geq 0})$ satisfies the usual conditions.
\end{enumerate}
\end{lemma}

\begin{remark}
The assertion of \cref{lemma:timeshift} is not trivial if $A$ has jumps: Simply extending $A^{-1}(\omega): \rm{Image}(A)(\omega) \mapsto [0,\infty)$ to $A^{-1}(\omega): [0,\infty) \mapsto [0,\infty)$ and then setting $\tilde{Y} \define Y_{A^{-1}_t}$ for $Y\in\{X, M, H\}$ will not work in general, because these processes will not satisfy  inequality \eqref{eq:GronwallAssumptionNoSup} of \nameref{def:nosup} for $t\notin\rm{ Image}(A)(\omega)$. For more details see the first step of the proof.
\end{remark}

\begin{remark}\label{rmk:timeshift}
\cref{lemma:timeshift} can be also applied if $A^c$ is not strictly increasing and if $A^c_\infty <\infty$ by setting $\bar{A}_t \define A_t + \delta t$ for some $\delta>0$ and applying \cref{lemma:timeshift} to $(X, \bar{A}, H, M)$.
\end{remark}

\begin{remark}\label{rmk:HgeqEps} 
In the proofs of the stochastic Bihari-LaSalle inequalities, we will always assume $X\geq \varepsilon+c_0$ and $H\geq \varepsilon+c_0$ for some $\varepsilon>0$ instead of $X\geq c_0$ and $H\geq c_0$ (where $c_0$ is the constant from $\eta:[c_0,\infty) \to [0,\infty)$ in  \nameref{def:sup}  and \nameref{def:nosup}).  We do this to ensure that terms like $G(X_t)$ or $G(H_t)$ are finite.
We may do this without loss of generality because we can add an arbitrary $\varepsilon>0$ to \eqref{eq:GronwallAssumption} (or similarly \eqref{eq:GronwallAssumptionNoSup}) and slightly weaken \eqref{eq:GronwallAssumption} (using that $\eta$ is non-decreasing) to obtain for all $t\in[0,T]$:
\begin{equation*}
(X\ti + \varepsilon) \leq \int_{(0,t]} \eta(X^*\sm+\varepsilon) \dd A\s + M\ti + (H\ti + \varepsilon) \qquad \PP\text{-a.s}.
\end{equation*}
Proving the assertions of the theorems for the processes $(X\ti + \varepsilon)_{t\geq 0}$, $(A\ti)_{t\geq 0}$, $(M\ti)_{t\geq 0}$ and $(H\ti + \varepsilon)_{t\geq 0}$ and then taking the limit $\varepsilon\to 0$ will imply the assertions for the general case $X\geq c_0$, $H\geq c_0$. 
\end{remark}

\begin{proof}[Proof of \cref{thm:stochBihariConvex}]
\noindent\textbf{Proof of \ref{item:convexNoSup}:} We first prove the claim for continuous $A$. We assume w.l.o.g. that $H\geq c_0 + \varepsilon$  and $X\geq c_0  + \varepsilon$ on $\Omega$ for some constant $\varepsilon>0$. \\
We define
\begin{equation*}
f\colon (c_0,\infty)\times[0,\infty)\mapsto (c_0,\infty), \qquad (x,a) \mapsto G^{-1}(G(x) - a),
\end{equation*}
noting that $\| f(X_T^*,A_T)\|_p$ is the quantity we need to find an upper bound for. The function $f$ is indeed well-defined: Due to the convexity of $\eta$ and $\eta(c_0) = 0$ there exists some $K>0$ s.t. $\eta(x) \leq K (x-c_0)$ for all $x\in[c_0,c]$ where $c$ denotes the constant from the definition of $G$, see \eqref{eq:defG}. This implies $\lim_{x\to c_0} G(x) = -\infty$ and therefore $\textrm{domain}(G^{-1}) = \textrm{range}(G) = (-\infty, \lim_{x\to\infty} G(x))$. Moreover, we have $\textrm{range}(G^{-1}) = \textrm{domain}(G) = (c_0,\infty)$. Therefore, $f$ is well-defined.

Moreover, we have for all $x\in(c_0,\infty), a\in[0,\infty)$:
\begin{equation*}
\begin{aligned}
\tfrac{\partial}{\partial x} f(x,a) & = \frac{G'(x)}{G'(G^{-1}(G(x)-a)} = \frac{\eta(f(x,a))}{\eta(x)}, \\
\tfrac{\partial}{\partial a} f(x,a) & = - \eta(f(x,a)), \\
\tfrac{\partial^2}{\partial x^2} f(x,a) & = \frac{\eta(f(x,a))}{(\eta(x))^2}\big( \eta'(f(x,a))-\eta'(x)\big). \\
\end{aligned}
\end{equation*}

Denote by $(Y_t)_{t\geq 0}$ to be the right-hand side of \eqref{eq:GronwallAssumptionNoSup}. Instead of finding an upper bound for $\| f(X_T^*,A_T)\|_p$, it suffices by $X_t \leq Y_t$ to find an upper bound for $\| f(Y_T^*,A_T)\|_p$. To this end we will estimate $(f(Y_t,A_t))_{t\geq 0}$ using  It\^o's formula.

 We first show that the jump term
\begin{equation*}
\sum_{s\leq t} f(Y_s, A_s)-f(Y_{s-}, A_s) - \tfrac{\partial}{\partial x} f(Y_{s^-},A_s)\Delta Y_s 
\end{equation*}
that occurs in the It\^o formula for $(f(Y_t,A_t))_{t\geq 0}$ is non-positive. (Recall that we assumed that $A$ is continuous.) By assumption $\eta'$ is non-decreasing and $f(x,a) \leq x$, therefore $\tfrac{\partial^2}{\partial x^2} f(x,a) \leq 0$ holds.  This implies by a Taylor's expansion that for all fixed $a>0$ and for all $x, x+\Delta x\in(c_0,\infty)$
\begin{equation*}
f(x+\Delta x, a) - f(x, a) -  \tfrac{\partial}{\partial x} f(x,a)\Delta x \leq 0,
\end{equation*}
and therefore the jump term in  It\^o formula for $(f(Y_t,A_t))_{t\geq 0}$ is non-positive. Hence, It\^o's formula implies:
\begin{equation*}
\begin{aligned}
f(Y_0,A_0) & = H_0, \\
\dd f(Y_t, A_t) & \leq \tfrac{\partial}{\partial x} f(Y_{t^-},A_t)  \dd Y_t + 
\tfrac{\partial}{\partial a} f(Y_{t^-},A_t)  \dd A_t + \frac{1}{2} 
\tfrac{\partial^2}{\partial x^2} f(Y_{t^-},A_t) \dd \langle Y^c,Y^c \rangle_t \\
& \leq \eta(f(Y_{t^-},A_t)) \frac{\eta(X_{t^-})}{\eta(Y_{t^-})}\dd A_t + \frac{\eta(f(Y_{t^-},A_t))}{\eta(Y_{t^-})} \dd H_t + 
\frac{\eta(f(Y_{t^-},A_t))}{\eta(Y_{t^-})} \dd M_t - \eta(f(Y_{t^-},A_t)) \dd A_t \\
& \leq  \dd \tilde M_t + \dd H_t,
\end{aligned}
\end{equation*}
where $\tilde{M}_t \define \int_{(0,t]}\frac{\eta(f(Y_{t^-},A_t))}{\eta(Y_{t^-})} \dd M_t$
is a local martingale starting in $0$.  Note that $\Delta M_t \geq 0$ implies $\Delta \tilde{M}_t \geq 0$ for all $t\geq 0$. Due to $f\geq 0$, the family of processes ($f(Y,A)$, the process which is constant $0$, $H$, $\tilde{M}$) satisfy \nameref{def:sup} e.g. for $\tilde{\eta}(x) \equiv x$. 

For predictable $H$ the estimate \eqref{eq:Mark1} follows immediately by applying Lenglart's inequality to $f(Y_t, A_t) \leq \tilde M_t + H_t$.
In the other cases, we apply \cref{lemma:1} to   ($f(Y,A)$, the process which is constant $0$, $H$, $\tilde{M}$) and obtain (using the notation of the lemma)
\begin{equation*}
\PP[\sup_{t\in[0,T]}f(Y_t, A_t)>u \mid \F_0] \leq \frac{1}{u}\E[\tilde{M}_T^{L,u} + H_T^{L,u}\mid \F_0]. 
\end{equation*}
Noting that $\tilde{M}^{L,u}$ is a local martingale, $\tilde{M}_{t}^{L,u} + H_t^{L,u}\geq 0$, we obtain by Fatou's lemma
\begin{equation*}
\begin{aligned}
\PP[\sup_{t\in[0,T]}f(X_t, A_t)>u\mid \F_0 ] 
& \leq \PP[\sup_{t\in[0,T]}f(Y_t, A_t)>u\mid \F_0 ] \\
& \leq \frac{1}{u}\EO[\liminf_{n\to\infty} (\tilde{M}_{T\wedge \tau_n}^{L,u} + H_{T\wedge \tau_n}^{L,u})] \\
& \leq \frac{1}{u} \liminf_{n\to\infty}\EO[\tilde{M}_{T\wedge \tau_n}^{L,u} + H_{T\wedge \tau_n}^{L,u}]  = \frac{1}{u}\EO[H_T^{L,u}] 
\end{aligned}
\end{equation*}
which is \eqref{eq:Mark1}. 
Inequality \eqref{eq:Mark1} implies \eqref{eq:Mark1-r} by the following caluclation. For all $u>c_0$, $w, R>0$ we have:
\begin{equation*}
\begin{aligned}
\PP\bigg[\sup_{t\in[0,T]} X_t > u \,\bigg|\, \F_0\bigg] 
&  \leq \PP\bigg[\sup_{t\in[0,T]} G^{-1}(G(X_t)- R) > G^{-1}(G(u)- R) , A_T \leq R \,\bigg|\, \F_0\bigg]
 + \PP[A_T > R \mid \F_0 ] \\
& \leq \PP\bigg[\sup_{t\in[0,T]} G^{-1}(G(X_t)- A_t) > G^{-1}(G(u)- R) \,\bigg|\, \F_0\bigg]
 + \PP[A_T > R \mid \F_0 ] \\ 
& \leq \begin{cases}
  \frac{\EO[H_T \wedge w ]}{G^{-1}(G(u) - R)} + \PP[H_T\geq w \mid \F_0 ] + \PP[A_T > R\mid \F_0] & \text{if } H \text{ is predictable or }  \Delta M \geq 0 \\[1em]
   \frac{\EO[H_T]}{G^{-1}(G(u) - R)}\wedge 1 + \PP[A_T > R\mid \F_0] & \text{if } \E[H_T] < \infty. 
\end{cases}
\end{aligned}
\end{equation*}

To obtain \eqref{eq:orignal-estimate}, apply \cref{lemma:2} to ($f(Y,A)$, the process which is constant $0$, $H$, $\tilde{M}$):
\begin{equation*}
\begin{aligned}
 \big\| G^{-1}(G(X^*_T) - A_T)\big \|\pF   &\leq  \big\| \sup_{t\in[0,T]} G^{-1}(G(X_t) - A_t)\big \|\pF 
 \leq  \big\| \sup_{t\in[0,T]} G^{-1}(G(Y_t) - A_t)\big \|\pF \\
 & = \big\| \sup_{t\in[0,T]}f(Y_t, A_t) \big\|\pF
 \leq \begin{cases}
  \alpha_1 \| H_T\|_{1,\F_0}  & \text{if } \E[H_T] < \infty, \\
  \alpha_1\alpha_2 \| H_T\|\pF & \text{if } \Delta M \geq 0, \,\, \E[H_T^p] < \infty, \\
\alpha_1\alpha_2 \| H_T\|\pF & \text{if }  H \text{ predictable,} \,\, \E[H_T^p] < \infty.  
\end{cases}
\end{aligned}
\end{equation*}
This proves the assertion for continuous $A$.

\smallskip

Now we prove the assertion for non-continuous (but predictable) $A$. We may assume w.l.o.g. that the continuous part of $A$ is strictly increasing and $A^c_\infty = \infty$, for details see \cref{rmk:timeshift}. Let  $(\tilde{X}, \tilde{A}, \tilde{H}, \tilde{M})$ the family of processes and $(\tilde{\F}_t)_{t\geq 0}$ the filtration we obtain by applying \cref{lemma:smoothenJumps} to $(X, A, H, M)$. Fix some arbitrary $T>0$. As
$A_T$ is a $(\tilde{\F}_t)_{t\geq 0}$ stopping time, $(t\wedge A_T)_{t\geq 0}$ is a continuous adapted process, so we may apply the first part of this proof to $((\tilde{X}_{t\wedge A_T})_{t\geq 0}, (t\wedge A_T)_{t\geq 0}, (\tilde{H}_{t\wedge A_T})_{t\geq 0}, (\tilde{M}_{t\wedge A_T})_{t\geq 0})$, to obtain for all $\tilde{T}>0$
\begin{equation*}
\begin{aligned}
& \PP\bigg[\sup_{t\in[0,\tilde{T}\wedge A_T]} G^{-1}(G(\tilde{X}_t) - t) > u \,\bigg|\, \F_0\bigg]  \\
& \leq  \begin{cases}
  \frac{1}{u}\EO[\tilde{H}_{\tilde{T}\wedge A_T} \wedge (\lambda u) ] + \PP[\tilde{H}_{\tilde{T}\wedge A_T} \geq \lambda u \mid \F_0 ] & \text{if } H \text{ is predictable or }  \Delta M \geq 0 \\
   \frac{1}{u}\EO[\tilde{H}_{\tilde{T}\wedge A_T}]\wedge u & \text{if } \E[H_T] < \infty. 
\end{cases}
\end{aligned}
\end{equation*} 
and
\begin{equation*}
\begin{aligned}
\big\| \sup_{t\in[0,A_T \wedge \tilde{T}]} G^{-1}(G(\tilde{X}_t) - t)\big \|\pF  
  \leq \begin{cases}
  \alpha_1 \| \tilde{H}_{A_T \wedge \tilde{T}}\|_{1,\F_0}  & \text{if } \E[H_T] < \infty, \\
  \alpha_1\alpha_2 \|\tilde{H}_{A_T \wedge \tilde{T}}\|\pF & \text{if } \Delta M \geq 0, \,\, \E[H_T^p] < \infty, \\
\alpha_1\alpha_2 \|\tilde{H}_{A_T \wedge \tilde{T}}\|\pF & \text{if }  H \text{ predictable,} \,\, \E[H_T^p] < \infty.  
\end{cases}
\end{aligned}
\end{equation*}
Here we used that \cref{lemma:smoothenJumps} ensures that $\tilde{H}$ is predictable if $H$ is predictable, and $\Delta \tilde{M} \geq 0$ if $\Delta M \geq 0$.

By letting $\tilde{T} \to \infty$, monotone convergence, $\tilde{X}_{A_t} = X_t$ and $\tilde{H}_{A_t} = H_t$ for all $t\geq 0$, this implies the assertion for non-continuous $A$. \\

\smallskip

\noindent \textbf{Proof of \ref{item:convexSup}:} It suffices to prove  \cref{thm:stochBihariConvex} \ref{item:convexSup} for continuous $A$, as the assertion can be extended by \cref{lemma:smoothenJumps} to predictable $A$ as in the proof of \cref{thm:stochBihariConvex} \ref{item:convexNoSup}. We assume w.l.o.g. that $H\geq c_0 + \varepsilon$  and $X\geq c_0  + \varepsilon$ on $\Omega$ for some constant $\varepsilon>0$. \\

We first sketch the idea of the proof, then we provide the details. Denote  by $(Y_t)_{t\geq 0}$ the right-hand side of inequality \eqref{eq:GronwallAssumption} of \nameref{def:sup} and let $f$ be as in the proof of b). Under \nameref{def:sup}, the inequality $\dd f(Y_t, A_t)\leq \dd \tilde{M}_t + \dd H_t$ does not hold because terms of It\^o's formula fail to cancel out. 
However, for any $p\in(0,1)$, it can be shown that 
$((X^*_t)^p)_{t\in[0,T]}$ satisfies \nameref{def:nosup} for $\eta_p(x) \define \frac{p}{1-p}\eta(x^{1/p})x^{1-1/p}$. Thus, similarly as in the proof of \ref{item:convexNoSup}, an application of It\^o's formula implies an estimate for $\tilde{G}_p^{-1}( \tilde{G}_p((X_T^*)^p) - A_T)$ where $\tilde{G}_p(x) \define \int_c^x  \eta_p(s)^{-1} \dd s$ for some $c>c_0$. \\

We first show that $((X^*_t)^p)_{t\in[0,T]}$ satisfies \nameref{def:nosup} for $\eta_p(x) \define \frac{p}{1-p}\eta(x^{1/p})x^{1-1/p}$:
We apply \cref{lemma:1} and obtain e.g. for the case $\E[H_T] <\infty$ for all $u>0$,  $t\in[0,T]$
\begin{equation}\label{eq:proofconvex-01}
\one_{\{X^*_t >u\}} u \leq  \int_{(0,t]} \eta(X^*_{s^-}) \one_{\{X^*_{s^-}\leq u\}} \dd A_s + M^{L,u}_t + \EO[H_T]\wedge u.
\end{equation}
Let $\tau$ be a bounded stopping time. The right-hand side of
\eqref{eq:proofconvex-01} is non-negative and $ M^{L,u}$ is a local martingale. Hence, taking the conditional expectation, Fatou's lemma and monotone convergence implies
\begin{equation}
\PO[X^*_{\tau\wedge T} >u ] u \leq  \EO\left[\int_{(0,{\tau\wedge T}]} \eta(X^*_{s^-}) \one_{\{X^*_{s^-}\leq u\}} \dd A_s \right] + \EO[H_T]\wedge u.
\end{equation}
We integrate this inequality w.r.t. $p u^{p-2} \dd u$ and apply the formulas \eqref{eq:formulaZ}. To this end we first compute the single terms:
\begin{equation*}
\begin{aligned}
\int_0^\infty \PO[X^*_{\tau\wedge T} >u ] p u^{p-1} \dd u
 & = \EO\left[\int_0^\infty \one_{\{X^*_{\tau\wedge T} >u \}} p u^{p-1} \dd u \right]  \overset{\eqref{eq:formulaZ}}{=}  \EO[(X^*_{\tau\wedge T})^p] \\
\int_0^\infty \EO\left[\int_{(0,{\tau\wedge T}]} \eta(X^*_{s^-}) \one_{\{X^*_{s^-}\leq u\}} \dd A_s \right]  p u^{p-2} \dd u   
& = 
\EO\left[\int_{(0,{\tau\wedge T}]} \eta(X^*_{s^-}) \int_0^\infty \one_{\{X^*_{s^-}\leq u\}}p u^{p-2} \dd u \dd A_s \right]   \\
& \overset{\eqref{eq:formulaZ}}{=}
\EO\left[\int_{(0,{\tau\wedge T}]} \eta(X^*_{s^-}) \frac{p}{1-p}
(X^*_{s^-})^{p-1} \dd A_s \right] \\
& =   \EO\left[\int_{(0,{\tau\wedge T}]} \eta_p\big( (X^*_{s^-})^p \big) \dd A_s \right]
\end{aligned}
\end{equation*}
Moreover, the calculations of the proof of \cref{lemma:2} imply:
\begin{equation*}
\EO\bigg[\int_0^\infty H^{L,u}_{\tau\wedge T} p u^{p-2} \dd u\bigg]
\leq \EO\bigg[\int_0^\infty H^{L,u}_T p u^{p-2} \dd u\bigg]
  =\begin{cases}
\alpha_1^p \alpha_2^p  \EO[H_T^p] & \text{ if } H \text{ is predictable or } \Delta M \geq 0 \\
\alpha_1^p  \EO[H_T]^p & \text{ if } \E[H_T] < \infty
\end{cases}  =: \bar{H}_{\tau \wedge T}.
\end{equation*}
Combining the calculations implies for any bounded stopping time $\tau$
\begin{equation}
\EO[ (X^*_{\tau\wedge T})^p]  \leq   \EO\bigg[\int_{(0,\tau\wedge T]} \eta_p\big((X^*\sm)^p\big) \dd A\s  + \bar{H}_{\tau\wedge T} \bigg].
\end{equation}
By applying \cref{cor:characterizationOfLenglartDomination} and the general Meyer-Doob decomposition to $\left((X^*_{t \wedge T})^p\right)_{t\geq 0}$ and \\ $\left(\int_{(0,t\wedge T]}\eta_p\big((X^*_{s\wedge T-})^p \big) \dd A\s +\bar{H}_t\right)_{t\geq 0}$ we obtain that there exists a \cadlagg local martingale $(M^{S}_t)_{t\geq 0}$ with $M^{S}_0 =0$ such that 
\begin{equation}\label{eq:proofconvex-02}
 (X^*_{t\wedge T})^p \leq \int_{(0,t\wedge T]} \eta_p\big((X^*_{s\wedge T-})^p \big) \dd A\s
+ \bar{H}_t + M^{S}_t \qquad \forall t\geq 0.
\end{equation}
Hence, $(\left((X^*_{t \wedge T})^p\right)_{t\geq 0}, A, \bar{H}, M^S)$ and $\eta_p$ satisfy \nameref{def:nosup}. 

The rest of the proof is very similar to the proof of assertion b): Denote by $Z_t$ the right-hand side of \eqref{eq:proofconvex-02}. We redefine $f$ now using $\tilde{G}_p \define 
\int_{c^{p}}^{x} \frac{\dd u}{\eta_p(u)}$ instead of $G$:
\begin{equation*}
f\colon(c_0^p,\infty)\times[0,\infty) \mapsto (c_0^p,\infty), \qquad (x,a) \mapsto \tilde{G}_p^{-1}(\tilde{G}_p(x) - a),
\end{equation*}
By the same calculation as in the proof of b) we have by It\^o's formula:
\begin{equation*}
f(Z_0,A_0)  = \bar{H}_0, \qquad 
\dd f(Z_t, A_t) \leq \frac{\eta_p(f(Z_{t^-},A_t))}{\eta_p(Z_{t^-})} \dd \bar{H}_t + 
\frac{\eta_p(f(Z_{t^-},A_t))}{\eta_p(Z_{t^-})} \dd M^S_t.
\end{equation*}
So, using that $f\geq 0$, an application of Fatou's lemma implies
\begin{equation}\label{eq:proofconvex3}
\EO[f(Z_T, A_T)] =  \EO[\tilde{G}_p^{-1}(\tilde{G}_p(Z_T)-A_T)] \leq \EO[\bar{H}_T]. 
\end{equation}
We have \textcolor{blue}{for all $x^{1/p} > c_0$ and $y\in \text{domain}(\tilde{G}^{-1}_p)$}
\begin{equation*}
\tilde G_p(x) = (1-p)G(x^{1/p}), \quad \tilde{G}_p^{-1}(y) = \big(G^{-1}\big(\tfrac{y}{1-p} \big)\big)^p
\end{equation*}
(see \eqref{eq:G_umrechnen} for details). Together, using $(X_T^*)^p\leq Z_T$, we have
\begin{equation*}
\begin{aligned}
\EO[G^{-1}(G(X^*_T)-\beta A_T)^p] 
&= \EO[\tilde{G}_p^{-1}(\tilde{G}_p( (X^*_T)^p)-A_T)] \\
&\leq \EO[\tilde{G}_p^{-1}(\tilde{G}_p(Z_T)-A_T)] \\
&\leq \EO[\bar{H}_T]
\end{aligned}
\end{equation*}
which is assertion b).
\end{proof}

\subsection{Proof of sharpness (\cref{lemma:pathdependentSDE}, \cref{thm:sharpnessBeta},  \cref{thm:sharp-tail})}

\begin{proof}[Proof of \cref{lemma:pathdependentSDE}]
We first prove that \eqref{eq:path-dependent-SDE-sharpness} has a non-negative solution exploiting that either $b$ or $\sigma$ are always $0$. More precisely, on the time intervals $(0,\varepsilon\delta)$, $(\delta, \delta +\varepsilon\delta)$, $(2\delta, 2\delta +\varepsilon\delta)$, ..., the coefficient $b$ is identically $0$. On  $(\varepsilon\delta, \delta)$, $(\delta+\varepsilon\delta, 2\delta)$, $(2\delta+\varepsilon\delta, 3\delta)$,  \dots   \, the coefficient $\sigma$ is $0$.

To simplify the notation we define the processes $(B^i_t)_{t\in[0,(i+\varepsilon) \delta)}$ by
\begin{equation*}
B^i_t  \define \begin{cases}
0 & \quad \forall t\in[0, i \delta) \\
\int_{i\delta}^{t} l(u-i\delta) \dd W_u & \quad \forall t\in[i \delta, i\delta + \delta \varepsilon)
\end{cases}
\quad =  \quad \begin{cases}
0 & \quad \forall t\in[0, i \delta) \\
\tilde{B}^i_{h(t-i\delta)} & \quad \forall t\in[i \delta, i\delta + \delta \varepsilon).
\end{cases}
\end{equation*}
where $\{(\tilde{B}^i\ti)_{t\geq 0}, i\in\N_0\}$ is a family of independent Brownian motions and $h(t) \define \int_0^t l^2(u) \dd u$ for all $t\in[0,\varepsilon\delta)$.  Note that $h(0)=0$, $h$ is increasing, continuous and $h(\varepsilon\delta) = +\infty$. Therefore, $(\tilde{B}^i_{h(t)})_{t\in[0,\varepsilon\delta)} $ can be seen as a sped-up Brownian motion.

\smallskip

\noindent \textbf{Step 1: Construction of $X_t$ for $t\in[0,\delta]$} \\ We define
\begin{equation*} 
\tau_0 \define \inf\{t\in[0,\varepsilon \delta) \mid 1 + B^0_{t}=0 \} \quad \text{ setting here } \inf\emptyset \define 0. 
\end{equation*}
As we assumed that the underlying filtered probability space satisfies the usual conditions, i.e. is in particular complete, $\tau_0$ is indeed a stopping time.
As on $[0,\delta \varepsilon]$ we have $g_{\varepsilon,\delta} = 0$, i.e. $b=0$, the path-dependent SDE \eqref{eq:path-dependent-SDE-sharpness} corresponds for all $t\in[0,\delta\varepsilon]$ to:
\begin{equation*}
X_t =  1 + \int_0^t b(s,X) \dd s + \int_0^t \sigma(s,X)\dd W_s 
= 1 +  \int_0^t \one_{\{X_s >0 \}} l(s) \dd W_s.
\end{equation*}
Hence, $X_t = 1 + B^0_{t\wedge \tau_0}$ satisfies \eqref{eq:path-dependent-SDE-sharpness}  for $t\in[0,\delta\varepsilon]$.
Note, that $\tau_0 < \varepsilon \delta$ on $\Omega$. By construction, we have $X_{\varepsilon\delta} = X_{\tau_0} = 0$ $\PP$-almost surely. For all $t\in]\delta\varepsilon, \delta]$ we have (due to $\sigma$ being $0$ here):
\begin{equation*}
X_t = X_{\varepsilon\delta} + \int_{\varepsilon\delta}^t b(s,X) \dd s + \int_{\varepsilon\delta} ^t \sigma(s,X)\dd W_s = \int_{\delta \varepsilon}^t X^*_s \dd s  = (t-\delta\varepsilon) X_{\tau_0}^*.
\end{equation*}
In particular, we have $X_{\delta} = \gamma X_{\tau_0}^*$ $\PP$-almost surely where $\gamma\define (1-\varepsilon)\delta<1$. So, we have constructed the following non-negative solution (upto a null set) on the time interval $[0,\delta]$:
\begin{equation*}
X_t = 
\begin{cases}
 1 + B^0_{t\wedge \tau_0} \quad & \forall t\in[0,\varepsilon\delta] \\
(t-\delta\varepsilon) X_{\tau_0}^* \quad & \forall t\in(\varepsilon\delta, \delta].
\end{cases}
\end{equation*}
Due to $\tau_0<\varepsilon \delta$ and $\int_0^t l^2(u) \dd u < \infty$ for all $t<\varepsilon\delta$, we have
\begin{equation*}
\int_0^{\delta} |b(s,X)| \dd s + \int_0^{\delta} |\sigma(s,X)|^2 \dd s 
= \int_{\varepsilon\delta}^{\delta} |b(s,X)| \dd s + \int_0^{\tau_0} l(u)^2 \dd s 
< \infty \qquad \PP\text{-a.s.}
\end{equation*}

\smallskip

\noindent \textbf{Step 2: Construction of $X_t$ for $t\in(k\delta,(k+1)\delta]$} \\ Assume we have constructed $(X_t)_{t\in[0,k\delta]}$ for some $k\in\N$. Now we construct a solution on $(k\delta, (k+1)\delta]$. Similarly as before, we set
\begin{equation*} 
\tau_k \define \inf\{t\in[k\delta,k\delta + \varepsilon \delta) \mid X_{k\delta} + B^k_{t} =0 \} \quad \text{ setting here } \inf\emptyset \define 0.
\end{equation*}
Due to completeness of the underlying filtered probability space $\tau_k$ is a stopping time. 
By definition, we have $g_{\varepsilon,\delta} = 0$ i.e. $b=0$ on $(k\delta, k\delta + \varepsilon\delta)$, and hence \eqref{eq:path-dependent-SDE-sharpness} corresponds for $t\in[k\delta, k\delta + \varepsilon\delta]$ to
\begin{equation*}
X_t =  X_{k\delta} + \int_{k\delta}^t b(s,X) \dd s + \int_{k\delta}^t \sigma(s,X)\dd W_s 
=  X_{k\delta} + \int_{k\delta}^t  \one_{\{X_s >0 \}} l(s-k\delta)  \dd W_s 
\end{equation*}
and therefore $X_t =  X_{k\delta} + B^k_{t\wedge \tau_k}$ is a solution of \eqref{eq:path-dependent-SDE-sharpness} for $t\in(k\delta, k\delta + \varepsilon\delta]$.
As before we have $\tau_k < k\delta + \varepsilon \delta$ on $\Omega$ and $X_{(k+\varepsilon) \delta} = X_{\tau_k} = 0$  $\PP$-almost surely. For all $t\in(k\delta + \delta\varepsilon, (k+1)\delta]$ we have (due to $\sigma$ being $0$ here):
\begin{equation*}
X_t = X_{k\delta + \varepsilon\delta} + \int_{k\delta + \varepsilon\delta}^t b(s,X) \dd s + \int_{k\delta + \varepsilon\delta} ^t \sigma(s,X)\dd W_s = \int_{k\delta + \delta \varepsilon}^t X^*_s \dd s  = (t-k\delta - \delta\varepsilon) X_{\tau_k}^* \leq \gamma X_{\tau_k}^*
\end{equation*}
and $X_{(k+1)\delta} = \gamma X_{\tau_k}^*$. Hence, 
\begin{equation*}
X_t = 
\begin{cases}
 X_{k\delta} + B^k_{t\wedge \tau_k} \quad & \forall t\in[k\delta,k\delta + \varepsilon\delta] \\
(t-(k\delta +\delta\varepsilon)) X_{\tau_k}^* \quad & \forall t\in(k\delta + \varepsilon\delta, (k+1)\delta]
\end{cases}
\end{equation*}
is a non-negative solution of \eqref{eq:path-dependent-SDE-sharpness}. By the same calculation as in step 1 it satisfies  \eqref{eq:lemma-path-dependent-SDE-1}.

\bigskip

\noindent \textbf{Step 3: Proof of $\E[(X_{k\delta + \varepsilon \delta}^*)^p] = \E[\Xvdkn]= (1-p)^{-1}\big(1 + p(1-p)^{-1} \gamma\big)^k$ for all  $p\in(0,1), \,k\in\N$}  \\
We prove the equality by induction over $k$. To this end, note that if $\Xvdkn >  \Xvdkln$, then the supremum of $X^p$ on $[0,\tau_k]$ must occur on $[k\delta, \tau_k]$, since by construction
\begin{equation*}
X_t 
= \begin{cases}
0  & \text{ if } t\in[\tau_{k-1}, (k-1)\delta +\varepsilon\delta] \\
\int^t_{(k-1+\varepsilon)\delta}X_s^* \dd s\leq \gamma X_{\tau_{k-1}}^*
& \text{ if } t\in[(k-1)\delta +\varepsilon \delta, k\delta] \\
X_{k\delta} +  B^k_{t\wedge \tau_k}
=   \gamma X_{\tau_{k-1}}^* +  B^k_{t\wedge\tau_k}  & \text{ if } t\in[k\delta,\tau_k]
\end{cases}
\end{equation*}
and $\gamma=(1-\varepsilon)\delta<1$. Define for some fixed $c>0$:
\begin{equation*}
\begin{aligned}
 \sigma & \define \inf\{t\in [k\delta, k\delta + \varepsilon \delta) \mid (B^k\ti + X_{k\delta})^p = 
 (X^*_{\tau_{k-1}})^p+ c \} \\
 & = \inf\{t\in [k\delta, k\delta + \varepsilon \delta) \mid B^k\ti =  ((X^*_{\tau_{k-1}})^p+ c)^{1/p}  -  X_{k\delta} \},\\
 \end{aligned}
 \end{equation*}
setting $\inf \emptyset = k\delta + \varepsilon\delta$. Due to $X_{k\delta} = \gamma X_{\tau_{k-1}}^*$ we have  $(\Xvdkln+ c)^{1/p} - X_{k\delta}  \geq 0$. By the definition of $\sigma$ we have
\begin{equation*}
\{ \Xvdkn >  \Xvdkln +c\}  = \{\sigma < \tau_k \}.
\end{equation*}
By the independence of the Brownian motions $\tilde{B}_k, k\in\N_0$ we have:
\begin{equation*}
 \PP[\Xvdkn \geq  \Xvdkln +c \mid \F_{k\delta} ]  = \PP[\sigma < \tau_k \mid  \F_{k\delta} ] = 
 \frac{X^{\varepsilon,\delta}_{k\delta}}{( \Xvdkln+ c)^{1/p}},
 \end{equation*}
as $\PP[\sigma < \tau_k \mid \hat{\F}_{k-1}]$ is the conditional probability that the Brownian motion $B^k$ hits  $((X_{\tau_{k-1}})^{*,p}+ c)^{1/p} - X_{k\delta}$ before $-X_{k\delta}$. Applying the previous equation gives:
\begin{equation*}
\begin{aligned}
\E[ \Xvdkn] 
&= 
\E\bigg[ \Xvdkln +  \int_0^{(\Xvdkn- \Xvdkln)\vee 0} \dd u \bigg] \\
&= 
\E\bigg[ \Xvdkln  + \int_0^\infty \PP[ \Xvdkn \geq \Xvdkln + u \mid \F_{k-1}] \dd u\bigg] \\
& = 
\E[\Xvdkln]  + \E\bigg[\int_0^\infty  \frac{X_{k\delta} }{(\Xvdkln + u)^{1/p}} \dd u \bigg]\\
& = 
\E[\Xvdkln]  + \frac{p}{1-p}  \E[X_{k\delta} ((X_{\tau_{k-1}})^*)^{-1+p}].
\end{aligned}
\end{equation*}
Applying that $X_{k\delta} = \gamma X^*_{\tau_{k-1}}$ implies 
\begin{equation}\label{eq:proofsharpness}
\E[ \Xvdkn] = (1+ p(1-p)^{-1}\gamma) \E[\Xvdkln].
\end{equation}
Noting that $\E[(X_{\tau_0})^{*,p}] = \E[\sup_{s\in[0,\tau_0]}(1+B_s^0)^p] = \frac{1}{1-p}$ and iterating \eqref{eq:proofsharpness} yields:
\begin{equation*}
\E[ \Xvdkn] = (1-p)^{-1} \big(1 + p(1-p)^{-1}\gamma\big)^k.
\end{equation*}
\end{proof}

\begin{proof}[Proof of \cref{thm:sharpnessBeta}]
For $\varepsilon, \delta \in (0,1)$ let $(X^{\varepsilon,\delta}_t)_{t\geq 0}$ denote the process from \cref{lemma:pathdependentSDE}. By \cref{lemma:pathdependentSDE} the process $(X^{\varepsilon,\delta})_{t\geq 0}$ is non-negative, adapted and continuous. Moreover, it satisfies
\begin{equation*}
X^{\varepsilon,\delta}_t \leq \int_0^t (X^{\varepsilon,\delta})_s^* \dd s + M^{\varepsilon,\delta}_t + 1 \qquad \forall t\geq 0
\end{equation*}
for $M^{\varepsilon,\delta}_t \define \int_0^t \sigma(s, X^{\varepsilon,\delta}) \dd W_s$, which is a continuous local martingale starting in $0$. Moreover, by \cref{lemma:pathdependentSDE} we have
\begin{equation}\label{eq:proof-beta-sharpness-2}
\E[ ((X^{\varepsilon, \delta})^*_{k\delta + \varepsilon\delta})^p] = \beta \big(1 + p\beta\gamma\big)^k.
\end{equation}
Due to $X^{\varepsilon,\delta}$, $M^{\varepsilon,\delta}$ and $H=1$ satisfying the assumptions of \cref{thm:sharpnessBeta}, the inequality 
\begin{equation*}
\|(X^{\varepsilon,\delta}_t)^*\|_p \leq \tilde{\alpha} H \exp(\tilde{\beta} t)
\end{equation*}
holds true by assumption \eqref{eq:beta-sharpness-assumption}. Rearranging the inequality  and choosing $t=(k+1)\delta$ implies for all $k\in\N$:
\begin{equation*}
\tilde{\beta} \geq ((k+1)\delta)^{-1} \log(\|(X^{\varepsilon, \delta}_{(k+1)\delta})^*\|_p) - ((k+1)\delta)^{-1} \log(\tilde{\alpha}).
\end{equation*}
By inserting \eqref{eq:proof-beta-sharpness-2}, we have for all $k\in\N$:
\begin{equation*}
\tilde{\beta} \geq  \frac{k}{k+1}p^{-1} \log\big\{\big( 1+ p\beta\gamma\big)^{1/\delta} \big\} + ((k+1)\delta)^{-1}\log(\beta^{1/p}) - ((k+1)\delta)^{-1} \log(\tilde{\alpha}),
\end{equation*}
i.e. taking the limits $k\nearrow\infty$, $\varepsilon \searrow 0$, $\delta \searrow 0$ gives
\begin{equation*}
\tilde{\beta} 
\geq  
\lim_{\delta \searrow 0}\lim_{\varepsilon \searrow 0}\frac{1}{p\delta}\log\big( 1 + \beta p\delta(1-\varepsilon)\big) 
=
\lim_{\delta \searrow 0} \frac{1}{p\delta}\log\big( 1 + \beta (p\delta) \big) 
=
\frac{\partial}{\partial x} \log(1 + \beta x) \bigg|_{x=0}  = \beta
\end{equation*}
which implies the assertion.
\end{proof}

\begin{proof}[Proof of \cref{thm:sharp-tail}]
Fix some $\varepsilon, \delta \in (0,1)$ and $k\in\N$ such that $T = k\delta + \varepsilon\delta$. Let $(X_t)_{t\geq 0}$ be the process from \cref{lemma:pathdependentSDE}, which satisfies \eqref{eq:beta-sharpness-assumption0-1} 
for $M_t \define \int_0^t \sigma(s,X)\dd W_s$. We have 
\begin{equation*}
\sup_{p\in(0,1)} \big((1-p) \E[(X^*_T)^p]\big)
= \sup_{p\in(0,1)} \left(1+\frac{p}{1-p} (1-\varepsilon)\delta\right)^{k} = \infty.
\end{equation*}
Assume that there exists a $0<C<\infty$ such that $\PP[X^*_T>u] \leq C\frac{1}{u}$ for all $u>0$. Then e.g. by \eqref{eq:formulaZ} we have
\begin{equation*}
\E[(X^*_T)^p] = 1^p +  p \int_1^\infty \PP[X_t>u] u^{p-1} \dd u \leq  1 + p C\int_1^\infty u^{p-2} \dd u
 = 1 + C \frac{p}{1-p}  
\end{equation*}
which implies 
\begin{equation*}
\sup_{p\in(0,1)} \left((1-p) \E[(X^*_T)^p]\right) \leq \sup_{p\in(0,1)}( (1-p) + Cp) \leq 1+C <\infty,
\end{equation*}
which is a contradiction. This proves $\sup_{u>0} \left(u\PP[X^*_T>u]\right)=\infty$.
\end{proof}

\section{Special case: Sharp stochastic Gronwall inequalities}\label{sec:Gronwall}
In this section we summarize the results in the literature for the linear case $\eta(x) =x$ and compare them to the inequalities of this paper. \\

Von Renesse and Scheutzow \cite[Lemma 5.4]{RenesseScheutzow} developed a stochastic Gronwall inequality for continuous martingales to study stochastic functional differential equations. This result was further generalized by Mehri and Scheutzow \cite[Theorem 2.1]{MehriScheutzow}, who  applied Lenglart's domination inequality in the proof.

\begin{theorem}[Mehri and Scheutzow: A stochastic Gronwall inequality for $\Asup$]
Let \nameref{def:sup} hold and assume that $A$ is deterministic and $\eta(x) \equiv x$. Then, the following estimates hold for $p\in(0,1)$ and $T>0$
\begin{equation*}
\|X_T^*\|\pF \leq
\begin{cases}
 p^{-1/p} c_p \|H_T\|\pF \e^{p^{-1}c_p A_T} & 
 \text{ if }  \E[H_T^p] < \infty \text{ and }  H \text{ is predictable,} \\
p^{-1/p} (c_p^p + 1)^{1/p} \|H_T\|\pF \e^{p^{-1}(c_p^p +1)^{1/p} A_T} & 
\text{ if } \E[H_T^p] < \infty \text{ and } \Delta M \geq 0, \\
p^{-1/p} c_p \|H_T\|_{1,\F_0} \e^{p^{-1}c_p A_T} & 
\text{ if }  \E[H_T] < \infty,
\end{cases}
\end{equation*}
where $c_p = \alpha_1 \alpha_2 =  (1-p)^{-1/p} p^{-1}$.
\end{theorem}

The following is a corollary of \cref{thm:stochBihariConvex} \ref{item:convexSup}, \cref{thm:sharpnessBeta} and \cref{thm:sharpnessAlpha}. It slightly sharpens the result above and extends it to predictable integrators $A$:
\begin{corollary}[A sharp stochastic Gronwall inequality for $\Asup$]\label{cor:GronwallSup}
Let \nameref{def:sup} hold and assume $\eta(x)\equiv x$ and $p\in(0,1)$.  
Then, the following estimates hold for all $T>0$.
\begin{equation*}
\|\e^{-\beta A_T}X^*_T\|\pF  \leq
\begin{cases}
\alpha_1  \alpha_1 \|H_T\|_{p,\F_0} 
& \text{if }  \E[H_T^p] < \infty \text{ and }  H \text{predictable,} \\
\alpha_1 \alpha_2 \|H_T\|_{p,\F_0} 
&  \text{if } \E[H_T^p] < \infty \text{ and } \Delta M \geq 0, \\
\alpha_1 \|H_T\|_{1,\F_0}   
& \text{if } \E[H_T] < \infty. \\
\end{cases}
\end{equation*}
The constants $\alpha_1 = (1-p)^{-1/p}, \alpha_1 \alpha_2 =(1-p)^{-1/p} p^{-1}$ and $\beta = (1-p)^{-1}$ are sharp. If $\|e^{\beta A_T}\|_{qp/(p-q),\F_0}$ is integrable, we have for $0<q<p<1$ and all $T\geq 0$
\begin{equation*}
\|X^*_T\|\qF  \leq
\begin{cases}
\alpha_1  \alpha_1 \|H_T\|_{p,\F_0} \|e^{\beta A_T}\|_{qp/(p-q),\F_0}
& \text{if }  \E[H_T^p] < \infty \text{ and }  H \text{predictable,}\\
\alpha_1 \alpha_2 \|H_T\|_{p,\F_0} \|e^{\beta A_T}\|_{qp/(p-q),\F_0}
&  \text{if } \E[H_T^p] < \infty \text{ and } \Delta M \geq 0, \\
\alpha_1 \|H_T\|_{1,\F_0}\|e^{\beta A_T}\|_{qp/(p-q),\F_0}
& \text{if } \E[H_T] < \infty.
\end{cases}
\end{equation*}
\end{corollary}

In Assumption $\Asup$ it is assumed that $A$ is predictable. For an example that this assumption cannot be dropped see \cref{example:notpredictable}.

A stochastic Gronwall lemma (in a setting nearly identical to $\Anosup$ with $\eta(x) \equiv x$) was proven for continuous martingales $M$ by Scheutzow  \cite[Theorem 4]{Scheutzow}. This result was extended by Xie and Zhang \cite[Lemma 3.7]{XieZhang} to c\`adl\`ag martingales:

\begin{theorem}[Xie and Zhang: A stochastic Gronwall inequality for $\Anosup$]\label{thm:XieZhang}
Let Assumption $\Anosup$ hold. Furthermore, assume that $\eta(x) \equiv x$ and that $A$ is continuous.
Then, for any $0<q<\tilde{p}<1$ and $t\geq 0$, we have:
\begin{equation*}
\| X^*\ti\|_q \leq \bigg(\frac{\tilde{p}}{\tilde{p}-q}\bigg)^{1/q} \|H\ti\|_{1} \, \|\e^{A\ti}\|_{\tilde{p}/(1-\tilde{p})}.
\end{equation*}
\end{theorem}
The following corollary of \cref{thm:stochBihariConvex} and \cref{thm:sharpnessAlpha} slightly extends and marginally sharpens \cite[Lemma 3.7]{XieZhang}. Recall that \cref{thm:stochBihariConvex} \ref{item:convexNoSup} was proven by further developing the proof idea of  \cite[Lemma 3.7]{XieZhang}.

\begin{corollary}[A sharp stochastic Gronwall inequality for $\Anosup$]\label{cor:GronwallNoSup}

Let Assumption $\Anosup$ (see \cref{def:nosup}) hold and assume $\eta(x)\equiv x$ and $p\in(0,1)$. Then, the following estimates hold.

\begin{enumerate}
\item ($L^p$ estimates, $p\in(0,1)$) 
 For all $T>0$ we have:
\begin{equation*}
\|\e^{-A\ti}X^*_T\|\pF  \leq
\begin{cases}
\alpha_1  \alpha_1 \|H_T\|_{p,\F_0} 
& \text{if }  \E[H_T^p] < \infty \text{ and }  H\text{ predictable,} \\
\alpha_1 \alpha_2 \|H_T\|_{p,\F_0} 
&  \text{if } \E[H_T^p] < \infty \text{ and } \Delta M \geq 0, \\
\alpha_1 \|H_T\|_{1,\F_0}   
& \text{if } \E[H_T] < \infty. 
\end{cases}
\end{equation*}
The constants $\alpha_1 = (1-p)^{-1/p}$ and $\alpha_1 \alpha_2 =(1-p)^{-1/p} p^{-1}$ are sharp. If $\|e^{\beta A\ti}\|_{qp/(p-q),\F_0}$ is integrable, we have for $0<q<p<1$
\begin{equation*}
\|X^*_T\|\qF  \leq
\begin{cases}
\alpha_1  \alpha_1 \|H_T\|_{p,\F_0} \|e^{A_T}\|_{qp/(p-q),\F_0}
& \text{if }  \E[H_T^p] < \infty \text{ and }  H\text{ predictable,} \\
\alpha_1 \alpha_2 \|H_T\|_{p,\F_0} \|e^{A_T}\|_{qp/(p-q),\F_0}
&  \text{if } \E[H_T^p] < \infty \text{ and } \Delta M \geq 0, \\
\alpha_1 \|H_T\|_{1,\F_0}\|e^{ A_T}\|_{qp/(p-q),\F_0}
& \text{if } \E[H_T] < \infty. \\
\end{cases}
\end{equation*}

\item ($L^{1,w}$ estimates)
We have for all $T>0, u>0, w>0$ and $R>0$
\begin{equation}\label{eq:Mark2}
\begin{aligned}
\PP[\e^{-A_T} X^*_T > u \mid \F_0 ] 
& \leq \begin{cases}
 \frac{1}{u}\EO[H_T \wedge (\lambda u) ] + \PP[H_T\geq \lambda u\mid \F_0] & \text{if } H \text{ is predictable,} \\
  \frac{1}{u}\EO[H_T \wedge (\lambda u) ] + \PP[H_T\geq \lambda u \mid \F_0] & \text{if } \Delta M \geq 0, \\
   \frac{1}{u}\EO[H_T]\wedge u & \text{if } \E[H_T] < \infty,
\end{cases}
\end{aligned}
\end{equation}
and 
\begin{equation}\label{eq:Mark2-r}
\begin{aligned}
\PP[X^*_T > u \mid \F_0] 
& \leq \begin{cases}
 \frac{\e^{R}}{u}\EO[H_T \wedge w ] + \PP[H_T\geq w\mid \F_0] 
+ \PP[A_T > R\mid \F_0] 
& \text{if } H \text{ is predictable,} \\
  \frac{\e^{R}}{u}\EO[H_T \wedge w ] + \PP[H_T\geq w \mid \F_0] + \PP[A_T> R\mid \F_0]  & \text{if } \Delta M \geq 0, \\
   \left(\frac{\e^{R}}{u}\EO[H_T]\right)\wedge 1 + \PP[A_T > R\mid\F_0]  & \text{if } \E[H_T] < \infty.
\end{cases}
\end{aligned}
\end{equation}
\end{enumerate}
\end{corollary}

\begin{remark}
\cref{thm:sharp-tail} shows that under $\Asup$ the $L^{1,w}$ norm may be infinite, hence the estimates \cref{eq:Mark2} and \cref{eq:Mark2-r} do not hold under the weaker assumption $\Asup$.
\end{remark}

\begin{remark}[Comparison of constants]
Choose any $0 < q < p < 1$ and set $\tilde{p} \define qp(qp+p-q)^{-1}$. Then, $q<\tilde{p}<1$ holds and \cref{thm:XieZhang} implies:
\begin{equation*}
\| X^*\ti\|_q \leq \bigg(\frac{p}{q}\frac{1}{1-p}\bigg)^{1/q}  \|H\ti\|_{1} \|\e^{A\ti} \|_{qp/(p-q)} .
\end{equation*}
Noting that, due to $0<q<p<1$ we have
\begin{equation*}
\bigg(\frac{p}{q}\frac{1}{1-p}\bigg)^{1/q}  \geq \bigg(\frac{1}{1-p}\bigg)^{1/q} \geq \bigg(\frac{1}{1-p}\bigg)^{1/p},
\end{equation*}
the constant in \cref{cor:GronwallNoSup} is slightly sharper than that in \cref{thm:XieZhang}.
However, for deterministic $A$, \cref{thm:XieZhang} yields the sharp constant: The choice of the norm $\|\e^{A\ti}\|_{\tilde{p}/(1-\tilde{p})}=\e^{A\ti}$ has no effect, so we may take the limit $\tilde{p} \to 1$ implying the (optimal) constant  $(1-q)^{-1/q}$.
\end{remark}

\begin{remark}
Under assumption $\Anosup$, random $A$, deterministic $H$ and $p\in(0,1)$ there exists no finite constant $\tilde{c}$ such that 
\begin{equation*}
\|X^*_T\|_p \leq \tilde{c} \|H \e^{A_T}\|_{p}
\end{equation*}
holds. For a counterexample, see \cite[Example 6.1]{GeissConcave}.
\end{remark}



\section{Application: Path-dependent SDEs}
\label{sec:applications}

The convex stochastic Bihari-LaSalle inequalities are, like the stochastic Gronwall inequalities, useful tools to study SDEs. We provide in the following two applications of Section \ref{sec:MainResults} to path-dependent SDEs driven by Brownian motion. Alternatively also more general  (e.g. Levy-driven) path-dependent SDEs could be studied by the same approach.

For solutions of non-path-dependent SDEs, exponential integrability bounds are known, see e.g. Cox, Hutzenthaler and Jentzen \cite{CoxHutzenthalerJentzen}, Hudde, Hutzenthaler and Mazzonetto \cite{Hudde} and the references therein. As explained in the introduction, these proofs do not extend to the case of path-dependent SDEs, because terms fail to cancel out in the path-dependent case. Using the stochastic Bihari-LaSalle inequality for $\eta(x) = x(\log(x) +c)$ we provide a similar result for path-dependent SDEs.

The second application is connected to tail estimates of path-dependent SDEs: We obtain as a corollary of \cref{lemma:pathdependentSDE}  (which is the key lemma to show the  sharpness of the constant $\beta$) that path-dependent SDEs may have tails of a different order than than non-path-dependent SDEs.

\bigskip

Assume an underlying filtered probability space $(\Omega, \F,\PP,(\F_t)_{t\geq 0})$ satisfying the usual conditions.  Let $B$ be an $m$-dimensional Brownian motion and let $| \cdot|_F$ denote the Frobenius norm on $\R^{d\times m}$. We study the path-dependent SDE with random coefficients driven by the Brownian motion $B$:
\begin{equation}\label{eq:SDE}
\begin{cases}
d X_t & = f(t,X) \dd t + g(t, X) \dd B\ti \\
X_t & = z_t, \qquad t\in[-r, 0],
\end{cases}
\end{equation}
where $r>0$ is some constant and the initial condition $(z_t)_{t\in[-r,0]}$  has continuous paths and is $\F_0$ measurable. Denote by $\pred$ the predictable $\sigma$-field on $[0,\infty)\times\Omega$. Let $\B(\cadlag))$ denote the Borel $\sigma$-field on the continuous functions $\cadlag$ induced by convergence in the uniform norm on compacts sets. Assume that the coefficients
\begin{equation*}
\begin{aligned}
& f:([0,\infty)\times  \Omega \times \cadlag, \pred \otimes \B(\cadlag)) \to (\R^d,\B(\R^d)), \\
& g:( [0,\infty) \times \Omega \times\cadlag, \pred \otimes \B(\cadlag)) \to (\R^{d\times m},\B(\R^{d\times m})) \\
\end{aligned}
\end{equation*}
are measurable mappings. For every $t\in[0,\infty)$, $\omega\in\Omega$, $\xi\in U$ assume that $f(t, \omega, x)$ and $g(t, \omega, x)$ only depend on the  path segment $x(s),s\in[-r, t]$. We denote by $f(t,x)$ and $g(t,x)$ the corresponding random variables.

\subsection{Exponential moment estimates}
For $U=(U(s,y))_{s\in[0,T], y\in\R^d}\in C^{1,2}([0,T]\times \R^d, \R)$
we define for all $x\in\cadlag$ and all $t\geq 0$
\begin{equation*}
\begin{aligned}
(\mathcal{G}_{f,g} U)(t,x) & \define  
\big(\tfrac{\partial}{\partial s} U\big) (t,x(t))
 + \big(D_y U\big) (t,x(t)) f(t,x)\\
 &\qquad  +  \tfrac{1}{2}\text{trace}(g(t,x) g(t,x)^T \text{Hess}_y U(t,x(t))),
\end{aligned}
\end{equation*}
where $D_y \define (\tfrac{\partial}{\partial y_1}, ..., \tfrac{\partial}{\partial y_d})$.

One difference of the following corollary to corresponding results for non-path-dependent SDEs, see \cite[Corollary 2.4]{CoxHutzenthalerJentzen} and \cite[Corollary 3.3]{Hudde}, is, that we assume that $U$ is non-negative.

See also \cite[Theorem 2.1]{HutzenthalerNguyen}, in which moment estimates for path-dependent SDEs are proven using the stochastic Gronwall inequality \cite[Theorem 2.1]{MehriScheutzow}. 

\begin{corollary}[Exponential moment estimates for path-dependent SDEs]\label{cor:exponentialMoments}
Let $X$ be a (strong) solution of the SDE \eqref{eq:SDE} satisfying $\int_0^t |f(s,X)|\dd s + \int_0^t |g(s,X)|^2_F \dd s < \infty \,\, \PP$-a.s. for all $t\in[0,T]$. Let $U=(U(s,y))_{s\in[0,T], y\in\R^d}\in C^{1,2}([0,T]\times \R^d, [0,\infty))$, and let $\gamma \geq 0$ and $\kappa\geq 0$. Assume that for all 
$x\in\cadlag, t\in[0,T]$
\begin{equation}\label{eq:Lassumption}
(\mathcal{G}_{f,g}U)(t,x) + \tfrac{1}{2} |\big(D_y U\big) (t,x(t)) g(t,x))|^2  \leq \gamma \sup_{s\in[-r,t]}U(s,x(s)) + \gamma \kappa.
\end{equation}
Then, for all $p\in(0,1)$ and all $t\in[0,T]$ 
\begin{equation*}
\EO\bigg[\sup_{t\in[0,T]}\exp(p U(t,X_t) \e^{-\gamma\beta T})\bigg]^{1/p} \leq \alpha_1\alpha_2 
\exp(U(0,X_0) + (\kappa + U_0)(1- e^{-\gamma \beta T})).
\end{equation*}
where $U_0 \define \sup_{s\in[-r,0]}U(s,z_s)$ and $\alpha_1$, $\alpha_2$ and $\beta$ only depend on $p$ (see \eqref{eq:constants}).
\end{corollary}
\begin{proof}[Proof of \cref{cor:exponentialMoments}]
We apply It\^o's formula to compute  $(Y_t)_{t\in[0,T]}  \define (U(t,X_t))_{t\in[0,T]}$:
\begin{equation*}
\begin{aligned}
\dd Y_t  = \dd U(t,X_t) & =  (G_{f,g} U)(t,X) \dd t  + \big(D_y U\big) (t,X_t) g(t,X) \dd B_t.
 \end{aligned}
\end{equation*}
We apply It\^o's formula to compute  $(Z_t)_{t\in[0,T]}  \define (\exp(Y_t))_{t\in[0,T]}$:
\begin{equation*}
\begin{aligned}
\dd Z_t  &= \dd \exp(Y_t)  = \exp(Y_t)\dd Y_t + \frac{1}{2} \exp(Y_t) \dd \langle Y, Y \rangle_t \\
& = \exp(Y_t)\bigg( (\mathcal{G}_{f,g} U)(t,X)  + \frac{1}{2} |\big(
D_y U\big) (t,X_t) g(t,X)|^2 \bigg) \dd t  + \dd \tilde{M}_t \\
& \leq   \gamma \exp(Y_t)(Y_t^* + \kappa  + \sup_{s\in[-r,0]}U(s,z_s)) \dd t  + \dd \tilde{M}_t
 \end{aligned}
\end{equation*}
where $\tilde{M} \define \int_0^t \exp(Y_s) \big(D_y U\big) (s,X_s) g(s,X) \dd B_s $ is a local martingale which starts in $0$. 
By assumption $U\geq 0$, and therefore $Z\geq 1$. Hence, we have for $\eta(x, u_0) = x(\log(x) + \kappa + u_0)$ for $x\geq 1$ and $u_0 \geq 0$: 
\begin{equation*}
Z_t \leq Z_0 + \int_0^t \eta(Z_s^*,U_0) \gamma \dd s  + \tilde{M}_t.
\end{equation*}
Using that $U_0$ is $\F_0$ measurable, we obtain by approximating $U_0$ from above by a sequence of random variables, where each random variable only takes countably many values, that \cref{thm:stochBihariConvex} implies for $A_t \define \gamma t$
\begin{equation*}
\| G^{-1}( G(Z_T^*) - \beta A_T)\|\pF \leq \alpha_1 \alpha_2 \|Z_0\|_{p,\F_0}
\end{equation*}
where
\begin{equation*}
G(x,\omega) = \log(\kappa + U_0(\omega) + \log(x)), \qquad 
G^{-1}(x,\omega) = \exp(\e^x-\kappa - U_0(\omega)),
\end{equation*}
where the inverse $G^{-1}$ is w.r.t to the variable $x$ for fixed $\omega\in\Omega$.
We compute the left-hand side:
\begin{equation*}
\begin{aligned}
\| G^{-1}( G(Z_T^*) - \beta A_T)\|\pF
& =\| \exp(\exp\{\log(\kappa + U_0 + \log(Z_T^*)) - \beta \gamma T\}-\kappa - U_0)\|\pF \\
&  =\| \exp(\exp\{\log(\kappa + U_0 + \log(Z_T^*))\} \e^{- \beta \gamma T}-\kappa - U_0)\|\pF \\
&  =\| \exp(\{\kappa + U_0 + \log(Z_T^*)\} \e^{- \beta \gamma T}) \e^{-\kappa - U_0}\|\pF\\
&  =\|(Z_T^*)^{e^{- \beta \gamma T}}  \exp(\{\kappa + U_0\}  \e^{- \beta \gamma T}))\e^{-\kappa - U_0}\|\pF \\
&  =\|(Z_T^*)^{e^{- \beta \gamma T}}\|\pF \e^{-(\kappa+U_0)(1- \e^{- \beta \gamma T})}.
\end{aligned}
\end{equation*}
Rearranging the terms and recalling the definition $Z_t = \exp(U(t,X_t))$ yields:
\begin{equation*}
\bigg\| \sup_{t\in[0,T]}\exp(U(t,X_t) \e^{-\gamma\beta T})\bigg\|\pF \leq \alpha_1\alpha_2 \e^{U(0,X_0)} \e^{(\kappa+U_0)(1- \e^{- \beta \gamma T})}.
\end{equation*}
\end{proof}


\begin{remark}
Note that \eqref{eq:Lassumption} is a weaker assumption than assuming
\begin{equation}\label{eq:LassumptionStrong}
\begin{aligned}
(\mathcal{G}_{f,g}U)(t,x) \leq \tilde{\gamma} \sup_{s\in[-r,t]}U(s,x(s)) + \tilde{\gamma}
\kappa,  \\
\text{and} \quad 
\tfrac{1}{2} |\big(D_y U\big) (t,x(t)) g(t,x))|^2  \leq \tilde{\gamma} \sup_{s\in[-r,t]}U(s,x(s)) + \tilde{\gamma} \kappa,
\end{aligned}
\end{equation}
since $(\mathcal{G}_{f,g}U)(t,x)$ can also be negative, and hence
\eqref{eq:Lassumption} allows for suitable $f$ a considerably weaker assumption on $g$.
Alternative proofs, using e.g. the pathwise  BDG inequality, Young's inequality and then e.g. It\^o's formula, only yield estimates under the stronger assumption  \eqref{eq:LassumptionStrong} in path-dependent case. See \cref{example:1} for an example, where \eqref{eq:Lassumption} is satisfied, but not \eqref{eq:LassumptionStrong}.
\end{remark}

We use in the following examples (as before) the notation $\beta = (1-p)^{-1}$, $\alpha_1 =(1-p)^{-1/p}$ and $\alpha_2 = p^{-1}$.

\begin{example} Let $X$ be a solution of the SDE \eqref{eq:SDE} satifying $\int_0^t |f(s,X)|\dd s + \int_0^t |g(s,X)|^2_F \dd s < \infty \,\, \PP$-a.s. for all $t\geq 0$.
For $U(t,x) \define R |x|^2$ for some constant $R>0$ the assumption \eqref{eq:Lassumption} can be slighty strengthened to
\begin{equation*}
2R \langle x(t), f(t,x)\rangle +
R |g(t,x)|_F^2 + 2R^2 |x(t)|^2 |g(t,x)|^2
\leq \gamma R \sup_{s\in[-r,t]}|x(s)|^2 + \gamma \kappa.
\end{equation*}
and obtain the estimate
\begin{equation*}
\EO\bigg[\sup_{t\in[0,T]}\exp(pR |X_t|^2 \e^{-\gamma\beta T})\bigg]^{1/p} \leq \alpha_1\alpha_2 
\e^{R|z_0|^2} e^{(\kappa + \sup_{t\in[-r,0]}R|z_t|^2)(1- e^{-\gamma \beta T})}.
\end{equation*}

In particular, if we assume there exist constants $\gamma_1 \geq 0$, $\gamma_2 \geq 0$ such that for all $\omega \in \Omega$, $t\in[0,T]$ and $x\in \cadlag$
\begin{equation*}
\langle x(t), f(t,x)\rangle \leq \gamma_1  \sup_{s\in[-r, t]} |x(s)|^2, \qquad 
|g(t,x)|_F^2 \leq \gamma_2.
\end{equation*}
Then, we have
\begin{equation*}
\EO\bigg[\sup_{t\in[0,T]}\exp(pR |X_t|^2 \e^{-\gamma\beta T})\bigg]^{1/p} \leq \alpha_1\alpha_2 
\exp(R|X_0|^2) e^{(\kappa + U_0)(1- e^{-\gamma \beta T})}
\end{equation*}
where $\gamma \define 2\gamma_1 + 2R \gamma_2$, $\kappa \define \frac{R \gamma_2}{2\gamma_1 + 2R \gamma_2}$ and $U_0\define \sup_{u\in[-r,0]}R |z_u|^2$.
\end{example}

\begin{example}\label{example:1}
Choose $U$ as in the previous example, $d=1$, $R=1$. Then, the following coefficients satisfy \eqref{eq:Lassumption} but not \eqref{eq:LassumptionStrong}: 
\begin{equation*}
f(t,x) = -\tfrac{1}{2} x(t) -x(t)^3 + \sup_{s<t} |x(s)|, \quad 
g(t,x) = x(t)^2 + (\sup_{s<t} |x(s)| \wedge 1) \quad \forall x\in\cadlag, \, \forall t\geq 0.
\end{equation*}

\end{example}

\begin{example}  Let $X$ be a solution of the SDE \eqref{eq:SDE} satisfying $\int_0^t |f(s,X)|\dd s + \int_0^t |g(s,X)|^2_F \dd s < \infty \,\, \PP$-a.s. for all $t\geq 0$. Choose $U(t,x) \define R (|x|^2+1)^{1/2}$ for some constant $R>0$.
Let the coefficients $f$ and $g$ satisfy 
\begin{equation}\label{eq:example2-Assumption}
\begin{aligned}
\langle x(t),f(t,x)\rangle + 
\tfrac{d}{2}   |g(t,x)|_F^2
+ \tfrac{1}{2} R (1+|x(t)|^2)^{-1/2} |g(t,x)|^
2 |x(t)|^2 \\ \leq \gamma \sup_{s\in[-r,t]} (|x(s)|^2+1)^{1/2}
(1+|x(t)|^2)^{1/2},
\end{aligned}
\end{equation}
then \cref{cor:exponentialMoments} implies for all $p\in(0,1)$ and all $T\geq 0$ 
\begin{equation*}
\EO\bigg[\sup_{t\in[0,T]}\exp(pR(|X_t|^2+1)^{1/2} \e^{-\gamma\beta T})\bigg]^{1/p} \leq \alpha_1\alpha_2 
\exp\big(R(|z_0|^2+1)^{1/2} +(\kappa +R(\sup_{t\in[-r,0]}|z_t|^2+1)^{1/2})(1- e^{-\gamma \beta T})\big).
\end{equation*}

This can be seen as follows: The function $y\mapsto  R (|y|^2+1)^{1/2}$ is convex, hence $\text{Hess}_y U(t,y)$ is positive semidefinite for any $t\geq 0, y\in\R^d$. Moreover, we have $\text{trace}(\text{Hess}_y U(t,x(t)))\leq Rd(1+|x(t)|^2)^{-1/2}$. Using that $\text{trace}(A_1 A_2) \leq \text{trace}(A_1) \text{trace}(A_2)$ for symmetric positive semidefinite matrices $A_1, A_2$, the assumption \eqref{eq:Lassumption} can be strengthened to
\begin{equation*}
\begin{aligned}
(1+|x(t)|^2)^{-1/2}R \langle x(t),f(t,x)\rangle + 
\tfrac{1}{2} \text{trace}(g(t,x)g(t,x)^T) Rd(1+|x(t)|^2)^{-1/2} \\
+ \tfrac{1}{2} (1+|x(t)|^2)^{-1}R^2 |g(t,x)|^2 |x(t)|^2 \leq \gamma R \sup_{s\in[-r,t]} (|x(s)|^2+1)^{1/2}.
\end{aligned}
\end{equation*}
Multiplying with $(1+|x(t)|^2)^{1/2}R^{-1}$ implies \eqref{eq:example2-Assumption}.

In particular, if there exist constants $\gamma_1 \geq 0$, $\gamma_2 \geq 0$ such that for all 
$\omega \in \Omega$, $t\in[0,T]$ and $x\in \cadlag$
\begin{equation*}
\begin{aligned}
\langle x(t),f(t,x)\rangle & \leq \gamma_1 \sup_{s\in[-r,t]} (|x(s)|^2+1)^{1/2}
(1+|x(t)|^2)^{1/2}, \\
|g(t,x)|_F^2 & \leq \gamma_2 \sup_{s\in[-r,t]} (|x(s)|^2+1)^{1/2},
\end{aligned}
\end{equation*}
then, we have
\begin{equation*}
\EO\bigg[\sup_{t\in[0,T]}\exp(pR(|X_t|^2 +1)^{1/2}\e^{-\gamma\beta T})\bigg]^{1/p} \leq \alpha_1\alpha_2 
\exp\big\{ R(|z_0|^2+1)^{1/2} + U_0 (1- e^{-\gamma \beta T})\big\}
\end{equation*}
where $\gamma \define \gamma_1 + \tfrac{d}{2} \gamma_2 + \tfrac{R}{2}\gamma_2$ and
 $U_0\define \sup_{u\in[-r,0]}R(|z_u|^2+1)^{1/2}$.
\end{example}

\subsection{Tail estimates}
\begin{remark}[Estimate for non-path-dependent SDEs]
Let $Y$ be a global solution of the following (\emph{not} path-dependent) Brownian-driven SDE
\begin{equation*}
d Y_t = \tilde{f}(t,Y_t) \dd t + \tilde{g}(t,Y_t) \dd B\ti, \qquad  Y_0 = y_0,
\end{equation*}
where $y_0\in\R^d$ is a deterministic initial value and $\tilde{f}$ and $\tilde{g}$ satify suitable measurability conditions to make the SDE well-defined. It can be easily seen that if the coefficients satisfy the following one-sided coercivity condition (for some $K>0$) 
\begin{equation*}
2 \langle y, \tilde{f}(t,y)) \rangle + |\tilde{g}(t,y)|_F^2 \leq K |y|^2 \qquad  \forall y\in \R^d, \,  \forall t\geq 0
\end{equation*}
then we have (e.g. by applying \cref{thm:stochBihariConvex} b) or \cref{cor:GronwallNoSup} to $(|Y|^2_t)_{t\geq 0}$) for all $u>0$
\begin{equation}\label{eq:estimate-SDE}
\PP\left[\sup_{t\in[0,T]} |Y|_t^2 >u\right] \leq  \frac{\e^{KT}}{u}|y_0|^2,
\end{equation}
i.e. in particular $\|(|Y|^*_T)^2\|_{L^{1,\text{w}}} \leq \e^{KT}|y_0|^2< \infty$ for all $T\geq 0$.
\end{remark}

The following corollary of \cref{lemma:pathdependentSDE} and \cref{thm:sharp-tail} shows that SDEs with a path-dependent drift coefficient enjoy in general a faster growth in $u$.
\begin{corollary}\label{cor:tail-of-path-SDE} Let $d=1$, $r>0$ and $T>0$. For the path-dependent SDE \eqref{eq:SDE} with initial value $z_t=\sqrt{2} \,\,\forall t\in[-r,0]$ there exist coefficients $f$ and $g$ and a global strong solution $(Y_t)_{t\in[-r,\infty)}$ satisfying
\begin{enumerate}
\item $\int_0^T |f(s,Y)|\dd s + \int_0^T |g(s,Y)|^2_F \dd s < \infty \,\, \PP$-a.s.,
\item $\forall t\geq 0$, $y \in \cadlag$
\begin{equation*}
2 \langle y(t), f(t,y) \rangle + |g(t,y)|_F^2 \leq \sup_{s\in[0,t]}|y(s)|^2,
\end{equation*}
\item and $\sup_{u>0} \big(u\PP[\sup_{t\in[0,T]}Y_t^2 >u]\big)  = \infty$.
\end{enumerate}
In particular an estimate of the type \eqref{eq:estimate-SDE} does not hold.
\end{corollary}

\begin{proof}[Proof of \cref{cor:tail-of-path-SDE}] 
Fix some $\varepsilon, \delta \in (0,1)$ and $k\in\N$ such that $T = k\delta + \varepsilon\delta$. Let $(X_t)_{t\geq 0}$ be the process from \cref{lemma:pathdependentSDE}. Define $Y_t \define \sqrt{X_t + 1}$ for all $t>0$ and $Y_t \define \sqrt{2} = z_t \, \forall t\in[-r,0]$. Since $X_t \geq 0$ for all $t$ we have 
$Y_t \geq 1$ for all $t$ and $X_t = Y_t^2 -1$ for all $t\geq 0$.
Moreover, by Ito's formula we have

\begin{equation*}
\dd Y_t = \frac{1}{2} (X_t + 1)^{-1/2} \dd X_t  - \frac{1}{8} (X_t + 1)^{-3/2}\dd \langle X, X\rangle_t  = f(t,Y) \dd t + g(t,Y) \dd W_t
\end{equation*}
where  $\forall t\geq 0$, $\forall y\in C([-r,\infty),\R)$
\begin{equation*}
\begin{aligned}
f(t,y) &\define \left( \frac{1}{2} \frac{b(t,y^2-1)}{y_t}
-\frac{1}{8}\frac{1}{y_t^3}\sigma^2(s,y^2 -1)\right) \one_{\{y_t>1/2\}} \\
g(t,y) &\define  \frac{1}{2} \frac{1}{y_t} \sigma(t,y^2-1) \one_{\{y_t>1/2\}}
\end{aligned}
\end{equation*}
where $b$ and $\sigma$ are defined as in \cref{lemma:pathdependentSDE}. 
$b(t,y^2-1)$ denotes $b(t,(y(t)^2-1)_{t\in[0,\infty)})$.

\begin{enumerate}
\item We have 
\begin{equation*}
\int_0^T |f(s,Y)| \dd s + \int_0^T |g(s,Y)|_F^2 \dd s
\leq \int_0^T \frac{1}{2} |b(s,X)| \dd s + 
\left(\frac{1}{8} + \frac{1}{4}\right)\int_0^T|\sigma(s,X)|_F^2 \dd s
\end{equation*}
which is finite $\PP$-a.s. by \cref{lemma:pathdependentSDE}.
\item We have for $y\in C([-r,\infty),\R)$ and $t\geq 0$ such that $y(t)>1/2$:
\begin{equation*}
\begin{aligned}
2\langle y(t), f(t,y) \rangle + |g(t,y)|_F^2  & = 
b(t,y^2-1)
-\frac{2}{8}\frac{1}{y_t^2}\sigma^2(t,y^2 -1) + \frac{1}{4} \frac{1}{y^2_t} \sigma^2(t,y^2-1) \\ & = b(t,y^2-1) \leq \sup_{u\in[0,t]} y^2(u)
\end{aligned}
\end{equation*}

\item By \cref{thm:sharp-tail} we have 
$\sup_{u>0}\big(u\PP[X_t >u]\big)= \infty$, which implies 
\begin{equation*}
\sup_{v>1}\big(v\PP[Y^2_T >v]\big)= \sup_{v>1}\big(v\PP[X_T>v-1]\big)=  \sup_{u>0}\big((u+1)\PP[X_T>u]\big)=\infty.
\end{equation*}
\end{enumerate}
\end{proof}


\section{Appendix}
\subsection{Proof of the sharpness of $\alpha_1$ and $\alpha_1\alpha_2$ (\cref{thm:sharpnessAlpha})}
\begin{proof}[Sketch of proof of \cref{thm:sharpnessAlpha}] 
The inequalities are proven in \cref{lemma:2}. Here we only shortly discuss the sharpness of the constants: \\

\noindent  Proof of \ref{item:sharpnessAlpha3}: Let $B$ be a Brownian motion on a suitable underlying filtered probability space. Choose $H\ti\equiv 1$. Let $\tau$ be the time $B$ first hits $-1$ and set $M\ti \define B_{t\wedge\tau}$ and $X\ti \define M\ti + H\ti$ for all $t\geq 0$. The stopping times $\tau$ ensures $X \geq 0$. An easy calculation (which was also used in step 4 of the proof of \cref{thm:sharpnessBeta} to compute $\E[(X^{\varepsilon,\delta}_{\tau_0})^{*,p}] = \frac{1}{1-p}$) gives $\|\sup_{t\geq 0} X\ti\|_p = (1-p)^{-1/p}$ for $p\in(0,1)$, which implies the assertion of \ref{item:sharpnessAlpha3}.

\bigskip

\noindent Proof  of \ref{item:sharpnessAlpha1} and \ref{item:sharpnessAlpha2}:
Note that the assertions of  \ref{item:sharpnessAlpha1} and \ref{item:sharpnessAlpha2}
concerning sharpness are identical. Fix some $p\in(0,1)$. In the proof of \cite[Theorem 2.1]{GeissScheutzow}) families of continuous processes $X^{(n)}, n\in\N$ and $H^{(n)}, n\in\N$ are defined, satisfying:
\begin{enumerate}[label=(\roman*)]
\item $X\nn$ is non-negative, adapted, continuous,
\item $H\nn$ is non-negative, adapted, continuous, non-decreasing,
\item $\E[X^{(n)}_{\tau}] \leq \E[H^{(n)}_{\tau}]$ for all bounded stopping times $\tau$,
\item and  \begin{displaymath} 
\alpha_1 \alpha_2 = \lim_{n\to\infty} \frac{\|\sup_{t\geq 0} X^{(n)}\ti \|_p}{\|\sup_{t\geq 0} H^{(n)}\ti\|_p}. 
\end{displaymath}  
\end{enumerate}
\bigskip

\noindent To prove the sharpness assertion of \ref{item:sharpnessAlpha1} and \ref{item:sharpnessAlpha2}, it remains to show the existence of a family of local martingales
$M\nn,n\in\N $ with no negative jumps and $M^{(n)}_0=0$ such that
$X\nn\ti \leq H\nn\ti + M\nn\ti$ a.e. for all $t\geq 0, n\in\N$. \\

\noindent To this end, we first shortly recall the definition of $X^{(n)}$ and $H^{(n)}$ from  \cite[Theorem 2.1]{GeissScheutzow}: Let $Z$ be an exponentially distributed random variable on a complete probability space $(\Omega, \F, \PP)$ with $\E[Z]=1$. Set
\begin{equation*}
a:[0,\infty) \to [0,\infty), \quad t\mapsto \exp(t/p).
\end{equation*}
Define for all $t\geq 0$
\begin{equation*}
\tilde{X}_t \define a(Z)\one_{[Z,\infty)}(t), \qquad \tilde{H}_t \define \int_0^{t\wedge Z} a(s)\dd s.
\end{equation*}

\noindent Choose $\tilde{\F}_t \define \sigma(\{ Z\leq r\} \mid 0\leq r \leq t)$ for all $t\geq 0$.
The compensator of $\tilde{X}$ is $\tilde{H}$ due to $Z$ being exponentially distributed. Now we use $\tilde{X}$ and $\tilde{H}$ to construct the families of processes $X\nn, n\in\N$ and $H\nn, n\in\N$.  Assume w.l.o.g. that there exists a Brownian motion $B$ on $(\Omega, \F,\PP)$ such that $B$ is independent of $Z$. Denote by $(\F_t)_{t\geq 0}$ the smallest filtration satisfying the usual conditions which contains $(\tilde{\F}_t)_t$ and with respect to which 
$B$ is a Brownian motion. Denote by $g_{n,n+1}:[0,\infty) \to [0,1]$ a continuous non-decreasing function such that \begin{equation*}
g_{n,n+1}(t) = 0  \quad \forall t\in[0,n], \,\, \text{and}  \,\,\, g_{n,n+1}(t) = 1 \quad \forall t\in[n+1,\infty).
\end{equation*}
Define:
\begin{equation*}
\begin{aligned}
\tau\nn & \define \inf\{t\geq n+1 \mid \tilde{X}_n + (B\ti - B_{n+1})\one_{\{t\geq n+1\}} = 0\}, \\
X\nn_t & \define  g_{n,n+1}(t) \tilde{X}_n + (B_{t\wedge\tau\nn} - B_{t\wedge(n+1)}), \\
H\nn_t & \define \tilde{H}_{t\wedge n}.
\end{aligned}
\end{equation*}
The stopping time $\tau\nn$ ensures that $X\nn$ is non-negative. Define
\begin{equation*}
M\nn\ti \define  \tilde{X}_{t\wedge n}-\tilde{H}_{t\wedge n} + B_{t\wedge \tau\nn} - B_{t\wedge(n+1)},
\end{equation*}
which is a local martingale (recall that $\tilde{H}$ is the compensator of $\tilde{X}$). Moreover, 
\begin{equation*}
X^{(n)}\ti \leq  \tilde{X}_{t\wedge n} + \big(B_{t\wedge\tau\nn} - B_{t\wedge(n+1)}\big) 
\leq H\nn\ti + M\nn\ti.
\end{equation*}
It is easily seen that $M$ only has non-negative jumps, as $\tilde{X}$ only has non-negative jumps.
\end{proof}

\subsection{Proof of time change lemma (\cref{lemma:smoothenJumps})}
\begin{proof}[Proof of \cref{lemma:smoothenJumps}]
We prove the lemma for \nameref{def:nosup}, the proof for \nameref{def:sup} is upto one minor difference (mentioned in the proof below) identical.

To keep the notation simple, we first prove the claim for the special case that $A$ has at most one jump. Afterwards we extend the same technique to prove the assertion for  $A$ having finitely many jumps on each path. By approximation we then can prove the assertion for general $A$. \\

\noindent \textbf{Step 1:} We prove the lemma under the additional assumption that $A$ has at most one jump. \\

We will first smoothen out the jump of $A$ yielding $(\hat{X},\,\hat{A},\,\hat{H},\,\hat{M})$ and $(\hat{\F}_t)_{t\geq 0}$  enjoying properties $\Anosup$ and c) - e) and such that $\hat{A}$ is continuous and strictly increasing. By a time shift the family $(\hat{X},\,\hat{A},\,\hat{H},\,\hat{M})$ will yield $(\tilde{X},\,\tilde{A},\,\tilde{H},\,\tilde{M})$ i.e. the claim of the lemma. 

More precisely, define $\tau_1 \define \inf\{t>0 \mid \Delta A_t >0\}$ where $\inf \emptyset \define \infty$.  We smoothen the jump of $A$ by inserting at time $\tau_1$ a time interval of length $1$:
\begin{equation*}
\hat{A}_t \define 
\begin{cases}
A_t & \forall t\in[0,\tau_1) \\
\text{linear interpolation between } A_{\tau_{1}-} \text{ and } A_{\tau_{1}}  &  \forall t\in[\tau_1,\tau_1+1) \\
A_{t-1} & \forall t\in[\tau_1+1,\infty). 
\end{cases}
\end{equation*}
We also define correspondingly time-changed processes $\hat{X}$, $\hat{H}$, $\hat{M}$ such that $(\hat{X},\,\hat{A},\,\hat{H},\,\hat{M})$ satisfy $\Anosup$. Note that the following definition would \textit{not} work unless $M$ has only non-negative jumps, which is also the reason why the choice $\tilde{X}_t \define X_{A_t^{-1}}$ (for a generalized inverse $A^{-1}$) does not work in general:
\begin{equation*}
\bar{Y}_t \define 
\begin{cases}
Y_t & \forall t\in[0,\tau_1) \\
Y_{\tau_1} & \forall t\in[\tau_1,\tau_1+1) \\
Y_{t-1} & \forall t\in [\tau_1+1,\infty)  
\end{cases} \qquad \text{for } Y\in\{X,M,H\}.
\end{equation*}
The processes $\bar{X}, \hat{A}, \bar{H},\bar{M}$ will not
satisfy \eqref{eq:GronwallAssumptionNoSup} on the interval $[\tau_1, \tau_1 +1)$ in general: At time $t=\tau_1$ the local martingale $M$ might have a large negative jump $\Delta M_{\tau_1} \ll 0$ such that the right-hand side of \eqref{eq:GronwallAssumptionNoSup} becomes negative, hence in particular less than $X_{\tau_1}$. Instead, we define in the case of \nameref{def:nosup}
\begin{equation*}
\hat{Y}_t \define 
\begin{cases}
Y_t & \forall t\in[0,\tau_1) \\
Y_{\tau_1-} & \forall t\in[\tau_1,\tau_1+1) \\
Y_{t-1} & \forall  t\in [\tau_1+1,\infty)   
\end{cases} \qquad \text{for } Y\in\{X,M,H\}.
\end{equation*}
By this definition, $(\hat{X},\,\hat{A},\,\hat{H},\,\hat{M})$ satisfy \eqref{eq:GronwallAssumptionNoSup}: For $t\in[0,\tau_1)$ this is trivial, for time $t=\tau_1$ this follows by taking the left limits, for $t\in(\tau_1,\tau_1+1)$ only the integral term $\int_{(0,t]} \eta(X_{s^-})\dd A_s$ is increasing, hence it follows from \eqref{eq:GronwallAssumptionNoSup} holding for $t=\tau_1$. For $t\geq \tau_1 + 1$ \eqref{eq:GronwallAssumptionNoSup} corresponds to 
\eqref{eq:GronwallAssumptionNoSup} for $(X, A, H, M)$ at time $t-1$.

In the case of \nameref{def:sup} we slightly modify the definition of $\hat{X}$ as we do not assume that $X$ has left limits in this case:
\begin{equation*}
\hat{X}_t \define 
\begin{cases}
X_t & \forall t\in[0,\tau_1) \\
\liminf_{s\nearrow \tau_1} X_{s} \quad  & \forall t\in[\tau_1,\tau_1+1) \\
X_{t-1} & \forall t\in [\tau_1+1,\infty).
\end{cases} 
\end{equation*}
Due to the supremum in the integral $\int_{(0,t]} \eta(\hat{X}^*\sm) \dd A\s$ the processes $(\hat{X},\,\hat{A},\,\hat{H},\,\hat{M})$ satisfy \eqref{eq:GronwallAssumption}.

Now we define a suitable filtration which ensures that $\hat{M}$ is a local martingale. Due to $A$ being predictable, there exists an announcing sequence $\tau_1^{(n)}$  of $\tau_1$. Define the following stopping times for each $n\in\N$ and $t\geq 0$:
\begin{equation*}
\sigma_n(t) \define 
\begin{cases}
t\wedge \tau_1^{(n)} & t\in[0,\tau_1^{(n)}+1) \\
t-1 & t\in[\tau_1^{(n)}+1,\infty).
\end{cases}
\end{equation*}
Note that $\sigma_n(t)$ in non-decreasing in $n$ (for each $t\geq 0$ and each $\omega\in\Omega$). Moreover, we have $\lim_{n\to\infty} Y_{\sigma_n(t)} = \hat{Y}_t$ for $Y=X,\,H,\,M$ in the case of \nameref{def:nosup}. (In the case of \nameref{def:sup} this only holds for $Y=\,H,\,M$ however the following argumentation is still possible.)
 We choose $\bar{\F}_t \define \vee_{n\in\N} \F_{\sigma_n(t)}$ and denote by $(\hat{\F}_t)_{t\geq 0}$ the smallest filtration that contains $(\bar{\F}_t)_{t\geq 0}$ and satisfies the usual conditions.
Clearly, $\hat{X}$, $\hat{H}$ and $\hat{M}$ are adapted w.r.t. to $(\hat{\F}_t)_{t\geq 0}$.  The property d) is immediate and by definition e) is satisfied.

We show that $\hat{M}$ is a local martingale w.r.t. $(\bar{\F}_t)_{t\geq 0}$ (and hence also $(\hat{\F}_t)_{t\geq 0}$): We may assume w.l.o.g. that $\E[|M_\infty|] < \infty$. Fix some $0\leq s < t$. Define 
\begin{equation*}
\mathcal{D} = \{A \in \F \mid \E[(\hat{M}_t - \hat{M}_s) \one_A] = 0\}.
\end{equation*}
To prove that  $\hat{M}$ is a local martingale it suffices to show $
\bar{\F}_s \subseteq \mathcal{D}$. Due to pointwise convergence $\lim_{n\to\infty} M_{\sigma_n(r)} = \hat{M}_r$ and uniform integrability of $\{M_{\sigma_n(r)}, n\in\N\}$ for $r\in\{s,t\}$, we have for all $A\in\F_{\sigma_m(s)}, m\in\N$
\begin{equation*}
\E[(\hat{M}_t - \hat{M}_s) \one_A] = \E[\lim_{n\to\infty} (M_{\sigma_n(t)} - M_{\sigma_n(s)}) \one_A] =
\lim_{n\to\infty} \E[(M_{\sigma_n(t)} - M_{\sigma_n(s)}) \one_A]  = 0.
\end{equation*}
This implies that $\F_{\sigma_n(s)}\subseteq \mathcal{D}$ for every $n$. Hence the $\pi$-system $\cup_{n\in\N}\F_{\sigma_n(t)}$ is a subset of the Dynkin system $\mathcal{D}$, so the $\lambda$-$\pi$ theorem implies that
$\bar{\F}_s\subseteq \mathcal{D}$. This proves that $\hat{M}$ is indeed a local martingale w.r.t. $(\bar{\F}_t)_{t\geq 0}$ (and hence also $(\hat{\F}_t)_{t\geq 0}$). 

We show that predictability of $H$ implies predictability of $\hat{H}$: Continuous adapted processes remain continuous and adapted by the transformation $Y\mapsto \hat{Y}$. So we apply the monotone class theorem to $
\mathcal{H} = \{Y:[0,\infty)\times\Omega\to \R \mid \hat{Y} \text{ is predictable and bounded}\}$, noting that $\mathcal{M} = \{Y:[0,\infty)\times\Omega\to \R \mid \hat{Y} \text{ continuous, adapted, bounded}\}$ is closed under multiplication and contained in $\mathcal{H}$.

Now we use $(\hat{X},\,\hat{A},\,\hat{H},\,\hat{M})$ and $(\hat{\F}_t)_{t\geq 0}$ to construct  $(\tilde{X}, \tilde{A}, \tilde{H}, \tilde{M})$ and $(\Omega, \F,\PP, (\tilde{\F}_t)_{t\geq 0})$, i.e. prove Step 1. Noting that $\hat{A}$ is continuous, strictly increasing and $\hat{A}_\infty = +\infty$, $\hat{A}^{-1}(\omega):[0,\infty) \to [0,\infty)$ is a well-defined mapping for every $\omega\in\Omega$. Note that $\hat{A}^{-1}_t$ is a  $(\hat{\F}_{t})_{t\geq 0}$ stopping time for all $t\geq 0$.  We define for all $t\geq 0$
\begin{equation*}
\begin{aligned}
\tilde{Y}_t \define \hat{Y}_{\hat{A}^{-1}_t} \, \text{ for }\, Y \in \{X, A, H, M\} \quad \text{ and } \quad \tilde{\F}_t \define \hat{\F}_{\hat{A}^{-1}_t}.
\end{aligned}
\end{equation*}
We verify that the family $(\tilde{X},\,\tilde{A},\,\tilde{H},\,\tilde{M})$  satisfies inequality \eqref{eq:GronwallAssumptionNoSup} of $\Anosup$:
\begin{equation*}
\begin{aligned}
\tilde{X}_t  = \hat{X}_{\hat{A}^{-1}_t} & \leq 
\int_{(0,\hat{A}^{-1}_t]} \eta(\hat{X}\sm) \dd \hat{A}\s + \hat{M}_{\hat{A}^{-1}_t} + \hat{H}_{\hat{A}^{-1}_t} \\
 & = \int_{(0,\hat{A}^{-1}_t]} \eta(\hat{X}\sm) \dd \hat{A}\s 
+ \tilde{M}_t + \tilde{H}_t.
\end{aligned}
\end{equation*}
By change of variables we obtain $\int_{(0,\hat{A}^{-1}_t]} \eta(\hat{X}\sm) \dd \hat{A}\s =\int_{(0,t]} \eta(\hat{X}_{(\hat{A}_s^{-1}) -}) \dd s$. Since $\hat{A}^{-1}$ is continuous and strictly increasing, we have 
$\hat{X}_{(\hat{A}_s^{-1})-}
= \lim_{r\nearrow \hat{A}_s^{-1}} \hat{X}_r
= \lim_{r\nearrow s}\hat{X}_{(\hat{A}_r^{-1})}
= \tilde{X}_{s^-}$, which proves that inequality \eqref{eq:GronwallAssumptionNoSup} is satisfied by $(\tilde{X},\,\tilde{A},\,\tilde{H},\,\tilde{M})$.
It is clear that a) holds. The construction of the processes $(\hat{X},\,\hat{A},\,\hat{H},\,\hat{M})$ and  $\tilde{Y}_{\hat{A}_t} = \hat{Y}_{t}$ implies $\tilde{Y}_{A_t} = Y_{t}$, and hence the equalities of b) hold.

We verify that $A_s$ is a $(\tilde{\F}_t)_{t\geq 0}$ stopping time for any $s\geq 0$: For this we first show that $\tau_1$ is a $(\hat{\F}_t)_{t\geq 0}$ stopping time. We have due to $t\wedge\tau^{(n)}_1\leq \sigma_n(t)$ that $\F_{t\wedge \tau^{(n)}_1} \subseteq \F_{\sigma_n(t)}$. This implies:
\begin{equation*} 
\{\tau_1\leq t\} = \bigcap_{n\in\N} \{\tau^{(n)}_1< t\} = \bigcap_{n\in\N} \{\tau^{(n)}_1\wedge t< t\} \in \vee_{n\in\N} \F_{\sigma_n(t)} \subseteq \hat{\F}_{t}.
\end{equation*}
This implies by the definition of $\hat{A}$ and by using that $\hat{A}_t^{-1}$ is a $(\hat{\F}_t)_{t\geq 0}$ stopping time:
\begin{equation*}
\begin{aligned}
\{A_s \leq t\} &  = \{\hat{A}_s \leq t, s \in [0,\tau_1)\} \cup 
\{\hat{A}_{s+1} \leq t, s+1 \in[\tau_1+1, \infty)\} \\
&=\left( \{s \leq \hat{A}^{-1}_t\} \cap \{s\leq \tau_1 \wedge \hat{A}^{-1}_t\}\right)
\cup \left( \{s+1 \leq \hat{A}^{-1}_t\} \cap \{s+1 \geq (\tau_1+1)\wedge \hat{A}^{-1}_t\}\right) \in \hat{\F}_{\hat{A}^{-1}_t} = \tilde{\F}_t.
\end{aligned}
\end{equation*}
Hence b) is satisfied.

As $\hat{A}^{-1}$ is continuous, it is clear that c) and d) hold true. The right-continuity of $(\hat{\F}_{t})_{t\geq 0}$  and $(\hat{A}_t^{-1})_{t\geq 0}$ implies that $(\tilde{\F}_{t})_{t\geq 0}$ is a right-continuous filtration, therefore e) holds. 

\bigskip

\noindent \textbf{Step 2:} We define $\tilde{X}$, $\tilde{H}$, $\tilde{M}$ for general $A$. \\

The paths of $A$ are strictly increasing. We define the generalized inverse $A^{-1}(\omega):[0,\infty)\to[0,\infty)$ pathwise by
\begin{equation*}
A^{-1}_t(\omega) = \inf\{ r\geq 0 \mid A_r(\omega) \geq t\} 
\end{equation*} 
which satisfies $A^{-1}(\omega)\circ A (\omega)= \id_{[0,\infty)}$ and $A(\omega)\circ A^{-1}(\omega)|_{\text{range}(A(\omega))} = \id_{\text{range}(A(\omega))}$ on each path. Note that due to $A$ being \cadlagg, we have $(A(\omega)\circ A^{-1}(\omega))_t \geq t$. Note that this implies 
\begin{equation}\label{eq:smoothenProperty}
\forall t_1 \in \text{range}(A)(\omega) \text{ and } t_2 > t_1 \text{ we have } A^{-1}_{t_2}(\omega)> A^{-1}_{t_1}(\omega),
\end{equation}
however $A^{-1}$ is not strictly increasing,  only non-decreasing. We define
\begin{equation}\label{eq:LemmaSmoothen}
\tilde{Y}_t  = 
\begin{cases}
Y_{A^{-1}_t} &  \text{ if } t\in\text{range}(A) \\
\liminf_{r\nearrow A^{-1}_t} Y_{r} & \text{ if } t\notin\text{range}(A) 
\end{cases} \qquad \text{for  all } Y \in \{X, H, M\}.
\end{equation}
This definition generalizes the definition of Step 1. 
It is easily checked that $\tilde{Y}$ is right-continuous: Let $t\notin\text{range}(A)$. Then, $A_{s-} \leq t < A_s$ for $s=A^{-1}_t$. This implies $[t,A_s)\cap\text{range}(A) = \emptyset$ and that we have $A^{-1}_t = A^{-1}_r$ for all $ r\in[t,A_s)$. Together this implies that $\tilde{Y}$ is right-continuous in $t$. Now consider $t\in\text{range}(A)$. $A^{-1}$ is continuous and non-decreasing, hence for $t_n \searrow t$ we have $A^{-1}_{t_n} \searrow A^{-1}_t$, and hence the right-continuity of $Y$  and \eqref{eq:smoothenProperty} imply $\tilde{Y}$ is right-continuous in $t$. By similar arguments, using that $A^{-1}$ is continuous and non-decreasing, it can be verified that $\tilde{Y}$ has left limits if $Y$ has left limits.

Moreover, for $t\in \text{range}(A)$ we have
\begin{equation*}
\tilde{X}\ti \leq \int_{(0,t]} \eta(\tilde{X}\sm) \dd s + \tilde{M}\ti + \tilde{H}\ti \qquad \PP\text{-a.s}
\end{equation*}
due to change of variables (using $\tilde{Y}_{s^-} = Y_{\cdot -} \circ A^{-1}_s$
and $\{x\in]0,t] \mid A^{-1}_x \in ]a, b]\} = ]A_a, A_b]$ due to \eqref{eq:smoothenProperty}. For $t\notin \text{range}(A)$ it follows from the same argument as in Step 1.

However, defining a suitable filtration $(\tilde{\F}_t)_{t\geq 0}$ seems to be non-trivial: The choice $\tilde{\F}_t = \F_{A^{-1}_t}$ is \textit{not possible} since $\tilde{M}$ is not a martingale with respect to this filtration. Instead, we  first find a filtration for the special case that $A$ has at most finitely many jumps on each path (Step 3) and then obtain the desired result by an approximation argument (Step 4). For Step 3 it is crucial that $A$ is predictable. \\

\noindent \textbf{Step 3:} We prove the lemma under the following additional assumption: Assume that the jumps of $A$ be bounded from below, i.e. that $\exists \varepsilon >0$ such that for all $t\geq 0$ and all $\omega \in \Omega$  it holds that $\Delta A_t(\omega) \notin (0,\varepsilon)$. \\

We construct the filtration $(\tilde{\F}_t)_{t\geq 0}$ needed to complete the construction of Step 2 by repeating
the construction of Step 1. Using that $A$ has on each path on every bounded time interval at most finitely many jumps, we may smoothen the jumps of $A$ by inserting at the $k$-th jump of $A$ a time interval of length $2^{-k}$. To this end, set $\tau_{0}\define 0$ and denote by $\tau_{k}$ the time of the $k$th jump of  $A$, i.e. $\tau_{k} \define \inf\{t> \tau_{k-1} \mid \Delta A_t >0\}$ using $\inf \emptyset \define \infty$. Moreover, set $s_0 \define 0$ and $s_k \define \sum_{1\leq i\leq k} 2^{-i}$. Noting that for any $T>0$ we have finitely many $k$ with $\tau_k<T$ and $s_\infty = 1$, we define:
\begin{equation*}
\begin{aligned}
\hat{A}_t & \define  
\begin{cases} 
A_{t-s_i}   & \,\, \text{for } t\in [\tau_i + s_i, \tau_{i+i}+s_i),\,\, i\in\N_0,    \\
\text{linear interpolation between } A_{\tau_{i+1}-} \text{ and } A_{\tau_{i+1}}  &\,\, \text{for }  t \in[\tau_{i+1} + s_i, \tau_{i+1}+s_{i+1}),\,\, i\in\N_0. \\
\end{cases} \\
\hat{Y}_t &\define \begin{cases}
Y_{t-s_i}   & \,\, \text{for } t\in [\tau_i + s_i, \tau_{i+i}+s_i),\,\, i\in\N_0, \\
Y_{\tau_{i+1}-} &\,\, \text{for }  t \in[\tau_{i+1} + s_i, \tau_{i+1}+s_{i+1}),\,\, i\in\N_0,\\
\end{cases}  \qquad \text{for } Y\in\{X,M,H\}.
\end{aligned}
\end{equation*}
We can construct a sequence of announcing times
$\tau^{(n)}_i$, $n\in\N$, $i\in\N$ such that each $(\tau^{(n)}_i)_n$ announces $\tau_i$ and
\begin{equation*}
0 = \tau_0 \leq \tau^{(n)}_1 < \tau_1 \leq \tau_2^{(n)} < \tau_2 \leq  ... \leq \tau^{(n)}_i < \tau_i \leq \tau^{(n)}_{i+1} < \tau_{i+1} \leq ...
\end{equation*}
on $\{\tau_{i+1}<\infty\}$. Moreover, we define $\tau_0^{(n)} \define 0$ for all $n\in\N$. We define analogously as before:
\begin{equation*}
\sigma_n(t) 
\define  \begin{cases}
t-s_i   & \,\, \text{for } t\in [\tau^{(n)}_i + s_i, \tau^{(n)}_{i+i}+s_i),\,\, i\in\N_0, \\
\tau^{(n)}_{i+1} &\,\, \text{for }  t \in[\tau^{(n)}_{i+1} + s_i, \tau^{(n)}_{i+1}+s_{i+1}),\,\, i\in\N_0,  
\end{cases}
\end{equation*}
Set $\bar{\F}_t \define \vee_{n\in\N} \F_{\sigma_n(t)}$ and 
denote by $(\hat{\F}_t)_{t\geq 0}$ the smallest filtration that contains $(\bar{\F}_t)_{t\geq 0}$ and satisfies the usual conditions.
By the same arguments as in the first part of the proof, 
$(\Omega, \F,\PP, (\hat{\F}_t)_{t\geq 0})$  and $(\hat{X}_t)_{t\geq 0}$,  $(\hat{A}_t)_{t\geq 0}$, $(\hat{H}_t)_{t\geq 0}$, $(\hat{M}_t)_{t\geq 0}$  satisfy $\Anosup$  and c)-e). As in Step 1 we define for all $t\geq 0$
\begin{equation*}
\begin{aligned}
\tilde{Y}_t \define \hat{Y}_{\hat{A}^{-1}_t} \, \text{ for }\, Y \in \{X, A, H, M\} \quad \text{ and } \quad \tilde{\F}_t \define \hat{\F}_{\hat{A}^{-1}_t},
\end{aligned}
\end{equation*} 
which satisfy by the same arguments the assertion of this lemma. Moreover, this definition yields the same processes as definition \eqref{eq:LemmaSmoothen}.\\

\noindent \textbf{Step 4:} We prove the assertion of the lemma (without additional assumptions on $A$). \\

In \eqref{eq:LemmaSmoothen} of Step 2 we already defined $\tilde{X}$, $\tilde{H}$, $\tilde{M}$. It remains to find a suitable filtration and prove that $\tilde{M}$ is indeed a local martingale. To this end we use Step 3 and approximations. We define
\begin{equation*}
A^{n, small}_t \define \sum_{s\leq t} \Delta A_s \one_{\{\Delta A_s \leq \frac{1}{n} \}}, \qquad A^{(n)}_t \define A_t - A^{n,small}_t, \qquad \forall t\geq 0,
\end{equation*}
i.e. for all $t\geq 0$ we have $\Delta A^{(n)}_t \notin (0,\tfrac{1}{n})$. We have
\begin{equation}\label{eq:smoothenJumps1}
X_t \leq \int_0^t \eta(X_{s^-}) \dd A^{(n)}_s + M_t + H^{(n)}_t, \quad \text{where } H^{(n)}_t \define H_t + \int_0^t \eta(X_{s^-}) \dd A^{n,small}_s \qquad \forall t\geq 0.
\end{equation}
We apply Step 3 to $(X, A^{(n)}, H^{(n)}, M)$ and $(\Omega, \F, \PP, (\F_t)_{t\geq 0})$ so that we obtain a sequence $(\tilde{X}^{(n)}, \tilde{A}^{(n)}, \tilde{H}^{(n)}, \tilde{M}^{(n)})$, $n\in\N_0$ and $(\Omega, \F, \PP, (\tilde{\F}^{(n)}_t)_{t\geq 0})$.
We define 
\begin{equation*}
\tilde{\F}_t \define \bigcap_{n\in \N_0} \tilde{\F}_t^{(n)}.
\end{equation*}
Proving claims 4a-4d finishes Step 4: \\
\noindent \textbf{Claim 4a:} We have $\tilde{Y}_t = \lim_{n\to\infty} \tilde{Y}_t^{(n)}$ for $Y\in\{X, H, M\}$. In particular we have that $\tilde{X}$, $\tilde{H}$, $\tilde{M}$ are $(\tilde{\F}_t)_{t\geq 0}$-adapted. \\

We first show $\tilde{Y}_t = \lim_{n\to\infty} \tilde{Y}_t^{(n)}$ for $t\in \range(A)$: For $t\in\range(A)$ we have $\tilde{Y}_t = Y_{A^{-1}_t}$. By construction we have $A_r^{(n)} \leq A_r^{(n+1)} \leq A_r$ for all $n\in \N$, $r\geq 0$, which implies $(A^{(n)})^{-1}_r \geq (A^{(n+1)})^{-1}_r \geq A^{-1}_r$ for all $r\geq 0$. This implies $(A^{(n)})^{-1}_t$ is non-increasing in $n$. It can be verified that $(A^{(n)})^{-1}_t\searrow A^{-1}_t$ for $n\nearrow \infty$.
For $t$ such that $(A^{(n)})^{-1}_t=A^{-1}_t$ and $t\in\range(A)$ 
it holds that $t\in\range(A^{(n)})$. So we obtain $\lim_{n\to\infty} \tilde{Y}^{(n)}_t = \lim_{r\searrow A^{-1}_t} Y_r = \tilde{Y}_t$. 
 
Now we show $\tilde{Y}_t = \lim_{n\to\infty} \tilde{Y}_t^{(n)}$ for $t\notin \range(A)$: In this case $\exists s\geq 0, \, \varepsilon>0$ s.t. $A_{s^-} \leq t < t + \varepsilon < A_s$. This implies that
 $A^{(n)}_{s^-} \leq A_{s^-} \leq t < t + \varepsilon \leq  A^{(n)}_s \leq A_s$ for sufficiently large $n$, and hence $(A^{(n)})^{-1}_r  = s = A^{-1}_r$ and $r\notin \range(A^{(n)})$ for sufficiently large $n$.
This implies  $\tilde{Y}_t = \lim_{n\to\infty} \tilde{Y}^{(n)}$. 

Noting that $\tilde{\F}^{(n+1)}_t \subseteq \tilde{\F}^{(n)}_t$ for all $n\in \N$ and using  $\tilde{Y}_t = \lim_{n\to\infty} \tilde{Y}_t^{(n)}$ implies that $\tilde{X}$, $\tilde{H}$, $\tilde{M}$ are  $(\tilde{\F}_t)_{t\geq 0}$ adapted. \\

\noindent \textbf{Claim 4b:} $\tilde{M}$ is a local $(\tilde{\F}_t)_{t\geq 0}$ martingale. \\

Let $(\tau^M_n)_n$ be a localizing sequence of $M$, then define a new localizing sequence with respect to $(\tilde{\F}_t)_{t\geq 0}$ by
\begin{equation*}
\tilde{\tau}^M_n \define \inf\{ t\geq 0 \mid (M^{\tau^M_n})^{\sim}_t - \tilde{M}_t \neq 0\} \geq A_{\tau^M_n},
\end{equation*}
where $M^{\tau^M_n}$ denotes the stopped process. The inequality
$\tilde{\tau}^M_n   \geq  A_{\tau^M_n}$ follows from $M^{\tau^M_n}$ and $M$ coinciding upto time $\tau^M_n$. By the Debut theorem 
$\tilde{\tau}^M_n$ are indeed $(\tilde{\F}_t)_{t\geq 0}$ stopping times. Due to $A_\infty=\infty$ it is indeed a localizing sequence.

Hence, we may assume w.l.o.g. that $M$ is a martingale with $\E[|M_\infty|]<\infty$. Recall that by Step 3  $\tilde{M}^{(n)}$ is a $(\tilde{\F}^{(n)}_r)_{r\geq 0}$ (local) martingale. Fix some $0\leq s < t$. We want to prove
\begin{equation*}
\tilde{\F}_s \subseteq \{ A\in \F \mid \E[(\tilde{M}_{t} - \tilde{M}_s)\one_A] = 0\}.
\end{equation*}
Let $A\in \tilde{\F}_s =\cap_{n\in\N}\tilde{\F}^{(n)}_s$. Then we have
\begin{equation*}
\E[(\tilde{M}_{t} - \tilde{M}_s)\one_A] \overset{\textrm{Claim 4a}}{=} \E[\lim_{n\to\infty} (\tilde{M}^{(n)}_{t} - \tilde{M}^{(n)}_{s})\one_A] = \lim_{n\to\infty}\E[(\tilde{M}^{(n)}_t - \tilde{M}^{(n)}_{s})\one_A] = 0.
\end{equation*}
For the second equality we used $\tilde{M}^{(n)}_r = \E[M_\infty \mid \tilde{\F}^{(n)}_r]$ (since $\tilde{M}^{(n)}_\infty = M_\infty$), i.e. uniform integrability with respect to $n$. For the third equality we used that $A\in \tilde{\F}^{(n)}_s$ and that $\tilde{M}^{(n)}$ is a $(\tilde{\F}^{(n)}_r)_{r\geq 0}$ martingale. This implies that $\tilde{M}$ is a martingale.\\

\noindent \textbf{Claim 4c:} If $H$ is predictable then $\tilde{H}$ is predictable. \\

This follows from applying the monotone class theorem to \\
$\mathcal{H} = \{Y:[0,\infty)\times\Omega\to \R \mid \tilde{Y} \text{ is predictable and bounded}\}$,
noting that $\mathcal{M} = \{Y:[0,\infty)\times\Omega\to \R \mid Y \text{ continuous, adapted, bounded}\}$ is closed under multiplication and contained in $\mathcal{H}$. \\

\noindent \textbf{Claim 4d:} For any $s\geq 0$ we have that $A_s$ is $(\tilde{\F}_t)_{t\geq 0}$ stopping time. \\

This follows from
\begin{equation*}
\{A_s \leq t\}  = \bigcap_{n\in\N} \{A^{(n)}_s \leq t\} =
 \bigcap_{n=m}^\infty \{A^{(n)}_s \leq t\} \in \tilde{\F}^{(m)}_t \quad \forall t\geq 0, m\in \N.
\end{equation*}
\end{proof}

\subsection{Proof of the Snell corollary (\texorpdfstring{\cref{cor:characterizationOfLenglartDomination}}{Corollary 6.6})}
\begin{proof}
We first prove the assertion for the special case that $H_0\leq n_0$ for some $n_0 \in \N$.  We start by defining a suitable localizing sequence of bounded stopping times $(\sigma_{n})_{n\geq n_0}$ which ensures that $\{X_{\tau \wedge \sigma_{n}} - H_{\tau \wedge \sigma_{n}} \mid \tau\text{ finite stopping time} \}$ is bounded from below by $-n$ and a uniformly integrable family of random variables. To this end let $(\tau_n)_{n\geq n_0}$ be a localizing sequence such that $H_{\tau_n}\leq n$ for all $n\in \N, n\geq n_0$, which exists due to the predictability of $H$ and the assumption that $H_0 \leq n_0$. For all $n\geq n_0$ set
\begin{equation*}
\sigma_{n} \define \inf\{t\geq 0 \mid X_t \geq n\} \wedge \tau_n \wedge n.
\end{equation*}
The family $\{X_{\tau \wedge \sigma_{n}} - H_{\tau \wedge \sigma_{n}} \mid \tau\text{ finite stopping time} \}$ is uniformly integrable due to 
\begin{equation*}
\E[\sup_{t\geq 0} |X_{t\wedge \sigma_{n}} - H_{t\wedge \sigma_{n}}|] \leq 
\E[\sup_{t\geq 0} X_{t\wedge \sigma_{n}} + H_{\sigma_{n}}] 
\leq \E[X_{\sigma_{n}} +n + H_{\tau_n}]  \leq n + 2\E[H_{\tau_n}] < \infty,
\end{equation*}
where we used that $H\geq 0$, $X\geq 0$ for the first inequality and \eqref{eq:def-lenglart} for the third inequality.

Due to the choice of $\sigma_n$ we may apply the Snell envelope theorem \cite[Appendix 1: (22), p.416-417]{DellacherieMeyer} to $Y^{(m,n)}_t = X_{t\wedge \sigma_{n}} - H_{t\wedge \sigma_{n}} + m$, where $m\in\N$ is a new parameter. We define $N^{(m,n)}_t \define Z^{(m,n)}_t -m$ where $Z^{(m,n)}$ denotes the optional strong supermartingale given by Snell envelope theorem. Note that $N^{(m,n)}$ is by definition the minimal optional strong supermartingale such that 
\begin{equation*}
\max\{X_{t\wedge \sigma_{n}} - H_{t\wedge \sigma_{n}}, -m\} \leq N^{(m,n)}_t \quad  \forall t\geq 0.
\end{equation*}
Hence, due to $\max\{X_{t\wedge \sigma_{n}} - H_{t\wedge \sigma_{n}}, -m\} \leq N^{(m,n+1)}_{t\wedge \sigma_n}$
and $\max\{X_{t\wedge \sigma_{n}} - H_{t\wedge \sigma_{n}}, -(m+1)\} \leq N^{(m,n)}_{t}$, we have
\begin{equation}\label{eq:proof-characterization-Lenglart-dom-1}
N^{(m,n)}_t \leq N^{(m,n+1)}_{t\wedge \sigma_n} \,\, \text{ and } \,\, N^{(m+1,n)}_t \leq N^{(m,n)}_{t} \quad  \forall t\geq 0, \,\, n, m\in \N, n\geq n_0.
\end{equation}

Let $I$ denote the set of all $(\F_t)_{t\geq 0}$ stopping times. In the Snell envelope theorem the convention $Y^{(m,n)}_\infty \define 0$ is used and the essential supremum runs over all (not necessarily finite) stopping times $S$. However, due to $Y^{(m,n)} \geq 0$ for all $m\geq n \geq n_0$  we have $Z^{(m,n)}_T = \esssup_{S\in I, T \leq S< \infty} \E[Y^{(m,n)}_S \mid \F_T]$, and therefore $Z^{(m,n)}_0 \leq 0 + m$ i.e. $N^{(m,n)}_0 \leq 0$ for all $m\geq n \geq n_0$  by using \eqref{eq:def-lenglart}. Due to $N^{(m,n)}$ being an optional strong supermartingale and $ 0 \leq N^{(m,n)}_{t\wedge \tau_k} + H_{t\wedge \tau_k}$ for all $k, m, n$, we have
\begin{equation}\label{eq:proof-characterization-Lenglart-dom-2}
\E[|N^{(m,n)}_{t\wedge \tau_k}|] \leq 2 \E[(N^{(m,n)}_{t\wedge \tau_k})^-]
\leq 2\E[H_{\tau_k}].
\end{equation} 
We define $N^{(n)}_t \define \lim_{m\to\infty} N^{(m,n)}_t$ for all $t\geq 0$, which converges $\PP$-a.s. by \eqref{eq:proof-characterization-Lenglart-dom-1} and in $L^1$ to a supermartingale due to \eqref{eq:proof-characterization-Lenglart-dom-2}. Clearly $N^{(n)}_0\leq 0$ and $X_{t\wedge \sigma_n} \leq H_{t\wedge \sigma_n} + N^{(n)}_{t\wedge \sigma_n}$ for all $t$.

We define $\tilde{N}_t \define \lim_{n\to\infty} N^{(n)}_t$ for all $t\geq 0$ which converges pointwise due to \eqref{eq:proof-characterization-Lenglart-dom-1} and $\tilde{N}_{t\wedge \tau_k} = L^1\text{-}\lim_{n\to\infty} N^{(n)}_{t\wedge \tau_k}$ due to \eqref{eq:proof-characterization-Lenglart-dom-2}.
Hence, $\tilde{N}$ is a local supermartingale with localizing sequence $(\tau_k)_{k\in \N}$. Clearly $\tilde{N}_0 \leq 0$ and $X\leq H + \tilde{N}$. 
As $X$ and $H$ are right-continuous, also the right-limits process $(\tilde{N}_{t+})_{t\geq 0}$ satisfies $X_t \leq H_t + \tilde{N}_{t+}$.
Since the filtration satisfies the usual conditions, $(\tilde{N}_{t+})_{t\geq 0}$ is a local supermartingale with a \cadlagg modification.  Setting $N_t \define \tilde{N}_{t+}$ for all $t\geq 0$ and noting $N_0\one_{\{\tau_k>0\}} = \E[N_0 \one_{\{\tau_k>0\}} \mid \F_0] \leq \liminf_{h\searrow 0} \E[\tilde{N}_{\tau_k\wedge h} \one_{\{\tau_k>0\}} \mid \F_0] \leq 0$ proves the claim for the special case.

For the general case that $\E[H_0]<\infty$ we define $A_k \define  \{\omega \in \Omega \mid H_0(\omega) \in [k,k+1)\}$. As $(X_t\one_{A_k})_{t\geq 0}$ is dominated by $(H_t \one_{A_k})_{t\geq 0}$ and $H_0 \one_{A_k} \leq k+1$, we can apply the first part of the proof to obtain local \cadlagg supermartingales $(N^{(k)}_t\one_{A_k})_{t\geq 0}$. Due to $H$ being predictable and $H_0$ integrable, there exists a localizing sequence for $(\tau_n)_{n\geq 0}$ such that $\E[H_{\tau_n}]<\infty$ for all $n$. Define $N_t \define \sum_{k\in\N_0} N^{(k)}_t\one_{A_k}$.  Using $0 \leq N_t + H_t$ for all $t\geq 0$, $\E[H_{\tau_n}]<\infty$ implies that $\E[|N_{t\wedge\tau_n}|] <\infty$ for all $n$, i.e. $(N_t)_{t\geq 0}$ is indeed a local supermartingale.
\end{proof}

\subsection{Counterexample: Predictability of integrator \texorpdfstring{$A$}{A} necessary}
The following example is similar to the proof of \cref{thm:sharpnessAlpha} i.e. \cite[Theorem 2.1]{GeissScheutzow}.
\begin{cntexample}\label{example:notpredictable}
We provide an example that the assumption, that $A$ is predictable, cannot be dropped in
\cref{cor:GronwallSup} and \cref{cor:GronwallNoSup}. 
As we used in the proofs of these corollaries the predictability of $A$ solely when applying \cref{lemma:smoothenJumps}, it implies in particular, that  also  \cref{lemma:smoothenJumps} is false if the predictability assumption is dropped.

 More precisely, we provide an example of an adapted continuous process $(X_t)_{t\geq 0}$, a \cadlagg martingale $(M_t)_{t\geq 0}$ and an adapted \cadlagg non-decreasing process $(A_t)_{t\geq 0}$ which satisfy

\begin{equation*}
0\leq X_t \leq \int_0^t X_{s^-} \dd A_s + M_t + 1.
\end{equation*}
with the property that $\E[\sup_{t\geq 0} X_t^p] = +\infty$ and $\sup_{t\geq 0} \E[e^{qA_t}] < \infty$ for $p\in(0,1)$, $q>0$.

\bigskip

 Let $Z$ be an exponentially distributed random variable on a complete probability space $(\Omega, \F, \PP)$ with $\E[Z]=1$. Define for all $t\geq 0$, $p\in(0,1)$
\begin{equation*}
\begin{aligned}
X_t & \define \int_0^{t\wedge Z} p^{-1} \exp(s/p) \dd s + 1 = \exp((Z\wedge t)/p) \\
M_t & \define \int_0^{t\wedge Z} p^{-1} \exp(s/p) \dd s - p^{-1}\exp(Z\wedge t)/p)\one_{[Z,\infty)}(t) \\
H_t & \define 1 \\
A_t & \define p^{-1} \one_{[Z,\infty)}(t)
\end{aligned}
\end{equation*}
Choose $\tilde{\F}_t := \sigma(\{Z \leq r\} \mid  0 \leq r \leq t)$ for all $t\geq 0$ and denote by $(\F_t)_{t\geq 0}$ the smallest filtration satisfying the usual conditions, that contains $(\tilde{\F}_t)_{t\geq 0}$. $M$ is a martingale because $(\one_{[Z,\infty)}(t) - Z\wedge t)_{t\geq 0}$ is a martingale.  We have for all $t\geq 0$
\begin{equation*}
\begin{aligned}
X_t &= p^{-1}\exp(p^{-1}(Z\wedge t))\one_{[Z,\infty)}(t)+ M_t + 1 \\
& = \int_0^t X_{s^-} \dd A_s + M_t + 1.
\end{aligned}
\end{equation*}
For all $q>0$ and $t\geq 0$ we have:
\begin{equation*}
\begin{aligned}
\E[X^p_t] &= \E[\exp(Z\wedge t)] = \int_0^t\exp(z)\exp(-z) \dd z  + \exp(t)\exp(-t)= t + 1 \\
\E[\exp(q A_t)] & = \E[\exp(qp^{-1} \one_{[Z<\infty)}(t))] \leq \exp(qp^{-1}).
\end{aligned}
\end{equation*}
Noting that $\lim_{t\to\infty} \E[X^p_t] = +\infty$ and $\sup_{t\geq 0} \E[\exp(q A_t)] \leq \exp(qp^{-1})$ for any $q>0$  implies the assertion.
\end{cntexample}

\subsection{Counterexample: Structure of upper bounds for convex and concave $\eta$ differ}

We provide a counterexample which shows that bounds of the type \eqref{eq:notPossibleBound} are in general not true for convex $\eta$ under Assumption $\Anosup$.

\begin{cntexample}\label{counterexample}
Let $(W\ti)_{t\geq 0}$ be a Brownian motion on a suitable underlying filtered probability space and let $x_0>0$ and $\gamma >0$ be constants. Define:
\begin{equation*}
X\ti \define \e^{\gamma(x_0 + W\ti)^2} \qquad \forall t\geq 0.
\end{equation*}
An application of It\^o's formula implies that $X$ satisfies the following equation:
\begin{equation*}
\dd X\ti  = 2 \gamma (x_0 + W\ti) X\ti \dd W\ti + 
(\gamma + 2\gamma^2 (x_0+W\ti)^2)X\ti \dd t  = \eta(X\ti)\dd t + \dd M\ti
\end{equation*}
where $\dd M\ti \define 2 \gamma (x_0 + W\ti) X\ti \dd W\ti, t\geq 0$ is a local martingale starting in $0$, and $\eta(x) \define \gamma (1+2\log(x))x$  for $x\geq 1$. Note, that $X\ti \geq 1$. The function $\eta$ can be extended to a convex non-decreasing function on $[0,\infty)$  with $\eta(0)=0$. Hence, $X$ satisfies
\begin{equation*}
X\ti = \int_0^t \eta(X\s)\dd s + M_t + \e^{\gamma x_0^2}, \qquad \forall t\geq 0.
\end{equation*}
It can be shown that $\| X^*_T\|_p$ explodes at time $T=\frac{1}{2p\gamma}$, but we only prove here that it explodes at some finite time point.
For $p \gamma T \geq 1/2$ we have
\begin{equation*}
\begin{aligned}
\|X^*_T\|_p^p \geq \E[X^p_T] & = \int_{-\infty}^\infty 
\exp(p\gamma (x_0 + \sqrt{T} w)^2) \exp(-w^2/2) \dd w \\
&= \int_{-\infty}^\infty \exp( (p \gamma T - 1/2) w^2 +2 x_0 p\gamma  \sqrt{T} w + p \gamma x_0^2) \dd w = + \infty.
\end{aligned}
\end{equation*}

We apply \cref{thm:stochBihariConvex} to verify that the quantity $\| X^*_T\|_p$ is finite for small $T$. To this end we first compute $G$ and $G^{-1}$.  We have for all $x\geq 1$, $y\geq 0$:
\begin{equation*} 
\begin{aligned}
G(x) &\define \int_1^x \tfrac{\dd u}{\eta(u)} =  \tfrac{1}{2\gamma} \log( 
2\log(x)  +1),  \qquad G^{-1}(y) & = \e^{\e^{2\gamma y}/2 -1/2}.
\end{aligned}
\end{equation*}
\cref{thm:stochBihariConvex}  implies for all $q\in(0,1)$:
\begin{equation*}
\|G^{-1}(G(X_T^*) - T)\|_q =   \|  (X_T^*)^{\exp(-2\gamma t)}\|_q \exp(1/2 \e^{-2\gamma T} -1) \leq (1-q)^{-1} \e^{\gamma x_0^2}
\end{equation*}
yielding that $\| X^*_T\|_p$ is finite for $p< \exp(-2 \gamma T)$.

As $G$ and $G^{-1}$ are bounded on bounded subintervals of $[1,\infty)$, the explosion of the quantity $\|X^*_t\|_p$ at a finite time $t$ is a contradiction to $X$ having an upper bound of the type \eqref{eq:notPossibleBound}.
\end{cntexample}

\renewcommand{\abstractname}{Acknowledgements}
\begin{abstract}
The author wishes to thank Michael Scheutzow for his valuable suggestions and comments, in particular for his idea on how to prove the sharpness of the constant $\beta$ in \cref{thm:sharpnessBeta}. She would like to thank Mark Veraar for his idea to include the estimates \eqref{eq:Mark1} and \eqref{eq:Mark1-r} in  \cref{thm:stochBihariConvex} b) i.e. \eqref{eq:Mark2} and \eqref{eq:Mark2-r} in \cref{cor:GronwallNoSup}.
\end{abstract}


\begin{thebibliography}{10}

\bibitem{AgrestiVeraar}
Antonio Agresti and Mark Veraar.
\newblock The critical variational setting for stochastic evolution equations,
  2022.
\newblock arXiv:2206.00230.

\bibitem{AgrestiVeraar2}
Antonio Agresti and Mark Veraar.
\newblock Reaction-diffusion equations with transport noise and critical
  superlinear diffusion: Global well-posedness of weakly dissipative systems,
  2023.
\newblock arXiv:2301.06897.

\bibitem{CIR1}
Leif B.~G. Andersen and Vladimir~V. Piterbarg.
\newblock Moment explosions in stochastic volatility models.
\newblock {\em Finance Stoch.}, 11(1):29--50, 2007.

\bibitem{AppSup1}
Hussein~K. Asker.
\newblock Well-posedness and exponential estimates for the solutions to neutral
  stochastic functional differential equations with infinite delay.
\newblock {\em Journal of Systems Science and Information}, 8(5):434--446,
  2020.

\bibitem{AppSup4}
Stefan Bachmann.
\newblock Well-posedness and stability for a class of stochastic delay
  differential equations with singular drift.
\newblock {\em Stoch. Dyn.}, 18(2):1850019, 27, 2018.

\bibitem{AppSup3}
Stefan Bachmann.
\newblock On the strong {F}eller property for stochastic delay differential
  equations with singular drift.
\newblock {\em Stochastic Process. Appl.}, 130(8):4563--4592, 2020.

\bibitem{AppSup2}
Jianhai Bao, Feng-Yu Wang, and Chenggui Yuan.
\newblock Asymptotic log-{H}arnack inequality and applications for stochastic
  systems of infinite memory.
\newblock {\em Stochastic Process. Appl.}, 129(11):4576--4596, 2019.

\bibitem{Bihari}
Imre Bihari.
\newblock A generalization of a lemma of {B}ellman and its application to
  uniqueness problems of differential equations.
\newblock {\em Acta Math. Acad. Sci. Hungar.}, 7:81--94, 1956.

\bibitem{Burkholder}
Donald~L. Burkholder.
\newblock Distribution function inequalities for martingales.
\newblock {\em Ann. Probability}, 1:19--42, 1973.

\bibitem{CoxHutzenthalerJentzen}
Sonja Cox, Martin Hutzenthaler, and Arnulf Jentzen.
\newblock Local lipschitz continuity in the initial value and strong
  completeness for nonlinear stochastic differential equations, 2021.
\newblock arXiv:1309.5595.

\bibitem{CIR2}
Andrei Cozma and Christoph Reisinger.
\newblock Exponential integrability properties of {E}uler discretization
  schemes for the {C}ox--{I}ngersoll--{R}oss process.
\newblock {\em Discrete \& Continuous Dynamical Systems - B},
  21(10):3359--3377, 2016.

\bibitem{DellacherieMeyer}
Claude Dellacherie and Paul-Andr\'{e} Meyer.
\newblock {\em Probabilities and potential. {B}}, volume~72 of {\em
  North-Holland Mathematics Studies}.
\newblock North-Holland Publishing Co., Amsterdam, 1982.
\newblock Theory of martingales, Translated from the French by J. P. Wilson.

\bibitem{AppNoSup4}
Xiaoming Fu.
\newblock On invariant measures and the asymptotic behavior of a stochastic
  delayed {SIRS} epidemic model.
\newblock {\em Phys. A}, 523:1008--1023, 2019.

\bibitem{GeissConcave}
Sarah Geiss.
\newblock Concave and other generalizations of stochastic {G}ronwall
  inequalities, 2022.
\newblock arXiv:2204.06042v2.

\bibitem{GeissScheutzow}
Sarah Geiss and Michael Scheutzow.
\newblock Sharpness of {L}englart's domination inequality and a sharp monotone
  version.
\newblock {\em Electron. Commun. Probab.}, 26:Paper No. 44, 8, 2021.

\bibitem{GlattZiane}
Nathan Glatt-Holtz and Mohammed Ziane.
\newblock Strong pathwise solutions of the stochastic {N}avier-{S}tokes system.
\newblock {\em Adv. Differential Equations}, 14(5-6):567--600, 2009.

\bibitem{AppNoSup3}
Xing Huang and Feng-Yu Wang.
\newblock Distribution dependent {SDE}s with singular coefficients.
\newblock {\em Stochastic Process. Appl.}, 129(11):4747--4770, 2019.

\bibitem{AppNoSup8}
Xing Huang and Feng-Yu Wang.
\newblock Mc{K}ean-{V}lasov {SDE}s with drifts discontinuous under
  {W}asserstein distance.
\newblock {\em Discrete Contin. Dyn. Syst.}, 41(4):1667--1679, 2021.

\bibitem{Hudde}
Anselm Hudde, Martin Hutzenthaler, and Sara Mazzonetto.
\newblock A stochastic {G}ronwall inequality and applications to moments,
  strong completeness, strong local {L}ipschitz continuity, and perturbations.
\newblock {\em Ann. Inst. Henri Poincar\'{e} Probab. Stat.}, 57(2):603--626,
  2021.

\bibitem{HutzenthalerNguyen}
Martin Hutzenthaler and Tuan~Anh Nguyen.
\newblock A path-dependent stochastic {G}ronwall inequality and strong
  convergence rate for stochastic functional differential equations, 2022.
\newblock arXiv:2206.01049.

\bibitem{AppNoSup10}
Theresa Lange and Wilhelm Stannat.
\newblock Mean field limit of ensemble square root filters - discrete and
  continuous time.
\newblock {\em Foundations of Data Science}, 3(3):563--588, 2021.

\bibitem{LaSalle}
Joseph~Pierre LaSalle.
\newblock Uniqueness theorems and successive approximations.
\newblock {\em Ann. of Math. (2)}, 50:722--730, 1949.

\bibitem{LLP}
E.~Lenglart, D.~L\'{e}pingle, and M.~Pratelli.
\newblock Pr\'{e}sentation unifi\'{e}e de certaines in\'{e}galit\'{e}s de la
  th\'{e}orie des martingales.
\newblock In {\em Seminar on {P}robability, {XIV} ({P}aris, 1978/1979)
  ({F}rench)}, volume 784 of {\em Lecture Notes in Math.}, pages 26--52.
  Springer, Berlin, 1980.
\newblock With an appendix by Lenglart.

\bibitem{Lenglart}
\'{E}rik Lenglart.
\newblock Relation de domination entre deux processus.
\newblock {\em Ann. Inst. H. Poincar\'{e} Sect. B (N.S.)}, 13:171--179, 1977.

\bibitem{AppNoSup11}
Chengcheng Ling and Longjie Xie.
\newblock Strong solutions of stochastic differential equations with
  coefficients in mixed-norm spaces.
\newblock {\em Potential Analysis}, pages 1--15, 03 2021.

\bibitem{LeLing}
Khoa Lê and Chengcheng Ling.
\newblock Taming singular stochastic differential equations: A numerical
  method, 2021.
\newblock arXiv:2110.01343.

\bibitem{Makasu2}
Cloud Makasu.
\newblock A stochastic {G}ronwall lemma revisited.
\newblock {\em Infin. Dimens. Anal. Quantum Probab. Relat. Top.},
  22(1):1950007, 5, 2019.

\bibitem{Makasu1}
Cloud Makasu.
\newblock Extension of a stochastic {G}ronwall lemma.
\newblock {\em Bull. Pol. Acad. Sci. Math.}, 68(1):97--104, 2020.

\bibitem{AppNoSup1}
Oliver Matte.
\newblock Continuity properties of the semi-group and its integral kernel in
  non-relativistic {QED}.
\newblock {\em Rev. Math. Phys.}, 28(5):1650011, 90, 2016.

\bibitem{MehriScheutzow}
Sima Mehri and Michael Scheutzow.
\newblock A stochastic {G}ronwall lemma and well-posedness of path-dependent
  {SDE}s driven by martingale noise.
\newblock {\em ALEA, Lat. Am. J. Probab. Math. Stat.}, 18(4198874):193--209,
  2021.

\bibitem{Nieto}
Slimane Mekki, Juan~J. Nieto, and Abdelghani Ouahab.
\newblock Stochastic version of {H}enry type {G}ronwall's inequality.
\newblock {\em Infin. Dimens. Anal. Quantum Probab. Relat. Top.}, 24(2):Paper
  No. 2150013, 10, 2021.

\bibitem{Osekowski}
Adam Os\c{e}kowski.
\newblock Sharp maximal inequalities for the martingale square bracket.
\newblock {\em Stochastics}, 82:589--605, 2010.

\bibitem{Pratelli}
Maurizio Pratelli.
\newblock Sur certains espaces de martingales localement de carr\'{e}
  int\'{e}grable.
\newblock In P.~Meyer, editor, {\em S\'{e}minaire de {P}robabilit\'{e}s, {X}
  ({S}econde partie: {T}h\'{e}orie des int\'{e}grales stochastiques, {U}niv.
  {S}trasbourg, {S}trasbourg, ann\'{e}e universitaire 1974/1975)}, pages
  401--413. Lecture Notes in Math., Vol. 511. Springer, 1976.

\bibitem{RenShen}
Yaofeng Ren and Jing Shen.
\newblock A note on the domination inequalities and their applications.
\newblock {\em Statist. Probab. Lett.}, 82:1160--1168, 2012.

\bibitem{AppNoSup7}
Michael R\"{o}ckner, Longjie Xie, and Xicheng Zhang.
\newblock Superposition principle for non-local {F}okker-{P}lanck-{K}olmogorov
  operators.
\newblock {\em Probab. Theory Related Fields}, 178(3-4):699--733, 2020.

\bibitem{AppNoSup5}
Michael R\"{o}ckner and Xicheng Zhang.
\newblock Well-posedness of distribution dependent {SDE}s with singular drifts.
\newblock {\em Bernoulli}, 27(2):1131--1158, 2021.

\bibitem{Scheutzow}
Michael Scheutzow.
\newblock A stochastic {G}ronwall lemma.
\newblock {\em Infin. Dimens. Anal. Quantum Probab. Relat. Top.},
  16(2):1350019, 4, 2013.

\bibitem{AppSup5}
Michael Scheutzow and Susanne Schulze.
\newblock Strong completeness and semi-flows for stochastic differential
  equations with monotone drift.
\newblock {\em J. Math. Anal. Appl.}, 446(2):1555--1570, 2017.

\bibitem{AppNoSup2}
Renming Song and Longjie Xie.
\newblock Well-posedness and long time behavior of singular {L}angevin
  stochastic differential equations.
\newblock {\em Stochastic Process. Appl.}, 130(4):1879--1896, 2020.

\bibitem{RenesseScheutzow}
Max-K. von Renesse and Michael Scheutzow.
\newblock Existence and uniqueness of solutions of stochastic functional
  differential equations.
\newblock {\em Random Oper. Stoch. Equ.}, 18(3):267--284, 2010.

\bibitem{AppNoSup6}
Pengcheng Xia, Longjie Xie, Xicheng Zhang, and Guohuan Zhao.
\newblock {$L^q(L^p)$}-theory of stochastic differential equations.
\newblock {\em Stochastic Process. Appl.}, 130(8):5188--5211, 2020.

\bibitem{XieZhang}
Longjie Xie and Xicheng Zhang.
\newblock Ergodicity of stochastic differential equations with jumps and
  singular coefficients.
\newblock {\em Ann. Inst. Henri Poincar\'{e} Probab. Stat.}, 56(1):175--229,
  2020.

\bibitem{AppNoSup9}
Xicheng Zhang and Guohuan Zhao.
\newblock Singular {B}rownian diffusion processes.
\newblock {\em Commun. Math. Stat.}, 6(4):533--581, 2018.

\end{thebibliography}

\end{document}